\documentclass{amsart}

\usepackage[utf8]{inputenc}

\usepackage{amssymb}
\usepackage{amsmath}
\usepackage{amsthm}
\usepackage{enumitem}
\usepackage{pgf,tikz}
\usetikzlibrary{arrows}

\numberwithin{equation}{section}

\theoremstyle{plain}
\newtheorem{thm}{Theorem}[section]
\newtheorem{lemma}[thm]{Lemma}
\newtheorem{prop}[thm]{Proposition}
\newtheorem{coroll}[thm]{Corollary}
\newtheorem{claim}[thm]{Claim}

\newtheorem{conj}[thm]{Conjecture}

\theoremstyle{definition}
\newtheorem{definition}[thm]{Definition}

\theoremstyle{remark}
\newtheorem{remark}[thm]{Remark}
\newtheorem{ex}[thm]{Example}

\usepackage{xspace} 

\newcommand{\R}{\mathbf{R}}
\newcommand{\Z}{\mathbf{Z}}
\newcommand{\HH}{\mathcal H}

\DeclareMathOperator{\bip}{Bip}
\DeclareMathOperator{\conv}{Conv}

\title{Hypergraph polynomials and the Bernardi process}

\author{Tam\'as K\'alm\'an}
\address{Department of Mathematics\\
	Tokyo Institute of Technology\\
	H-214, 2-12-1 Ookayama, Meguro-ku, Tokyo 152-8551, Japan}
\email{kalman@math.titech.ac.jp}
\thanks{TK was supported by consecutive Japan Society for the Promotion of Science (JSPS) Grant-in-Aids, for Young Scientists B (no.\ 25800037) and for Scientific Research C (no.\ 17K05244).}

\author{Lilla T\'othm\'er\'esz}
\address{Cornell University,
	Ithaca, New York 14853-4201, USA, 
	MTA-ELTE Egerv\'ary Research Group, P\'azm\'any P\'eter s\'et\'any 1/C, Budapest, Hungary}
\email{tmlilla@cs.elte.hu}
\thanks{LT was supported by the National Research, Development and Innovation Office -- NKFIH, grant no.\ 109240, and by the NSF grant DMS-14552.}

\keywords{hypergraph, bipartite graph, ribbon structure, Tutte polynomial, interior polynomial, embedding activity, root polytope, dissection, shelling order, $h$-vector}

\date{\today}

\subjclass[2010]{05C10, 05C31, 05C50, 05C57, 05C65}

\begin{document}  

\begin{abstract}
Bernardi gave a formula for the Tutte polynomial $T(x,y)$ of a graph, based on spanning trees and activities just like the original definition, but using a fixed ribbon structure to order the set of edges in a different way for each tree. The interior polynomial $I$ is a generalization of $T(x,1)$ to hypergraphs. We supply a Bernardi-type description of $I$ using a ribbon structure on the underlying bipartite graph $G$. Our formula works because it is determined by the Ehrhart polynomial of the root polytope of $G$ in the same way as $I$ is. To prove this we interpret the Bernardi process as a way of dissecting the root polytope into simplices, along with a shelling order. We also show that our generalized Bernardi process gives a common extension of bijections (and their inverses), constructed by Bernardi and further studied by Baker and Wang, between spanning trees and break divisors.
\end{abstract}

\maketitle

\section{Introduction}

A few years ago a pair of polynomial invariants of hypergraphs was introduced \cite{hiperTutte}, which generalize the valuations $T(x,1)$ and $T(1,y)$ of the two-variable graph invariant $T(x,y)$ due to Tutte \cite{Tutte}. They are called the interior and exterior polynomials because they are generating functions of `interior and exterior activity.' In the case of graphs, activities were associated to spanning trees by Tutte himself in his original definiton of $T$. In the hypergraph case, instead of spanning trees one considers `hypertrees' and their activities, in a spirit very close to \cite{Tutte}. Hypertrees were introduced in \cite{alex} (and so named in \cite{hiperTutte}). They generalize characteristic vectors of spanning trees of a graph, preserving some nice polyhedral properties.

Both for graphs and hypergraphs, the computation of individual activities requires fixing an order of the set of edges or hyperedges, respectively, albeit temporarily, because the aggregate polynomials do not depend on it. See Definition \ref{def:activity}. In his remarkable paper \cite{Bernardi_Tutte}, O. Bernardi removed the fixed order from the definition (in the case of graphs) and replaced it with another kind of auxiliary data: a ribbon structure and a base point. Loosely speaking, for a given spanning tree, he traced the boundary of the neighborhood of the tree and numbered the edges of the graph along the way. He used this order to compute the internal and external activities of the tree. He repeated this for all spanning trees and organized the information in a two-variable generating function which happens to satisfy the same deletion-contraction formulas as the Tutte polynomial --- hence the two agree, regardless of what ribbon structure we use. 

In this paper we extend Bernardi's work to the case of the interior polynomial of a hypergraph. A similar formula is conjectured for the exterior polynomial.

Any hypergraph $\HH = (V,E)$ naturally yields a bipartite graph $\bip\HH$ in which one color class corresponds to the vertices of the hypergraph, the other color class to the hyperedges, and edges correspond to containment. We assume $\bip \HH$ to be connected and endow it with a ribbon structure and a base point. (The base point can be thought of as a boundary point of the thickened graph.)

A hypertree is essentially a `possible degree distribution vector' of a spanning tree of $\bip\HH$ taken at the elements of $E$, cf.\ Definition \ref{def:hypertree}. We note that hypertrees of ordinary graphs are exactly the characteristic functions of their spanning trees.

Our first order of business is to define what it means to `trace the boundary of the neighborhood of a hypertree,' which turns out to be a process constructing a certain spanning tree that realizes the hypertree. In fact we define two versions of such a `Bernardi process.' Contrary to the case of graphs, the fact that the Bernardi process results in a tree is not trivial at all. As a byproduct, we also obtain an order on the set of edges of $\bip\HH$ which we then use to order $E$ as well. Now it makes sense to take the interior and exterior activities, just like in \cite{hiperTutte}, of the hypertree with respect to this order. After repeating the procedure for all hypertrees, we write two one-variable generating functions $\tilde I$ and $\tilde X$ for the two `embedding activities.' (As to why not a single, two-variable function, see \cite[Remark 5.7]{hiperTutte}, cf.\ \cite{Cameron-Fink}. There is in fact some new development on this issue \cite{universal}, which we hope to incorporate in future work.)

The main result of the paper is that {\bf the generating function $\tilde I$ of internal embedding activities coincides with the interior polynomial}. The interior and exterior polynomials of a hypergraph, $I$ and $X$, look similar to each other but their behavior is rather different. For instance, the former is invariant under taking the transpose of the hypergraph but the latter is not. (Here the transpose of the hypergraph $\HH=(V,E)$ is the hypergraph $\overline\HH=(E,V)$ with the roles of vertices and hyperedges interchanged.) In other words, $I$ is an invariant of the bipartite graph $\bip \HH$. This fact is proven in \cite{KP_Ehrhart} by noting that (essentially) the same polynomial may be obtained as the Ehrhart polynomial of the so called root polytope of $\bip \HH$. This depends, among other things, on the basic observation that spanning trees of a bipartite graph correspond to maximal simplices in its root polytope. We exploit the same connection to prove our main theorem. Since we do not have an analogous description for the exterior polynomial, the exterior version of our result remains, for the time being, a conjecture.

The notion of root polytope (Definition \ref{def:rootpolytope}) is due to Hibi and Ohsugi \cite{hibi}. Postnikov \cite{alex} discovered it independently and studied its triangulations to great effect. A consequence of our proof is an unexpected link between Bernardi's and Postnikov's work: {\bf when the Bernardi trees for all hypertrees are translated to simplices in the root polytope, they form a dissection}. (I.e., the simplices fill the polytope and their interiors are mutually disjoint. They typically do not form a triangulation though, cf.\ Examples \ref{ex:non-triangulation} and \ref{ex:otszog} --- that is, some pairs of simplices may not intersect in a common face.) In Section \ref{sec:examples} we will see that this generalizes the well-known triangulation of the product of two simplices by non-crossing trees.

We get an alternative description of the dissection by reinterpreting Bernardi trees as `Jaeger trees,' in honor of F. Jaeger's beautiful paper \cite{Jaeger} in which they appear as the main terms in a certain expansion of the Homfly polynomial. (The overlap between Jaeger's cases and ours is when $\bip\HH$ is embedded in the plane so that the so called median construction can be performed, resulting in a (typically non-planar) ribbon structure for the graph, as well as in an alternating link. See Figures \ref {fig:apapirrasokatir} and \ref{fig:knot}.) This description does not refer to hypertrees, instead Jaeger trees are defined by a simple local rule that is obeyed when we trace the boundaries of their neighborhoods. This also leads to the definition of a natural order among Bernardi/Jaeger trees and we prove, 
as a key step to our main theorem, that {\bf this order is a shelling order of the dissection}. 

If $\mathbf f\colon E\to\Z_{\ge0}$ is a hypertree in $\HH$, we may view the `other' degree vector of its Bernardi tree as a hypertree $\bar{\mathbf f}\colon V\to\Z_{\ge0}$ in $\overline\HH$. The dissection property, just like in \cite{alex}, implies that this is a bijection between the two sets of hypertrees. In the case of graphs, hypertrees on the vertex set are easily seen to be the duals of the so called break divisors. In this special case, the above bijection-by-dissection agrees with the \emph{bijection between spanning trees and break divisors}, defined by Bernardi \cite{Bernardi_Tutte} and studied further by Baker and Wang \cite{Baker-Wang}. In \cite{Baker-Wang} the inverse of Bernardi's bijection is described in a way that is formally different from the original Bernardi algorithm. In hypergraph language (where the transpose is an obvious involution and transpose hypergraphs share the same bipartite graph and root polytope), {\bf the bijection and its inverse are revealed to be of the same nature, defined by the same dissection}.

Finally, we note that our family of dissections of the root polytope (depending on ribbon structure and base point) contains several previously known triangulations. Namely, in addition to the triangulation by non-crossing trees \cite{GGP} (which applies to the root polytope of a complete bipartite graph), the triangulation by duals of arborescences \cite{KM} (which works in the case of a plane bipartite graph) is also a special case.

We should warn the reader that there are many orders in this paper, of edges, nodes, hyperedges etc.\ induced by spanning trees, hypertrees etc. This can be cumbersome but we need each for its own technical reason. All of these orders, however, are defined by the same simple principle: some process propagates through the graph and objects are listed in the order in which they are first reached. 

The paper is organized as follows. 
In Section \ref{sec:prelim} we set our definitions and summarize some of the necessary background, including Bernardi's embedding activities. 
Section \ref{sec:root_polytope} surveys Postnikov's work and describes the link from $h$-vectors of shellable dissections of the root polytope of $\bip \HH$ to the interior polynomial of $\HH$. 
We define the Bernardi process for hypergraphs in Section \ref{sec:Bernardi-process}, and establish its well-definedness. 
In Section \ref{sec:statements} we give the Bernardi-type description of the interior polynomial and state the equivalence of the two definitions (Theorem \ref{thm:Bernardi-interior_well_def}, our main result), as well as several conjectures. 
In Section \ref{sec:Jaeger_trees} we define Jaeger trees, prove their basic properties and show that the set of spanning trees arising as outcomes of the Bernardi process is exactly the set of Jaeger trees. We also discuss the connection of our work to that of Baker and Wang \cite{Baker-Wang}. 
In Section \ref{sec:shelling} we show that Jaeger trees induce a shellable dissection of the root polytope of $\bip \HH$ with a natural shelling order, and we relate the resulting $h$-vector to Bernardi-type activities. This allows us to prove Theorem \ref{thm:Bernardi-interior_well_def}.
In Section \ref{sec:examples} we show how certain previously known triangulations arise from the Bernardi process.
Finally, in an appendix we observe that for graphs (i.e., hypergraphs $\HH$ where the cardinality of each hyperedge is two), the Bernardi process behaves in a special way in that it induces an activity-preserving bijection between the hypertrees of $\HH$ and $\overline{\HH}$. 

{\bf Acknowledgment.} This version of our paper is identical to the one published in Algebraic Combinatorics, vol.\ 3 (2020), no.\ 5, pp.\ 1099--1139. We thank the anonymous referee for an exceptionally close reading of the manuscript and a number of suggestions for improvement.

\section{Preliminaries} \label{sec:prelim}

\subsection{Basics}

A \emph{hypergraph} is an ordered pair $\HH=(V,E)$, where $V$ is a finite set and $E$ is a finite multiset of subsets of $V$. We refer to elements of $V$ as \emph{vertices} and to elements of $E$ as \emph{hyperedges}. For a hypergraph $\mathcal{H}=(V,E)$, let the \emph{underlying bipartite graph} $\bip\HH$ be the bipartite graph with vertex classes $V$ and $E$, where $v\in V$ is connected to $e\in E$ if $v\in e$ in $\HH$. In the context of bipartite graphs such as $\bip\HH$, instead of vertices, we will call the elements of $V\cup E$ \emph{nodes}. Specifically, the elements of $V$ and $E$ will be called \emph{violet} and \emph{emerald} nodes, respectively. (In our figures, violet appears as blue and emerald appears as green.) Throughout the paper, we assume that $\HH$ is \emph{connected}, which means that $\bip \HH$ is connected.

We also assume that $\bip\HH$ has a ribbon graph structure. Here for a graph $G$ without loop edges, a \emph{ribbon structure} is a family of cyclic permutations: namely for each vertex $x$ of $G$, a cyclic permutation of the edges incident to $x$ is given. For an edge $xy$ of $G$, we use the following notations:
\begin{itemize}
\item $yx^+_G$: the edge following $yx$ at $y$
\item $yx^-_G$: the edge preceding $yx$ at $y$
\item $xy^+_G$: the edge following $xy$ at $x$
\item $xy^-_G$: the edge preceding $xy$ at $x$.
\end{itemize}
If the graph $G$ is clear from the context, we omit the subscript. We will sometimes need the operation of removing an edge from a ribbon graph. If $G$ is a ribbon graph, and $\varepsilon$ is an edge of $G$, then $G-\varepsilon$ is the ribbon graph with
\[
xy^+_{G-\varepsilon}= \left\{\begin{array}{cl} xy_G^+ & \text{if } xy_G^+\neq \varepsilon \\
(xy_G^+)_G^+ & \text{if } xy_G^+=\varepsilon
\end{array} \right.
\]
for any edge $xy$ of $G-\varepsilon$. More generally, any subgraph of $G$ inherits a ribbon structure from $G$ in the obvious way, by restrictions of the cyclic orders.

Throughout the paper, when we consider ribbon structures, we will assume that there is a fixed vertex $b_0$ of $G$ that we call the \emph{base vertex} (or \emph{base node}, in cases when $G$ is bipartite), and a fixed edge $b_0b_1$ incident to $b_0$ that we call the \emph{base edge}. 

\begin{remark}
\label{rem:surface}
Ribbon structures may be equivalently described by \emph{ribbon surfaces}, as follows. See Figure \ref{fig:apapirrasokatir} for an example. We consider the graph as a topological space, thicken a small neighborhood of each vertex to a disk, and orient it (hence also orient its boundary) so that the edges incident to the vertex intersect the boundary of the disk in their prescribed cyclic order. (When the vertex has degree three or more, this orientation is uniquely determined once the disk has been constructed.) Then we thicken each edge into a rectangle, attached along two opposite sides to the appropriate disks, so that the orientations extend over the rectangle. Thinking of the two attaching sides of the rectangle as `short' and the other two, running along the edge, as `long,' explains the name of the structure. Conversely, if a graph is embedded in an oriented surface, the orientation induces a ribbon structure on it which is equivalent to taking a regular neighborhood of the embedding. Finally, one may equivalently specify the base vertex and the base edge by placing a \emph{base point} on the boundary of the disk centered at $b_0$, along the bit running from $b_0b_1^-$ to $b_0b_1$.
\end{remark}

Let $G$ be a graph, $T$ be a spanning tree of $G$, and $\varepsilon\in T$ be an edge. As $T$ is a spanning tree, $T-\varepsilon$ is a graph with two connected components. We call the set of edges of $G$ connecting two vertices from different components of $T-\varepsilon$ the \emph{fundamental cut} of $\varepsilon$ in $T$, and denote it by $C^*(T,\varepsilon)$. Now for an edge $\varepsilon$ of $G$ that is not part of $T$, adding $\varepsilon$ to $T$ creates a unique cycle, which we call the \emph{fundamental cycle} of $\varepsilon$ with respect to $T$ and denote with $C(T,\varepsilon)$.

\subsection{The interior polynomial}
\label{ssec:interior}

The following definition plays a central role in our paper. 

\begin{definition}
\label{def:hypertree}
Let $G$ be a bipartite graph and $E$ one of its vertex classes. We say that the vector $\mathbf f\colon E\to\Z_{\ge0}$ is a  \emph{hypertree} on $E$ if there exists a spanning tree $T$ of $G$ that has degree $d_T(e)=\mathbf f(e)+1$ at each node $e\in E$. We denote the set of all hypertrees on $E$ by $B_{E}$.
\end{definition}

Disconnected bipartite graphs have no spanning trees and thus no hypertrees, either. It is slightly more natural to call the objects above hypertrees in the hypergraph $\HH=(V,E)$, as opposed to in $G=\bip\HH$, and to denote their set with $B_\HH$ instead of $B_E$. In that sense, hypertrees generalize (characteristic vectors of) spanning trees from graphs to hypergraphs (cf.\ \cite[Remark 3.2]{hiperTutte}). We will often adopt this point of view, even though the wording of Definition \ref{def:hypertree} suits our current purposes better.

The set $B_E$ is such that $(\conv B_E)\cap\Z^E=B_{E}$, where $\conv$ denotes the usual convex hull in $\R^E$, cf.\ \cite[Lemma 3.4]{hiperTutte}. We will call $\conv B_{E}=\conv B_\HH$ the \emph{hypertree polytope} of $\HH$. In the special case of (characteristic vectors of) spanning trees, hypertrees are exactly the vertices of $\conv B_E$, which in that case is known as the spanning tree polytope.

We note that for all hypertrees $\mathbf f$ on $E$, it holds that 
\[\sum_{e\in E}\mathbf f(e)=|V|-1\]
is independent of $\mathbf f$ \cite[Theorem 3.4]{hiperTutte}. Hence $B_E$ and $\conv B_E$ lie along an affine hyperplane of $\R^E$.

For a spanning tree $T$ of $\bip \HH$, let $\mathbf f_E(T)$ be the hypertree on $E$ \emph{realized} or \emph{induced} by $T$, i.e.,
\[\mathbf f_E(T)(e)=d_T(e)-1  \quad \text{for all }e\in E.\]
Similarly, let $\mathbf f_V(T)$ be the hypertree on $V$ realized by $T$.

The definition of the interior polynomial is based on hypertrees and a natural generalization of internal activity used in the case of graphs (and matroids). First, if the hypertree $\mathbf{f}\in B_{E}$ and the hyperedges $e, e'\in E$ are such that changing the value $\mathbf{f}(e)$ to $\mathbf{f}(e)-1$ and the value $\mathbf{f}(e')$ to $\mathbf{f}(e')+1$ results in another hypertree $\mathbf{f'}$, then let us say that $\mathbf{f}$ and $\mathbf{f'}$ are related by a \emph{transfer of valence} from $e$ to $e'$. Another expression we will use is that $\mathbf f$ is such that $e$ \emph{can transfer valence} to $e'$.

\begin{definition}
\label{def:activity}
Let $(V,E)$ be a hypergraph with an order on the set $E$. A hyperedge $e\in E$ is \emph{internally active} for the hypertree $\mathbf{f}$, with respect to the order, if $\mathbf f$ is such that $e$ cannot transfer valence to any smaller hyperedge. Let $\iota(\mathbf{f})$ denote the number of internally active hyperedges with respect to $\mathbf{f}$ and call this value the \emph{internal activity} of $\mathbf f$. 
\end{definition}

We call a hyperedge \emph{internally inactive} for a hypertree if it is not internally active and denote the number of such hyperedges (for a given $\mathbf{f}$) by $\bar\iota(\mathbf{f})=|E|-\iota(\mathbf{f})$. This value will be called the \emph{internal inactivity} of $\mathbf{f}$. Note that $\iota(\mathbf{f})$ and $\bar\iota(\mathbf{f})$ both depend on the order used on $E$.

\begin{definition}
\label{def:I}
Let $\HH=(V,E)$ be a hypergraph so that $\bip\HH$ is connected. For some fixed linear order on $E$ we consider the generating function of internal inactivity, $I_\HH(\xi)=\sum_{\mathbf{f}\in B_{E}} \xi^{\bar\iota(\mathbf{f})}$, and call it the \emph{interior polynomial} of $\HH$. By \cite[Theorem 5.4]{hiperTutte} (see also \cite[Subsection 2.2]{KP_Ehrhart}), $I_\HH$ does not depend on the order.
\end{definition}

\begin{ex}
Any connected bipartite graph serves as the underlying bipartite graph for two hypergraphs, a transpose pair. For the graph that appears in examples throughout the paper (first in Figure \ref{fig:apapirrasokatir}), both of these hypergraphs have the interior polynomial $1+3\xi+3\xi^2$, as computed in \cite[Example 5.6]{hiperTutte}. In particular, the number of hypertrees on either vertex class is $7$.
\end{ex}

Note that if we specialize our notion of internal activity to graphs, then external edges of a spanning tree become internally active. This is not the case for the definition used by Tutte and Bernardi, which we review in the next subsection. However, the number of internally inactive edges is the same as the number of `internal, not internally active' edges in the original definition. This subtlety can hardly be avoided because for a hyperedge, there is no natural notion of being inside or outside of a hypertree. Even if we defined `$e$ being external to $\mathbf f$' by $\mathbf f(e)=0$, the number of such hyperedges \emph{would} depend on $\mathbf f$.

Also, as opposed to the definition of the Tutte polynomial, in Definition \ref{def:I} we count inactive hyperedges instead of active ones. But since the number of external edges \emph{is} the same for all spanning trees (namely, the first Betti number $\beta_1=|E|-|V|+1$ of the graph), all these tweaks in the definition just mirror and shift the distribution of the `classical' internal activity statistic. The precise claim is that if the hypergraph $\HH$ happens to be a graph with Tutte polynomial $T(x,y)$, then its interior polynomial is 
\[I_\HH(\xi)=\xi^{|V|-1}T(1/\xi,1).\]

We will almost always work in the larger context of hypergraphs and use the notions of Definitions \ref{def:activity} and \ref{def:I}. Because the difference is so minimal, we will not introduce separate terminology, only separate notation for internal activity in the `usual' sense: for a spanning tree $T$ of the graph $G$ with vertex set $V$, and an ordering of the set $E$ of edges, we let $i(T)=\iota(T)-\beta_1(G)=|V|-1-\bar\iota(T)$.

Hypergraphs also have an exterior polynomial invariant. For this paper it is less important but we will indicate its definition and various properties in Section \ref{sec:statements}.

\subsection{The Bernardi process for graphs}
\label{subsec:Bernardi_for_graphs}

In this subsection we recall the Bernardi process for ordinary graphs, as well as the Bernardi-type definition of the Tutte polynomial \cite{Bernardi_first,Bernardi_Tutte}.

The basic notion in \cite{Bernardi_first} is the \emph{tour of a spanning tree}. Let $G$ be a ribbon graph, and $T$ be a spanning tree of $G$. We specify a base vertex $b_0$ of $G$ and a base edge $b_0b_1$ incident to $b_0$.

\begin{figure} 
	\begin{center}
		\begin{tikzpicture}[-,>=stealth',auto,scale=0.6,
		thick]
		\tikzstyle{o}=[circle,draw]
		\node[o] (1) at (8, 0) {{\small $v_1$}};
		\node[o] (2) at (4, 1.5) {{\small $v_2$}};
		\node[o] (3) at (0, 0) {{\small $v_3$}};
		\node[o] (4) at (4, -1.5) {{\small $v_4$}};
		\path[every node/.style={font=\sffamily\small}, line width=0.8mm]
		(1) edge node [above] {$e_1$} (2)
		(3) edge node [below] {$e_3$} (4)
		(2) edge node {$e_5$} (4);
		\path[every node/.style={font=\sffamily\small},dashed]
		(4) edge node [below] {$e_4$} (1)
		(2) edge node [above] {$e_2$} (3);
		\end{tikzpicture}
	\end{center}
	\caption{An example of the tour of a spanning tree. Let the ribbon structure be the one induced by the positive orientation of the plane. The edges of the tree are drawn by thick lines, the non-edges by dashed lines. With $b_0=v_1, b_1=v_2$, we get the tour $v_1,e_1$; $v_2,e_2$; $v_2,e_5$; $v_4,e_3$; $v_3,e_2$; $v_3,e_3$; $v_4,e_4$; $v_4,e_5$; $v_2,e_1$; $v_1,e_4$.}
\label{fig:tour_of_a_tree}
\end{figure} 

The tour of $T$ is the following sequence of vertex-edge pairs: The current vertex at the first step is $b_0$, and the current edge is $b_0b_1$. If the current vertex is $x$, the current edge is $xy$, and $xy\notin T$, then the current vertex of the next step is $x$, and the current edge of the next step is $xy^+$. If the current vertex is $x$, the current edge is $xy$, and $xy\in T$, then the current vertex of the next step is $y$, and the current edge of the next step is $yx^+$. In the first case we say that the tour \emph{skips} $xy$ and in the second case we say that the tour \emph{traverses} it. The tour stops right before when $b_0$ would once again become current vertex with $b_0b_1$ as current edge. See Figure \ref{fig:tour_of_a_tree} for an example.

\begin{remark}
	\label{rem:topological_interpretation_of_tour_of_a_tree}
	The sequence of current vertex-edge pairs of the tour of a spanning tree can also be described using the topology of the ribbon surface of Remark \ref{rem:surface}. Namely, we start from the base point (meant as a boundary point of the surface) and proceed along the boundary in the positive direction. When we reach a vertex of one of the rectangle pieces, then
	\begin{itemize}
		\item the midpoint of the disk, together with the core of the rectangle, become current, and
		\item if the edge is in the tree, we proceed along the long side of the rectangle to the next adjacent disk (traversal), and continue along its boundary, or
		\item if the edge is not in the tree, then we proceed along the short side (skipping the edge), and then on along the boundary of the same disk. 
	\end{itemize}
\end{remark}

Bernardi proved the following:

\begin{lemma}[{\cite[Lemma 5]{Bernardi_first}}]
\label{l:T-tour_cyclic_perm}
In the tour of a spanning tree $T$, each edge $xy$ of $G$ becomes current edge twice, in one case with $x$ as current vertex, and in the other case with $y$ as current vertex.
\end{lemma}
 
In other words, the tour of $T$ lists the pairs $(u,\varepsilon)$, where $u$ is a vertex of $G$ and $\varepsilon$ is an edge of $G$ incident to $u$, in a linear order from smallest to largest. We will denote this ordering by $<_T$, and write $(u,\varepsilon)\leq_T (v,\varepsilon')$ if either $u=v$ and $\varepsilon=\varepsilon'$, or $(u,\varepsilon)<_T (v,\varepsilon')$.

Now $\leq_T$ induces an ordering of the edges of $G$: Let $uv$ be smaller than $zw$ if $\min_{\leq_T}\{(u,uv),(v,uv)\}\leq_T \min_{\leq_T}\{(z,zw),(w,zw)\}$. We denote this order, too by $\leq_T$.

The \emph{internal and external embedding activities} of a spanning tree $T$ of $G$ are defined as the internal and external activities of $T$ with respect to the order $\leq_T$ of edges. Let us denote them by $ie(T)$ and $ee(T)$, respectively. That is,
\begin{itemize}
\item $ie(T)$ is the number of edges $\varepsilon$ of $T$ so that $\varepsilon$ is the $<_T$-minimal element of the fundamental cut $C^*(T,\varepsilon)$;
\item $ee(T)$ is the number of non-edges $\varepsilon$ of $T$ so that $\varepsilon$ is the $<_T$-minimal element of the fundamental cycle $C(T,\varepsilon)$.
\end{itemize}
Bernardi gave the alternative definition
\[T_G(x,y)=\sum_{T \text{ is a spanning tree in }G} x^{ie(T)}y^{ee(T)}\]
for the Tutte polynomial of a graph. In particular, it follows from his results that this expression does not depend on the ribbon structure, base vertex, or base edge.

The tour of a spanning tree can also be used to give a bijection between spanning trees and so called break divisors. This was done originally by Bernardi \cite{Bernardi_first} (not yet using the terminology of break divisors), then further studied by Baker and Wang \cite{Baker-Wang}. A \emph{break divisor} for a graph $G=(V,E)$ is a vector $\mathbf z\in\Z^V$ so that $\mathbf{d}-\mathbf{1}-\mathbf z$ is a hypertree on $V$ in $\bip G$, that is the graph obtained from $G$ by adding a new vertex halfway along each edge. (Here $\mathbf{d}-\mathbf{1}$ denotes the vector whose $v$-component is the degree of $v$ in $G$ minus one for each $v\in V$.) Baker and Wang \cite{Baker-Wang} proved that given a spanning tree $T$, if one takes its tour and at each vertex, counts the non-edges of $T$ that first become current in conjunction with that vertex, the resulting values give a break divisor. Conversely, given a break divisor $\mathbf z$, one may start a walk on the ribbon graph and whenever an edge is encountered which is such that $\mathbf z$ remains a break divisor in the smaller graph after the edge's removal, cut that edge. By the end of the walk, the remaining edges form a spanning tree.

Said in another way, the Bernardi process on a graph gives a bijection between the sets of hypertrees on $V$ and on $E$. Our Bernardi process for hypergraphs generalizes this bijection (see Remark \ref{rem:Bernardi_altalanosit_BW} for more detail).

\section{The root polytope and its dissections} 
\label{sec:root_polytope}

The root polytope of a bipartite graph first appeared in the work of Ohsugi and Hibi \cite{hibi} and was defined anew by Postnikov \cite{alex}. For a detailed list of its properties we refer the reader to \cite[Section 3]{KP_Ehrhart} (and to \cite{alex} for many of the proofs). Here we only repeat the most important points. Let $G$ be a bipartite graph with color classes $E$ (as in emerald) and $V$ (as in violet). 

\begin{definition} 
\label{def:rootpolytope}
In the Euclidean space $\R^E\oplus\R^V$ let us write $\mathbf x$ for the standard basis vector that corresponds to $x\in E \cup V$. The \emph{root polytope} of $G$, denoted by $Q_G$, is the convex hull of the vectors $\mathbf e +\mathbf v$ for all edges $ev$ of $G$.
\end{definition}

We get an isometric polytope if we replace $\mathbf e +\mathbf v$ with $\mathbf e -\mathbf v$ in the construction. Note how the definition is inspired by the standard proof that a graphic matroid is representable. However in the theory that Postnikov built it is important that we specifically use real coefficients and examine $Q_G$ from the point of view of convex geometry. It turns out that the root polytope reflects certain properties of the bipartite graph in a non-trivial and very effective way.
 
A set of vertices of $Q_G$ is affine independent if and only if the corresponding subgraph of $G$ is cycle-free. (Note that vertices of $Q_G$ and edges of $G$ correspond bijectively.) In particular, the dimension of $Q_G$ is one less than the number of edges in a spanning forest of $G$. In the case when $G$ is connected, which we usually assume, a spanning forest is a spanning tree, and the dimension is $|E|+|V|-2$. The one-to-one correspondence (for connected $G$)
\[\{\,\text{spanning trees of }G\,\}\longleftrightarrow\{\,\text{maximal simplices of }Q_G\,\}\]
will be crucial for the rest of the paper. Note that the simplex corresponding to the tree $T$ is none other than its root polytope $Q_T$.

The maximal simplices of $Q_G$ share the same volume. There is a description \cite[Lemma 12.6]{alex} for when two maximal simplices intersect in a common face, given in terms of the corresponding spanning trees: two trees are compatible in this sense if and only if there does not exist a cycle in $G$ so that its first, third, fifth etc.\ edges come from one tree and its second, fourth, sixth etc.\ edges come from the other. A \emph{triangulation} of the root polytope is a collection of maximal simplices whose union is $Q_G$ and each pair of which do intersect in a common face. When the second condition is weakened to require only that the interiors of the simplices be disjoint, we get the notion of a \emph{dissection}. The observation on volumes implies that the number of maximal simplices in a dissection of $Q_G$ depends only on $G$. In Examples \ref{ex:non-triangulation} and \ref{ex:otszog} we exhibit dissections that are not triangulations.

A thorough look into dissections reveals some spectacular properties. Let us fix a dissection of $Q_G$ and consider the corresponding collection of spanning trees in $G$. We claim that any hypertree (either on $E$ or on $V$) is realized by exactly one of our chosen trees. (Consequently the numbers of hypertrees on $E$ and on $V$ are the same.) This is established for triangulations in \cite{alex}, but the same proof applies in general, as follows. 

The polytope $Q_G$ contains the set of \emph{emerald markers} (which form an affine transformation of the set $B_E$ of hypertrees on $E$)
\[\frac1{|V|}B_E+\frac1{|E|\cdot|V|}\mathbf i_E+\frac1{|V|}\mathbf i_V\]
and a similarly defined set of \emph{violet markers}. Here $\mathbf i_S$ stands for the characteristic function of a subset $S$ of $E\cup V$, viewed as a vector in $\R^E\oplus\R^V$. Any maximal simplex in $Q_G$ contains, in its interior, exactly one marker of each color: these are (essentially) the hypertrees on $E$ and on $V$ realized by the tree that corresponds to the simplex \cite[Lemma 14.9]{alex}. If some simplices form a dissection, then it is also true that each marker is contained by a unique simplex. Hence the two sets of markers (that is, the two sets of hypertrees $B_E$ and $B_V$) are equinumerous with each other and with the set of maximal simplices in our (arbitrary) dissection. In particular, the following holds.

\begin{thm}
\label{l:quasitriang_gives_bijection}
Let $\mathcal{T}$ be a set of spanning trees of a bipartite graph $G$ such that the simplices $\{\,Q_T\mid T\in\mathcal{T}\,\}$ form a dissection of $Q_G$. Then the mapping assigning $\mathbf f_V(T)$ to $\mathbf f_E(T)$ for each $T\in\mathcal{T}$ is a bijection between $B_E$ and $B_V$.
\end{thm}

By \cite[Theorem 3.10]{KP_Ehrhart}, the $h$-vector of any triangulation of $Q_G$ is equivalent to the Ehrhart polynomial of $Q_G$. This property, too, extends to dissections, at least in those cases when we are able to define an $h$-vector for such objects. In other words, one is able to generalize \cite[Remark 3.12]{KP_Ehrhart} as follows.

Let us call a dissection \emph{shellable} if it has a \emph{shelling order}, that is, a total order of the maximal simplices so that, starting from the second one, each maximal simplex $\sigma$ intersects the union of the previous ones in a non-empty union of facets (codimension one faces) of $\sigma$. For a dissection with a shelling order, let $a_i$ denote the number of maximal simplices for which the number of said facets is exactly $i$. We put $a_0=1$, accounting for the first simplex in the order. Let us define the \emph{$h$-vector} of the shelling order to be the finite sequence $(a_0,a_1,\ldots)$.

The Ehrhart polynomial of $Q_G$ is the unique polynomial $\varepsilon_G$ such that for nonnegative integers $k$, we have $\varepsilon_{G}(k)=|(k\cdot Q_G)\cap(\Z^E\oplus\Z^V)|$. If we have a shellable dissection of $Q_G$ (for a connected $G$), then, putting $d=\dim Q_G=|E|+|V|-2$, the Ehrhart polynomial of $Q_G$ can be expressed as 
\begin{equation}
\label{eq:Ehrhart}
\varepsilon_{G}(k)=a_0\binom{d+k}{d}+a_1\binom{d+k-1}{d}+\cdots+a_i\binom{d+k-i}{d}+\cdots.
\end{equation}

The proof of \eqref{eq:Ehrhart} is based on \cite[Lemma 3.3]{KP_Ehrhart} and an easy simplex-by-simplex counting argument. Indeed, $\binom{d+k-i}{d}$ is the number of lattice points in a standard $d$-dimensional simplex of sidelength $k-i$, and \cite[Lemma 3.3]{KP_Ehrhart} says that for our purposes, all maximal simplices in the root polytope (and all their faces) behave just like standard simplices. Now if, in the $k$ times inflated root polytope, the lattice points along $i$ of the facets of a maximal simplex have already been counted, then what remains to count is the lattice points in a simplex of sidelength not $k$ but $k-i$.

Since the binomial coefficients in \eqref{eq:Ehrhart} are linearly independent as polynomials of $k$ \cite[Lemma 3.8]{KP_Ehrhart}, from the uniqueness of the Ehrhart polynomial it follows that the $h$-vectors of all shelling orders of all shellable dissections of $Q_G$ coincide. (Triangulations have $h$-vectors even when they are not shellable. For triangulations of $Q_G$ by maximal simplices, all their $h$-vectors are the same, too \cite[Theorem 3.10]{KP_Ehrhart}.) Furthermore, \cite[Equation (5.1)]{KP_Ehrhart} (the main theorem of that paper) states that the same coefficient sequence gives the interior polynomial of both hypergraphs $(V,E)$ and $(E,V)$ that are induced by $G$:

\begin{thm}
\label{t:interior_poly_ehrhart_h-vector}
For a connected hypergraph $\HH=(V,E)$, if the Ehrhart polynomial $\varepsilon_{\bip \HH}$ of $Q_{\bip \HH}$ can be expressed as 
\begin{equation*}
\varepsilon_{\bip \HH}(k)=a_0\binom{d+k}{d}+a_1\binom{d+k-1}{d}+\cdots+a_i\binom{d+k-i}{d}+\cdots,
\end{equation*}
where $d=\dim Q_G=|E|+|V|-2$, then the interior polynomial of $\HH$ is 
\[I_{\HH}(x)=a_0+a_1x+\cdots+a_ix^i+\cdots.\]
\end{thm}

Both sums in Theorem \ref{t:interior_poly_ehrhart_h-vector} are of course finite. The largest $i$ so that $a_i\ne0$ is definitely no more than $d+1$; by \cite[Proposition 6.1]{hiperTutte} it is in fact at most $\min\{\,|E|,|V|\,\}-1$. For more on the degree of the interior polynomial, see \cite{Frank-K}.

K. Kato \cite{Kato} found a concise formulation of Theorem \ref{t:interior_poly_ehrhart_h-vector}, restating it as a connection between $I_\HH$ and the Ehrhart series $\mathrm{Ehr}_{\bip \HH}(x)=\sum_{s=0}^\infty\varepsilon_{\bip \HH}(s)\,x^s$ of $Q_{\bip\HH}$. Namely, if $\bip\HH$ is connected, then we have
\[\frac{I_\HH(x)}{(1-x)^{|E|+|V|-1}}=\mathrm{Ehr}_{\bip \HH}(x),\]
in other words, the interior polynomial of $\HH$ is in fact the $h^*$-vector of $Q_{\bip\HH}$.

\section{The Bernardi process for hypergraphs}
\label{sec:Bernardi-process}

In this section we describe two 
processes, both of which generalize what Bernardi defined for ordinary graphs. Sometimes, in order to distinguish them from Bernardi's original algorithm, we will refer to them as the \emph{hypergraphical Bernardi processes}. In both cases, the input consists of 
\begin{enumerate}[label=(\alph*)]
\item\label{egy} a connected hypergraph $\HH=(V,E)$
\item\label{ketto} a ribbon structure for $\bip\HH$, a base node, and a base edge
\item\label{harom} a hypertree $\mathbf f$ in $\HH$, that is, on $E$.
\end{enumerate}
Here \ref{egy} and \ref{harom} generalize Bernardi's inputs of a connected graph and a spanning tree. There is a slight difference between the graphs $G$ and $\bip G$ in that the latter is obtained from the former by placing a new vertex halfway along each edge. Bernardi fixed a ribbon structure on $G$ and not on $\bip G$, but as the new vertices of $\bip G$ are of degree $2$, the structure extends uniquely to $\bip G$. Therefore \ref{ketto} above also generalizes what Bernardi used in his approach.
 
Let us first describe the Bernardi process informally, in terms of a walk on $\bip\HH$, traversing or cutting edges as we go. We outline two different processes, depending on whether we are allowed to cut edges at their violet or emerald endpoint. (One of the two will be allowed, the other prohibited.) If the walk reaches an edge from the endpoint where we are not allowed to cut, there is no choice but to traverse the edge and continue on the other side. In the other case we will have a choice: either traverse the edge as above, or cut it (remove it from the graph) and continue with the next edge incident to our current node. What governs this choice is whether the values of $\mathbf f$ still define a hypertree after cutting the edge. (That is, whether the smaller graph still has a spanning tree that realizes $\mathbf f$.) If they do, we cut; otherwise, we keep and traverse. (Note that this is the same principle with which Baker and Wang define their mapping from break divisors to spanning trees.)

We formalize these ideas as follows. Let the base node be $b_0$ and the base edge be $b_0b_1$. See Figure \ref{fig:apapirrasokatir} for an illustration of Definition \ref{d:Bernardi,vi,vi}.

\begin{definition}[Bernardi process, hypertree on emerald nodes, cut at violet nodes (ht:$E$, cut:$V$)] \label{d:Bernardi,em,vi}
Given is a hypertree $\mathbf f$ on $E$. The process maintains a current edge and a current graph at any moment. At the beginning, the current graph is $\bip \HH$. If $b_0$ is a violet node, then at the beginning, the current edge is $b_0b_1$. If $b_0$ is an emerald node, then the current edge at the beginning is $b_1b_0^+$, and we say that $b_0b_1$ was traversed from the emerald direction.

In each step, we check whether for the current graph $G$ and the current edge $ve$, the vector $\mathbf f$ is a hypertree on the emerald nodes of $G-ve$. If the answer is yes, let the current graph of the next step be $G-ve$ and the current edge be $ve_G^+$. We say that $ve$ was \emph{removed} or \emph{deleted} from the graph. If the answer is no, let the current graph of the next step be $G$,  and let the current edge of the next step be $we_G^+$, where $ew = ev_G^+$. In this case we say that $ve$ is traversed from the violet direction and $ew$ is traversed from the emerald direction.

The process stops right before when an edge would be traversed for the second time from the same direction.
\end{definition}

\begin{definition}[Bernardi process, hypertree on emerald nodes, cut at emerald nodes (ht:$E$, cut:$E$)]
\label{d:Bernardi,vi,vi}
Given is a hypertree $\mathbf f$ on $E$. The process maintains a current edge and a current graph at any moment. At the beginning, the current graph is $\bip \HH$. If $b_0$ is an emerald node, then at the beginning, the current edge is $b_0b_1$. If $b_0$ is a violet node, then the current edge at the beginning is $b_1b_0^+$, and we say that $b_0b_1$ was traversed from the violet direction.

In each step, we check whether for the current graph $G$ and the current edge $ev$, the vector $\mathbf f$ is a hypertree on the emerald nodes of $G-ev$. If the answer is yes, let the current graph of the next step be $G-ev$ and the current edge be $ev_G^+$. We say that $ev$ was \emph{removed} or \emph{deleted} from the graph. If the answer is no, let the current graph of the next step be $G$,  and let the current edge of the next step be $dv_G^+$, where $vd = ve_G^+$. In this case we say that $ev$ is traversed from the emerald direction and $vd$ is traversed from the violet direction.

The process stops right before when an edge would be traversed for the second time from the same direction.
\end{definition}

In both cases the current graphs form a decreasing sequence. We say that an edge was \emph{kept} by the process if it was examined as a current edge, and was not removed from the current graph. If an edge $\varepsilon$ is kept then `$\mathbf f$ cannot be realized without it' in the current graph, that is, $\varepsilon$ is part of any spanning tree realizing $\mathbf f$ in the current graph and hence in all subsequent current graphs, too. In particular, once an edge is kept, it can never get removed --- the decision of keeping is final, just like the decision of removal.

We say that an edge has been \emph{traversed} by the process if it has been traversed either from the violet direction or from the emerald direction. Traversed edges form an increasing sequence of subgraphs of $\bip\HH$. In fact, the graph of traversed edges will always be a subgraph of the current graph, but for this we have yet to show that traversed edges never get removed (cf.\ Lemma \ref{l:traversed_edge_is_not_deleted}). The problem of course is with the edges $ew$ of Definition \ref{d:Bernardi,em,vi} and $vd$ of Definition \ref{d:Bernardi,vi,vi}, which our walk traverses without examining them as current edges. It turns out that even if they are previously unexamined edges, later they will be examined and kept, but these facts are not obvious from the definition. For now, let us reinforce that even if a traversed edge was later removed, we would still count it as traversed.

\begin{figure}[h]
\begin{tikzpicture}[scale=.23]

\path [fill=gray] (0,4.25) to [out=225,in=120] (-1,1.5) to [out=-60,in=200] (2,1) to [out=110,in=-110] (2,3) to [out=160,in=-45] (0,4.25);
\path [fill=gray] (6,1) to [out=-20,in=-120] (9,1.5) to [out=60,in=-45] (8,4.25) to [out=225,in=20] (6,3) to [out=-70,in=70] (6,1);
\path [fill=gray] (6,7.5) to [out=110,in=0] (4,9.7) to [out=180,in=70] (2,7.5) to [out=-20,in=135] (4,6.25) to [out=45,in=200] (6,7.5);
\path [fill=lightgray] (2,3) to [out=70,in=225] (4,6.25) to [out=-45,in=110] (6,3) to [out=200,in=-20] (2,3);
\path [fill=lightgray] (6,7.5) to [out=20,in=120] (9,7) to [out=-60,in=45] (8,4.25) to [out=135,in=-70] (6,7.5);
\path [fill=lightgray] (0,4.25) to [out=135,in=-120] (-1,7) to [out=60,in=160] (2,7.5) to [out=-110,in=45] (0,4.25);
\path [fill=lightgray] (6,1) to [out=-110,in=0] (4,-1.2) to [out=180,in=-70] (2,1) to [out=20,in=160] (6,1);

\draw [ultra thick,red] (0,2) -- (4,0);
\draw [ultra thick,red] (4,0) -- (8,2);
\draw [ultra thick,red] (8,2) -- (8,6.5);
\draw [ultra thick,red] (8,6.5) -- (4,8.5);
\draw [ultra thick,red] (4,8.5) -- (0,6.5);
\draw [ultra thick,red] (0,6.5) -- (0,2);
\draw [ultra thick,red] (0,2) -- (4,4);
\draw [ultra thick,red] (4,4) -- (4,8.5);
\draw [ultra thick,red] (4,4) -- (8,2);

\draw [fill=blue,blue] (4, 0) circle [radius=0.4];
\draw [fill=blue,blue] (4, 4) circle [radius=0.4];
\draw [fill=blue,blue] (0, 6.5) circle [radius=0.4];
\draw [fill=blue,blue] (8, 6.5) circle [radius=0.4];
\draw [fill=green,green] (0, 2) circle [radius=0.4];
\draw [fill=green,green] (8, 2) circle [radius=0.4];
\draw [fill=green,green] (4, 8.5) circle [radius=0.4];
\draw [fill] (-1,1.5) circle [radius=.3];

\draw [ultra thick] (-.15,4.1) to [out=225,in=120] (-1,1.5) to [out=-60,in=200] (2,1) to [out=20,in=160] (5.8,1.1);
\draw [ultra thick] (6.2,.9) to [out=-20,in=-120] (9,1.5) to [out=60,in=-45] (8,4.25) to [out=135,in=-70] (6.1,7.3);
\draw [ultra thick] (5.9,7.7) to [out=110,in=0] (4,9.7) to [out=180,in=70] (2,7.5) to [out=-110,in=45] (.15,4.4);
\draw [ultra thick] (1.9,1.2) to [out=110,in=-110] (2,3) to [out=70,in=225] (3.85,6.1);
\draw [ultra thick] (4.15,6.4) to [out=45,in=200] (6,7.5) to [out=20,in=120] (9,7) to [out=-60,in=45] (8.15,4.4);
\draw [ultra thick] (7.85,4.1) to [out=225,in=20] (6,3) to [out=200,in=-20] (2.2,2.9);
\draw [ultra thick] (1.8,3.1) to [out=160,in=-45] (0,4.25) to [out=135,in=-120] (-1,7) to [out=60,in=160] (1.8,7.6);
\draw [ultra thick] (2.2,7.4) to [out=-20,in=135] (4,6.25) to [out=-45,in=110] (5.9,3.2);
\draw [ultra thick] (6.1,2.8) to [out=-70,in=70] (6,1) to [out=-110,in=0] (4,-1.2) to [out=180,in=-70] (2.1,.8);

\draw [<-,thick] (-2,2) arc [radius=2,start angle=180,end angle=240];

\draw [fill=blue,blue] (18, 0) circle [radius=0.6];
\draw [fill=blue,blue] (18, 4) circle [radius=0.6];
\draw [fill=blue,blue] (14, 6.5) circle [radius=0.6];
\draw [fill=blue,blue] (22, 6.5) circle [radius=0.6];
\draw [fill=green,green] (14, 2) circle [radius=0.6];
\draw [fill=green,green] (22, 2) circle [radius=0.6];
\draw [fill=green,green] (18, 8.5) circle [radius=0.6];
\draw [fill] (13,1.5) circle [radius=.3];

\node at (14,2) {{\tiny{$1$}}};
\node at (22,2) {{\tiny{$2$}}};
\node at (18,8.5) {{\tiny{$0$}}};

\begin{scope}[shift={(14,0)}]
\draw [lightgray,ultra thick] (13,1.5) to [out=120,in=-135] (14,4.25) to [out=45,in=250] (16,7.5) to [out=70,in=180] (18,9.7);

\draw [ultra thick, red] (18,8.5) -- (14,6.5);
\draw [line width=5, red] (14,6.5) -- (14,2);
\draw [fill=blue,blue] (18, 0) circle [radius=0.6];
\draw [fill=blue,blue] (18, 4) circle [radius=0.6];
\draw [fill=blue,blue] (14, 6.5) circle [radius=0.6];
\draw [fill=blue,blue] (22, 6.5) circle [radius=0.6];
\draw [fill=green,green] (14, 2) circle [radius=0.6];
\draw [fill=green,green] (22, 2) circle [radius=0.6];
\draw [fill=green,green] (18, 8.5) circle [radius=0.6];
\draw [fill] (13,1.5) circle [radius=.3];

\node at (14,2) {{\tiny{$1$}}};
\node at (22,2) {{\tiny{$2$}}};
\node at (18,8.5) {{\tiny{$0$}}};
\end{scope}

\begin{scope}[shift={(28,0)}]
\draw [lightgray,ultra thick] (13,1.5) to [out=120,in=-135] (14,4.25) to [out=45,in=250] (16,7.5) to [out=70,in=180] (18,9.7) to [out=0,in=120] (19.2,9) to [out=-60,in=20] (18.8,7.3);
\draw [ultra thick,red] (18,8.5) -- (14,6.5);
\draw [ultra thick,red] (14,6.5) -- (14,2);
\draw [dashed,line width=3,red] (18,8.5) -- (22,6.5);
\draw [fill=blue,blue] (18, 0) circle [radius=0.6];
\draw [fill=blue,blue] (18, 4) circle [radius=0.6];
\draw [fill=blue,blue] (14, 6.5) circle [radius=0.6];
\draw [fill=blue,blue] (22, 6.5) circle [radius=0.6];
\draw [fill=green,green] (14, 2) circle [radius=0.6];
\draw [fill=green,green] (22, 2) circle [radius=0.6];
\draw [fill=green,green] (18, 8.5) circle [radius=0.6];
\draw [fill] (13,1.5) circle [radius=.3];

\node at (14,2) {{\tiny{$1$}}};
\node at (22,2) {{\tiny{$2$}}};
\node at (18,8.5) {{\tiny{$0$}}};
\end{scope}

\begin{scope}[shift={(-14,-13)}]
\draw [lightgray,ultra thick] (13,1.5) to [out=120,in=-135] (14,4.25) to [out=45,in=250] (16,7.5) to [out=70,in=180] (18,9.7) to [out=0,in=120] (19.2,9) to [out=-60,in=20] (18.8,7.3) to [out=200,in=-10] (17.2,7.2);
\draw [ultra thick,red] (18,8.5) -- (14,6.5);
\draw [ultra thick,red] (14,6.5) -- (14,2);
\draw [dashed,thick,red] (18,8.5) -- (22,6.5);
\draw [dashed,line width=3,red] (18,8.5) -- (18,4);
\draw [fill=blue,blue] (18, 0) circle [radius=0.6];
\draw [fill=blue,blue] (18, 4) circle [radius=0.6];
\draw [fill=blue,blue] (14, 6.5) circle [radius=0.6];
\draw [fill=blue,blue] (22, 6.5) circle [radius=0.6];
\draw [fill=green,green] (14, 2) circle [radius=0.6];
\draw [fill=green,green] (22, 2) circle [radius=0.6];
\draw [fill=green,green] (18, 8.5) circle [radius=0.6];
\draw [fill] (13,1.5) circle [radius=.3];

\node at (14,2) {{\tiny{$1$}}};
\node at (22,2) {{\tiny{$2$}}};
\node at (18,8.5) {{\tiny{$0$}}};
\end{scope}

\begin{scope}[shift={(0,-13)}]
\draw [lightgray,ultra thick] (13,1.5) to [out=120,in=-135] (14,4.25) to [out=45,in=250] (16,7.5) to [out=70,in=180] (18,9.7) to [out=0,in=120] (19.2,9) to [out=-60,in=20] (18.8,7.3) to [out=200,in=-10] (17.2,7.2) to [out=170,in=-20] (16,7.5) to [out=160,in=60] (13,7) to [out=-120,in=135] (14,4.25) to [out=-45,in=150] (14.8,3.6);
\draw [line width=5,red] (18,8.5) -- (14,6.5);
\draw [ultra thick,red] (14,6.5) -- (14,2);
\draw [dashed,thick,red] (18,8.5) -- (22,6.5);
\draw [dashed,thick,red] (18,8.5) -- (18,4);
\draw [fill=blue,blue] (18, 0) circle [radius=0.6];
\draw [fill=blue,blue] (18, 4) circle [radius=0.6];
\draw [fill=blue,blue] (14, 6.5) circle [radius=0.6];
\draw [fill=blue,blue] (22, 6.5) circle [radius=0.6];
\draw [fill=green,green] (14, 2) circle [radius=0.6];
\draw [fill=green,green] (22, 2) circle [radius=0.6];
\draw [fill=green,green] (18, 8.5) circle [radius=0.6];
\draw [fill] (13,1.5) circle [radius=.3];

\node at (14,2) {{\tiny{$1$}}};
\node at (22,2) {{\tiny{$2$}}};
\node at (18,8.5) {{\tiny{$0$}}};
\end{scope}

\begin{scope}[shift={(14,-13)}]
\draw [lightgray,ultra thick] (13,1.5) to [out=120,in=-135] (14,4.25) to [out=45,in=250] (16,7.5) to [out=70,in=180] (18,9.7) to [out=0,in=120] (19.2,9) to [out=-60,in=20] (18.8,7.3) to [out=200,in=-10] (17.2,7.2) to [out=170,in=-20] (16,7.5) to [out=160,in=60] (13,7) to [out=-120,in=135] (14,4.25) to [out=-45,in=150] (14.8,3.6) to [out=-30,in=90] (15.8,2);
\draw [ultra thick,red] (18,8.5) -- (14,6.5);
\draw [ultra thick,red] (14,6.5) -- (14,2);
\draw [dashed,thick,red] (18,8.5) -- (22,6.5);
\draw [dashed,thick,red] (18,8.5) -- (18,4);
\draw [dashed,line width=3,red] (14,2) -- (18,4);
\draw [fill=blue,blue] (18, 0) circle [radius=0.6];
\draw [fill=blue,blue] (18, 4) circle [radius=0.6];
\draw [fill=blue,blue] (14, 6.5) circle [radius=0.6];
\draw [fill=blue,blue] (22, 6.5) circle [radius=0.6];
\draw [fill=green,green] (14, 2) circle [radius=0.6];
\draw [fill=green,green] (22, 2) circle [radius=0.6];
\draw [fill=green,green] (18, 8.5) circle [radius=0.6];
\draw [fill] (13,1.5) circle [radius=.3];

\node at (14,2) {{\tiny{$1$}}};
\node at (22,2) {{\tiny{$2$}}};
\node at (18,8.5) {{\tiny{$0$}}};
\end{scope}

\begin{scope}[shift={(28,-13)}]
\draw [lightgray,ultra thick] (13,1.5) to [out=120,in=-135] (14,4.25) to [out=45,in=250] (16,7.5) to [out=70,in=180] (18,9.7) to [out=0,in=120] (19.2,9) to [out=-60,in=20] (18.8,7.3) to [out=200,in=-10] (17.2,7.2) to [out=170,in=-20] (16,7.5) to [out=160,in=60] (13,7) to [out=-120,in=135] (14,4.25) to [out=-45,in=150] (14.8,3.6) to [out=-30,in=90] (15.8,2) to [out=-90,in=110] (16,1) to [out=-70,in=180] (18,-1.2) to [out=0,in=-110] (20,1) to [out=70,in=-90] (20.2,2);
\draw [ultra thick,red] (18,8.5) -- (14,6.5);
\draw [ultra thick,red] (14,6.5) -- (14,2);
\draw [dashed,thick,red] (18,8.5) -- (22,6.5);
\draw [dashed,thick,red] (18,8.5) -- (18,4);
\draw [dashed,thick,red] (14,2) -- (18,4);
\draw [line width=5,red] (14,2) -- (18,0);
\draw [ultra thick,red] (18,0) -- (22,2);
\draw [fill=blue,blue] (18, 0) circle [radius=0.6];
\draw [fill=blue,blue] (18, 4) circle [radius=0.6];
\draw [fill=blue,blue] (14, 6.5) circle [radius=0.6];
\draw [fill=blue,blue] (22, 6.5) circle [radius=0.6];
\draw [fill=green,green] (14, 2) circle [radius=0.6];
\draw [fill=green,green] (22, 2) circle [radius=0.6];
\draw [fill=green,green] (18, 8.5) circle [radius=0.6];
\draw [fill] (13,1.5) circle [radius=.3];

\node at (14,2) {{\tiny{$1$}}};
\node at (22,2) {{\tiny{$2$}}};
\node at (18,8.5) {{\tiny{$0$}}};
\end{scope}

\begin{scope}[shift={(-14,-26)}]
\draw [lightgray,ultra thick] (13,1.5) to [out=120,in=-135] (14,4.25) to [out=45,in=250] (16,7.5) to [out=70,in=180] (18,9.7) to [out=0,in=120] (19.2,9) to [out=-60,in=20] (18.8,7.3) to [out=200,in=-10] (17.2,7.2) to [out=170,in=-20] (16,7.5) to [out=160,in=60] (13,7) to [out=-120,in=135] (14,4.25) to [out=-45,in=150] (14.8,3.6) to [out=-30,in=90] (15.8,2) to [out=-90,in=110] (16,1) to [out=-70,in=180] (18,-1.2) to [out=0,in=-110] (20,1) to [out=70,in=-90] (20.2,2) to [out=90,in=-70] (20,3) to [out=110,in=0] (18,5.2) to [out=180,in=120] (16.8,3.4) to [out=-60,in=200] (20,3) to [out=20,in=210] (21.2,3.6);
\draw [ultra thick,red] (18,8.5) -- (14,6.5);
\draw [ultra thick,red] (14,6.5) -- (14,2);
\draw [dashed,thick,red] (18,8.5) -- (22,6.5);
\draw [dashed,thick,red] (18,8.5) -- (18,4);
\draw [dashed,thick,red] (14,2) -- (18,4);
\draw [ultra thick,red] (14,2) -- (18,0);
\draw [ultra thick,red] (18,0) -- (22,2);
\draw [line width=5,red] (22,2) -- (18,4);
\draw [fill=blue,blue] (18, 0) circle [radius=0.6];
\draw [fill=blue,blue] (18, 4) circle [radius=0.6];
\draw [fill=blue,blue] (14, 6.5) circle [radius=0.6];
\draw [fill=blue,blue] (22, 6.5) circle [radius=0.6];
\draw [fill=green,green] (14, 2) circle [radius=0.6];
\draw [fill=green,green] (22, 2) circle [radius=0.6];
\draw [fill=green,green] (18, 8.5) circle [radius=0.6];
\draw [fill] (13,1.5) circle [radius=.3];

\node at (14,2) {{\tiny{$1$}}};
\node at (22,2) {{\tiny{$2$}}};
\node at (18,8.5) {{\tiny{$0$}}};
\end{scope}

\begin{scope}[shift={(0,-26)}]
\draw [lightgray,ultra thick] (13,1.5) to [out=120,in=-135] (14,4.25) to [out=45,in=250] (16,7.5) to [out=70,in=180] (18,9.7) to [out=0,in=120] (19.2,9) to [out=-60,in=20] (18.8,7.3) to [out=200,in=-10] (17.2,7.2) to [out=170,in=-20] (16,7.5) to [out=160,in=60] (13,7) to [out=-120,in=135] (14,4.25) to [out=-45,in=150] (14.8,3.6) to [out=-30,in=90] (15.8,2) to [out=-90,in=110] (16,1) to [out=-70,in=180] (18,-1.2) to [out=0,in=-110] (20,1) to [out=70,in=-90] (20.2,2) to [out=90,in=-70] (20,3) to [out=110,in=0] (18,5.2) to [out=180,in=120] (16.8,3.4) to [out=-60,in=200] (20,3) to [out=20,in=210] (21.2,3.6) to [out=30,in=225] (22,4.25) to [out=45,in=-60] (23.2,7.1) to [out=120,in=60] (20.8,7.1) to [out=-120,in=135] (22,4.25) to [out=-45,in=60] (23,1.5);
\draw [ultra thick,red] (18,8.5) -- (14,6.5);
\draw [ultra thick,red] (14,6.5) -- (14,2);
\draw [dashed,thick,red] (18,8.5) -- (22,6.5);
\draw [dashed,thick,red] (18,8.5) -- (18,4);
\draw [dashed,thick,red] (14,2) -- (18,4);
\draw [ultra thick,red] (14,2) -- (18,0);
\draw [ultra thick,red] (18,0) -- (22,2);
\draw [ultra thick,red] (22,2) -- (18,4);
\draw [line width=5,red] (22,2) -- (22,6.5);
\draw [fill=blue,blue] (18, 0) circle [radius=0.6];
\draw [fill=blue,blue] (18, 4) circle [radius=0.6];
\draw [fill=blue,blue] (14, 6.5) circle [radius=0.6];
\draw [fill=blue,blue] (22, 6.5) circle [radius=0.6];
\draw [fill=green,green] (14, 2) circle [radius=0.6];
\draw [fill=green,green] (22, 2) circle [radius=0.6];
\draw [fill=green,green] (18, 8.5) circle [radius=0.6];
\draw [fill] (13,1.5) circle [radius=.3];

\node at (14,2) {{\tiny{$1$}}};
\node at (22,2) {{\tiny{$2$}}};
\node at (18,8.5) {{\tiny{$0$}}};
\end{scope}

\begin{scope}[shift={(14,-26)}]
\draw [lightgray,ultra thick] (13,1.5) to [out=120,in=-135] (14,4.25) to [out=45,in=250] (16,7.5) to [out=70,in=180] (18,9.7) to [out=0,in=120] (19.2,9) to [out=-60,in=20] (18.8,7.3) to [out=200,in=-10] (17.2,7.2) to [out=170,in=-20] (16,7.5) to [out=160,in=60] (13,7) to [out=-120,in=135] (14,4.25) to [out=-45,in=150] (14.8,3.6) to [out=-30,in=90] (15.8,2) to [out=-90,in=110] (16,1) to [out=-70,in=180] (18,-1.2) to [out=0,in=-110] (20,1) to [out=70,in=-90] (20.2,2) to [out=90,in=-70] (20,3) to [out=110,in=0] (18,5.2) to [out=180,in=120] (16.8,3.4) to [out=-60,in=200] (20,3) to [out=20,in=210] (21.2,3.6) to [out=30,in=225] (22,4.25) to [out=45,in=-60] (23.2,7.1) to [out=120,in=60] (20.8,7.1) to [out=-120,in=135] (22,4.25) to [out=-45,in=60] (23,1.5) to [out=-120,in=-20] (20,1) to [out=160,in=20] (16,1) to [out=200,in=-60] (13,1.5);
\draw [ultra thick,red] (18,8.5) -- (14,6.5);
\draw [ultra thick,red] (14,6.5) -- (14,2);
\draw [dashed,thick,red] (18,8.5) -- (22,6.5);
\draw [dashed,thick,red] (18,8.5) -- (18,4);
\draw [dashed,thick,red] (14,2) -- (18,4);
\draw [ultra thick,red] (14,2) -- (18,0);
\draw [line width=5,red] (18,0) -- (22,2);
\draw [ultra thick,red] (22,2) -- (18,4);
\draw [ultra thick,red] (22,2) -- (22,6.5);
\draw [fill=blue,blue] (18, 0) circle [radius=0.6];
\draw [fill=blue,blue] (18, 4) circle [radius=0.6];
\draw [fill=blue,blue] (14, 6.5) circle [radius=0.6];
\draw [fill=blue,blue] (22, 6.5) circle [radius=0.6];
\draw [fill=green,green] (14, 2) circle [radius=0.6];
\draw [fill=green,green] (22, 2) circle [radius=0.6];
\draw [fill=green,green] (18, 8.5) circle [radius=0.6];
\draw [fill] (13,1.5) circle [radius=.3];

\node at (14,2) {{\tiny{$1$}}};
\node at (22,2) {{\tiny{$2$}}};
\node at (18,8.5) {{\tiny{$0$}}};
\end{scope}

\begin{scope}[shift={(42,-26)}]
\draw [ultra thick,red] (4,0) -- (8,2);
\draw [ultra thick,red] (8,2) -- (8,6.5);
\draw [ultra thick,red] (0,6.5) -- (0,2);
\draw [ultra thick,red] (0,2) -- (4,4);
\draw [ultra thick,red] (4,4) -- (4,8.5);
\draw [ultra thick,red] (4,4) -- (8,2);

\draw [fill=blue,blue] (4, 0) circle [radius=0.4];
\draw [fill=blue,blue] (4, 4) circle [radius=0.4];
\draw [fill=blue,blue] (0, 6.5) circle [radius=0.4];
\draw [fill=blue,blue] (8, 6.5) circle [radius=0.4];
\draw [fill=green,green] (0, 2) circle [radius=0.4];
\draw [fill=green,green] (8, 2) circle [radius=0.4];
\draw [fill=green,green] (4, 8.5) circle [radius=0.4];
\end{scope}

\end{tikzpicture}
\caption{The (ht:$E$, cut:$E$) Bernardi process, run on a hypergraph (with ribbon structure) and the indicated hypertree. The last panel shows the outcome of the (ht:E, cut:V) 
process on the same hypertree.}
\label{fig:apapirrasokatir}
\end{figure}

\begin{ex}
\label{ex:apapirrasokatir}
Let $G$ be the plane bipartite graph shown in Figure \ref{fig:apapirrasokatir}, with the three emerald (green) nodes forming the color class $E$ and the four violet (blue) ones the color class $V$. We will refer to the elements of $E$ as top, left, and right. A ribbon structure is chosen so that at emerald nodes the cyclic order of the incident edges is clockwise, whereas at violet nodes it is counterclockwise. In the first panel of Figure \ref{fig:apapirrasokatir} we show (using a particular embedding in $3$-space) the associated ribbon surface that was described in Remark \ref{rem:surface}. Also indicated (in the lower left) is the base point; equivalently, our base node is the left emerald node and our base edge is the vertical one on the left.

The numbers $0,1,2$ written over the emerald nodes form a hypertree on $E$. In panels $2$--$11$ we show how the (ht:$E$, cut:$E$) Bernardi process operates with this input. At first there is just the hypertree and the process is at the base point. 

In the first step, which is probably the most interesting one in this example, the base edge is current. The hypertree calls for a realization (a tree) of degree $1+1=2$ at the base node, so one might expect that the first of the three incident edges will be removed. However if we did that, then in the remaining graph, all four edges of the bottom rhombus would be `wanted' by the hypertree: the left two to make the base node degree $2$, and the right two to make the right emerald node degree $2+1=3$. In other words, in the remaining graph it would not be possible to realize the given numbers as the degrees (minus $1$) of a spanning tree. Thus the Bernardi process will not remove the base edge, rather it will traverse it, which then will force it to traverse the upper left edge of $G$ as well.

In the next eight panels we show the remaining steps of the process. The current edge of each step is highlighted. The gray curves are included to help keep track of the cyclic orders of the ribbon structure, but they can also be understood as portions of a continuous path similar to the one we described in Remark \ref{rem:topological_interpretation_of_tour_of_a_tree}. The decisions are rather straightforward: removal, removal, traversal (in fact for the second time; note that if we did not traverse here, the top node would become isolated so it could no longer have degree $0+1=1$), removal, traversal (of the current edge and one more edge), traversal (here and in the next step, as the violet vertex we reach is already a leaf, the other edge we are forced to traverse coincides with the current edge), traversal, traversal. Note that each edge was current exactly once and that at the end, the subgraph of those edges that we did not remove coincides with the subgraph of traversed edges, and this subgraph is a spanning tree that realizes the given hypertree.

The last panel of Figure \ref{fig:apapirrasokatir} depicts the outcome of the (ht:$E$, cut:$V$) Bernardi process with the same ribbon structure and on the same input hypertree. It is again a spanning tree realization. We leave it to the interested reader to construct the steps leading to it, and to check that again there are nine such steps, with each edge becoming current exactly once. (Note that in this case the first current edge is not the base edge, rather it is the upper left one; the base edge becomes current in the second step.)
\end{ex}

We will often think of part \ref{harom} of the input, the hypertree, as a `variable,' so that the process itself is determined by \ref{egy} and \ref{ketto} only. In this sense we may in fact speak of four Bernardi processes on the ribbon bipartite graph $\bip\HH$, by applying the two definitions above to $\HH$ and to $\overline\HH$. In the latter case the hypertree is given on $V$ and we denote those two processes with (ht:$V$, cut:$E$) and (ht:$V$, cut:$V$).

Many more instances of the process can be generated by varying part \ref{ketto} of the input. See for example Lemma \ref{l:emerald_tree_is_also_violet} for the case when the ribbon structure is reversed (even though that process turns out not to be completely new).

\begin{remark}
\label{rem:altalanosit}
Both of our processes for hypergraphs do indeed generalize Bernardi's original process for graphs, and in fact, they can be thought of as a common generalization of the Bernardi process and its inverse, given by Baker and Wang \cite{Baker-Wang}. That is, for each spanning tree $T$ of the ribbon graph $G$, Bernardi's tour of $T$ is basically equivalent to our walk on $\bip G$, where the latter is defined using the uniquely extended ribbon structure as part \ref{ketto} of the input and the characteristic function of $T$ as part \ref{harom}. We sketch the main ideas of an induction proof. 

First note that a (current vertex, current edge) pair $(x,xy)$ in $G$ can be equivalently given as a `current half-edge' $xe$ between $x$ and the node $e$ placed at the center of $xy$. While $xe$ is a half-edge in $G$, it is an edge in $\bip G$. In particular, where Bernardi chose a base vertex $b_0$ and base edge $e_0=b_0b_1$, we may speak instead of the base node $b_0$ and the base edge $b_0e_0$.

If $xy$ is in the spanning tree $T$, then the corresponding hypertree has the value $1$ at $e$ and hence to realize it in $\bip G$, both $ex$ and $ey$ are necessary. Therefore both versions of our process will traverse both of those edges. In the first version, having arrived at $xe$ `near' $x$, we decide to keep it and then the ribbon structure at $e$ forces us to traverse $ey$, too. In the second version we have to traverse $xe$ anyway and then at $e$ we make the decision to keep $ey$ as well. After this we continue with the (half-)edge that follows $ey$ in the cyclic order at $y$. In both cases this is exactly what the original process would have done, too. 

If $xy$ is not in $T$, that is when the hypertree takes the value $0$ at $e$, we consider two sub-cases. If $ey$ has not been cut thus far, then the first version of our process, upon arrival at $xe$, will cut $xe$ because it is not necessary for a realization of the hypertree. The second version will traverse $xe$ and then cut $ey$, which forces it to backtrack to $x$. Else if $ey$ is not in the graph any more, then the first version will decide to keep $xe$ so that $e$ does not become an isolated node; the second version will have to traverse $xe$ anyway but in both versions, since $e$ was already a leaf, the process will bounce back to $x$ and continue with the edge that follows $xe$ in the cyclic order at $x$. This again matches the behavior of the original Bernardi process.

When applied to a graph $G$, our versions of the Bernardi process do not only trace a given spanning tree $T$, they also select one of the two half-edges of each non-edge $\varepsilon$ of $T$ so as to enlarge $T$ into a spanning tree of $\bip G$. As is clear from the above, the difference between the two generalizations is whether this half-edge is opposite to (first version) or on the same side as (second version) the first endpoint of $\varepsilon$ that Bernardi's tour of $T$ visits.
\end{remark}

The main theorem of this section draws identical conclusions for the processes of Definitions \ref{d:Bernardi,em,vi} and \ref{d:Bernardi,vi,vi}. It is convenient to state and prove it for the process of Definition \ref{d:Bernardi,em,vi} applied to $\HH$, coupled with the process of Definition \ref{d:Bernardi,vi,vi} applied to $\overline\HH$, so that in both cases we are allowed to cut edges at their violet endpoints. We do not lose any generality by this because $\HH=\overline{(\overline\HH)}$.

\begin{thm}
\label{thm:Bernardi_well_def}
For any hypertree $\mathbf f$ on $E$ (respectively, on $V$), the (ht:$E$, cut:$V$) Bernardi process (respectively, the (ht:$V$, cut:$V$) Bernardi process) takes each edge of $\bip \HH$ exactly once as current edge. The current graph at the end of the process is a spanning tree of $\bip \HH$ realizing $\mathbf f$.
\end{thm}

This theorem generalizes Lemma \ref{l:T-tour_cyclic_perm}, cf.\ Remark \ref{rem:altalanosit}. In Section \ref{sec:examples} we will see that it also generalizes \cite[Theorem 10.1]{hiperTutte}. 
It is no wonder then that the proof is somewhat lengthy\footnote{The referee assigned by Algebraic Combinatorics kindly suggested that a proof by induction on the number of edges of $\bip \HH$ might work, too. We agree, but since it does not seem significantly shorter, we will not pursue it here. On the other hand, late in the publication process we found yet another, indeed more brief argument for Theorem \ref{thm:Bernardi_well_def}. Instead of realizing hypertrees, it is based on ``realizing'' arbitrary points of the root polytope. It will be included in a future publication.}. The key will be Lemma \ref{l:no_cycl_in_Bernardi}, in particular the construction of the tree $T'$ after Claim \ref{p:each_edge_in_base_cut}, aided by Figure \ref{fig:T-walk}.

By a \emph{violet Bernardi run} we mean a running of either the (ht:$E$, cut:$V$) Bernardi process or the (ht:$V$, cut:$V$) Bernardi process on some hypertree.

\begin{lemma}
\label{c:hedges_traversed_in_good_order}
In a violet Bernardi run, if until some moment there is no cycle in the subgraph of traversed edges, then until that moment, 
\begin{enumerate}[label=(\roman*)]
\item\label{mindketszin} at any node $x\in V\cup E$, the edges of the current graph $G$ incident to $x$ that were traversed from the direction of $x$ are consecutive edges in the cyclic order at $x$ (in $G$), covering less than one full turn, and were traversed in a consecutive order (edges not incident to $x$ may also have been traversed between the times of these traversals).
\item\label{csaklila} at any $x\in V$, those edges of $\bip \HH$ incident to $x$ that have already been current edges, became current edges in an order compatible with the cyclic order at $x$ (in $\bip \HH$), covering less than one full turn.
\end{enumerate}
\end{lemma}

\begin{proof}
If $x=e\in E$, then, having arrived at $e$ on an edge $ve$, the Bernardi process next traverses $ev_{G'}^+$ by definition. Here $G'$ is the current graph at the time of these traversals. If there are any edges of $\bip\HH$ incident to $e$ that fall between $ev$ and $ev_{G'}^+$ in the cyclic order (at $e$ in $\bip\HH$), then they have already been cut at their violet endpoints. Hence in any later current graph, if $ev$ and $ev_{G'}^+$ are still present in it, then they are still consecutive.

The next time the walk associated to the process arrives at $e$, it arrives on $ev^+_{G'}$ since otherwise it would have traversed a cycle. To get our conclusion we just have to repeat our argument and note that if the sequence of traversed edges covered a full turn, then some edge would be traversed twice from the direction of $e$, which would cause the process to stop.

For a vertex $x=v\in V$, it suffices to show \ref{csaklila} since it implies \ref{mindketszin}. If the current edge $ve$ is removed from the graph, the next current edge is $ve_G^+$ by definition. Here $G$ is the current graph when $ve$ is the current edge. On the other hand, if $ve$ is traversed, the next time the process arrives at $v$, it has to arrive on $ev$, otherwise it would traverse a cycle. But then the next current edge incident to $v$ is $ve_G^+$. In this case, the $G$ in $ve_G^+$ a priori refers to the current graph at the moment of traversing $ev$, but we can also take it, as before, to mean the current graph when $ve$ is the current edge. Indeed the edge $ve_G^+$ can only be cut at $v$, so it cannot be cut while the process is away from $v$.

We claim that if the Bernardi process does not terminate upon returning to $v$ along $ev$, then $ve_G^+=ve_{\bip \HH}^+$, in other words, $ve_{\bip \HH}^+$ is not yet removed from the graph. Take the first moment when for a current edge $ve$, the edge $ve_{\bip \HH}^+$ is already missing from the current graph. Until this moment, the edges incident to $v$ became current edges in an order compatible with the cyclic order at $v$ (in $\bip \HH$), and as $ve_{\bip \HH}^+$ is already removed from the graph, it has already been a current edge. Hence we conclude that all the edges incident to $v$ have already become current edges. This includes  $ve_G^+$. As $ve_G^+$ is still part of the graph, it was kept (and thus traversed) when it first became a current edge. Now when it becomes current edge for the second time, it will be kept once again. Therefore it will be traversed again. But that means that the Bernardi process terminates after traversing $ev$.
\end{proof}

\begin{lemma} 
\label{l:no_cycl_in_Bernardi}
During a violet Bernardi run, the subgraph of traversed edges of $\bip \HH$ never contains a cycle.
\end{lemma}

\begin{proof}
Suppose that our violet Bernardi run uses the hypertree $\mathbf f$ as input (where $\mathbf f$ is a hypertree on $E$ or on $V$ depending on whether our process is of type (ht:$E$, cut:$V$) or (ht:$V$, cut:$V$)). Suppose for contradiction that after a while, a cycle appears in the subgraph of traversed edges of $\bip \HH$ and stop the process at the first moment when this occurs. Let $O$ be this cycle, and $G$ be the current graph at the moment. (Edges are typically traversed in pairs to form a $\text{violet}\to\text{emerald}\to\text{violet}$ path. If the first of the two edges completes a cycle, then we stop the process right there, midway through a step.)

In the rest of the proof we refer to this aborted process only, so that we may apply Lemma \ref{c:hedges_traversed_in_good_order} throughout. Note that $O$ is the only cycle in the subgraph of traversed edges, as the addition of an edge may create at most one cycle in a cycle-free graph.

We claim that for each edge of $O$, we can choose an orientation in which it was traversed, such that the chosen orientations give a cyclic orientation of $O$. Let us call this the \emph{positive orientation} of $O$. Indeed, if such an orientation did not exist, then there would be two edges $xy$ and $zw$ of $O$, such that $xy$ was only traversed from the direction of $x$ and $zw$ was only traversed from the direction of $w$, furthermore these two directions are opposite with respect to $O$. Suppose that $xy$ was first to be traversed among the two edges. Then after traversing $xy$ the walk associated to the process was at $y$. Later it needed to reach $w$ to be able to traverse $zw$ from the direction of $w$. But as neither $xy$ was traversed from the direction of $y$ nor $wz$ was traversed from the direction of $z$, the process could not go to $w$ using the edges of $O$, hence by the time the edges of $O$ are traversed, there needs to be another cycle in the subgraph of traversed edges, which is a contradiction.

Let us stress that by the time that all edges of $O$ have been traversed and the process is aborted, all edges of $O$ have been not just traversed but traversed in the positive direction.

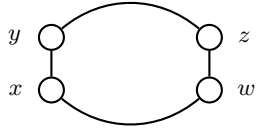
\begin{center}
\begin{tikzpicture}[-,>=stealth',auto,scale=0.35,
                    thick,every node/.style={circle,draw,font=\sffamily\small}]
  \node[label=right:{\small $w$}] (2) at (3, -1) {};
  \node[label=right:{\small $z$}] (3) at (3, 1) {};
  \node[label=left:{\small $y$}] (6) at (-3, 1) {};
  \node[label=left:{\small $x$}] (7) at (-3, -1) {};
  \path[every node/.style={font=\sffamily\small}]
    (2) edge node {} (3)
    (3) edge [bend right=40] node {} (6)
    (7) edge [bend right=40] node {} (2)
    (6) edge node {} (7);
\end{tikzpicture}
\end{center}

Pick a violet node along $O$ and call it $v_0$. Then name the other nodes of the cycle $e_0, v_1, e_1, \dots, v_{t-1}, e_{t-1}$ consecutively in the positive direction, i.e., so that $v_ie_i\in O$ and $e_iv_{i+1}\in O$ for each $i$, and $v_ie_i$ was traversed from $v_i$ and $e_iv_{i+1}$ was traversed from $e_i$ for each $i$ (where indices are now understood modulo $t$). Since this is a violet Bernardi run, all the edges of the form $v_ie_i$ are kept.

We will need the following two technical claims.

\begin{claim}\label{c:Ber_cycle_cl0}
An edge $e_{i-1}v_{i}$ cannot be deleted before $v_ie_i$ becomes current edge.
\end{claim}

\begin{proof}
Suppose on the contrary that for some $i$, the edge $e_{i-1}v_{i}$ is deleted before $v_ie_i$ becomes current edge. Since $e_{i-1}v_{i}$ is traversed by the Bernardi process, it is traversed from the emerald direction before it gets deleted. At the time of the traversal of $e_{i-1}v_{i}$ from the emerald direction, $v_ie_{i-1}$ cannot have been current edge yet (because then it would have gotten deleted before traversed) and the next current edge is $v_{i}e_{i-1}^+$. By Lemma \ref{c:hedges_traversed_in_good_order}, while there is no cycle in the subgraph of traversed edges, the edges incident to a violet node become current edges in an order compatible with the cyclic order. Hence $v_{i}e_ {i}$ needs to become current edge before $v_{i}e_{i-1}$, which is a contradiction.
\end{proof}

\begin{claim}\label{c:Ber_cycle_cl1}
An edge $e_iv_{i+1}$ cannot be deleted before $v_ie_i$ becomes current edge.
\end{claim}

\begin{proof}
Suppose on the contrary that for some $i$, the edge $e_iv_{i+1}$ is deleted before $v_ie_i$ becomes current edge. Since $e_iv_{i+1}$ is traversed by the Bernardi process, it is traversed from its emerald node $e_i$ (i.e., in the positive direction) before it gets deleted. As $v_ie_i$ does not become current edge before $e_iv_{i+1}$ is deleted, it also does not get traversed in the positive direction until that time. Hence there must be an edge $xy$ in $O$ such that the edges along the arc between $v_{i+1}$ and $x$ are all traversed in the positive direction after the traversal and 
before the deletion of $e_iv_{i+1}$, but $xy$ is not. Let $zx$ denote the edge of $O$ preceding $xy$ in the positive direction. We claim that $xy$ cannot have been traversed from the direction of $x$ before the traversal of $e_iv_{i+1}$, either. Indeed, if that was the case, then by the time that all edges of $O$ between $v_{i+1}$ and $x$ are traversed, there would be a cycle in the graph of traversed edges: The set of edges that we traverse during this time is connected and cannot be a tree because the set includes $e_iv_{i+1}$, which is only traversed in one direction. Moreover, this cycle is different from $O$, since it appears in the subgraph of traversed edges before the traversal of $v_ie_i$ in the positive direction. This contradiction shows that $xy$ is not traversed in the positive direction before the deletion of $e_iv_{i+1}$. From here on we separate two cases.

Case 1: $xy=v_{i+1}e_{i+1}$ (hence $z=e_i$).
As $v_{i+1}e_{i+1}$ is not deleted during the process, if it is not traversed until $e_iv_{i+1}$ is removed, then it does not become a current edge until $v_{i+1}e_i$ becomes current edge. After the traversal of $e_iv_{i+1}$ from the emerald direction, the next current edge is $v_{i+1}e_i^+$. By Lemma \ref{c:hedges_traversed_in_good_order}, while there is no cycle in the subgraph of traversed edges, the edges incident to a violet node become current edges in an order compatible with the cyclic order. Hence $v_{i+1}e_{i+1}$ needs to become current edge before $v_{i+1}e_i$, which is a contradiction.

Case 2: $xy\neq v_{i+1}e_{i+1}$.
This means that the Bernardi process reaches $x$ through the edge $zx$ before it deletes $e_iv_{i+1}$. Hence after reaching $x$, the walk associated to the process needs to get back to $v_{i+1}$ with $v_{i+1}e_i$ as current edge before traversing $xy$ from the direction of $x$. If the process does not traverse $xz$ from the direction of $x$ before it returns to $v_{i+1}$, then necessarily it traverses a cycle, from which we get a contradiction with the fact that there is no cycle in the graph of traversed edges until we traverse all the edges of $O$ in the positive direction. On the other hand, Lemma \ref{c:hedges_traversed_in_good_order} tells us that until all the edges of $O$ are traversed, the edges adjacent to $x$ are traversed (from the direction of $x$) in an order compatible with the ribbon structure. Hence $xy$ must be traversed before $xz$, which is again a contradiction.
\end{proof}

\begin{figure}
\begin{center}
	\begin{tikzpicture}[-,>=stealth',auto,scale=0.7,
	thick]
	\tikzstyle{c}=[{circle,draw,font=\sffamily\small,minimum size=0.5cm}]
	\tikzstyle{o}=[circle,fill,draw]
	\tikzstyle{b}=[circle]
	\node[c,color=blue] (0) at (2, -0.8) {};
	\node[b] (00) at (2, -0.8) {\small $v_0$};
	\node[c,color=green] (1) at (2, 0.8) {};
	\node[b] (11) at (2, 0.8) {\small $e_0$};
	\node[c,color=blue] (2) at (0.8, 2) {};
	\node[b] (12) at (0.8, 2) {\small $v_1$};
	\node[c,color=green] (3) at (-0.8, 2) {};
	\node[b] (13) at (-0.8, 2) {\small $e_1$};
	\node[c,color=blue] (4) at (-2, 0.8) {};
	\node[b] (14) at (-2, 0.8) {\small $v_2$};
	\node[c,color=green] (5) at (-2, -0.8) {};
	\node[b] (15) at (-2, -0.8) {\small $e_2$};
	\node[c,color=blue] (6) at (-0.8, -2) {};
	\node[b] (16) at (-0.8, -2) {\small $v_3$};
	\node[c,color=green] (7) at (0.8, -2) {};
	\node[b] (17) at (0.8, -2) {\small $e_3$};
	\node[o,label=right:{\small $u_0$}] (8) at (1.6, 3) {};
	\node[o,label=left:{\small $u_1$}] (9) at (-2.5, -1.8) {};
	\node[o,label=below:{\small $u_2$}] (10) at (-1.8, -2.5) {};
	\path[every node/.style={font=\sffamily\small}]
	(1) edge [dashed] node {} (2)
	(3) edge [dashed] node {} (4)
	(5) edge [dashed] node {} (6);
	\path[every node/.style={font=\sffamily\small}, line width=0.7mm]
	(0) edge [] node {} (1)
	(2) edge node {} (3)
	(4) edge node {} (5)
	(6) edge node {} (7)
	(7) edge node {} (0)
	(5) edge [bend right=10] node {} (9)
	(9) edge [bend right=10] node {} (10)
	(10) edge [bend right=10] node {} (6)
	(1) edge [bend right=20] node {} (8)
	(8) edge [bend right=20] node {} (3);
	\end{tikzpicture}
	\hspace{0.8cm}
	\begin{tikzpicture}[->,>=stealth',auto,scale=1.5,
	thick]
	\tikzstyle{c}=[{circle,draw,font=\sffamily\small}]
	\tikzstyle{o}=[circle,fill,draw]
	\node[c] (0) at (1, 0) {\small $0$};
	\node[c] (1) at (0, 1) {\small $1$};
	\node[c] (2) at (-1, 0) {\small $2$};
	\node[c] (3) at (0, -1) {\small $3$};
	\path[every node/.style={font=\sffamily\small}]
	(0) edge [bend right=15] node {} (1)
	(2) edge node {} (1)
	(3) edge node {} (1)
	(1) edge [bend right=15] node {} (0);
	\end{tikzpicture}
\end{center}
\caption{A cycle $O$, spanning tree $T$, and the corresponding auxiliary graph in the (ht:$E$, cut:$V$) case. The $T$-walk of $O$ is $v_0$,\,$e_0$,\,$u_0$,\,$e_1$,\,$v_1$,\,$e_1$,\,$u_0$,\,$e_0$,\,$v_0$,\,$e_3$,\,$v_3$,\,$u_2$,\,$u_1$,\,$e_2$,\,$v_2$,\,$e_2$,\,$u_1$,\,$u_2$,\,$v_3$,\,$e_3$,\,$v_0$.}
\label{fig:T-walk}
\end{figure}

Let us now refocus on the moment when, upon the traversal of all edges of $O$, we stopped the Bernardi process. Take a spanning tree $T$ in the current graph $G$ that realizes $\mathbf f$. Such a tree exists by the definition of the Bernardi process and it contains every edge that we so far decided to keep. In particular, $T$ contains each edge of $O$ of the form $v_ie_i$. On the other hand, as $O\not\subseteq T$, there are edges of $O$ of the form $e_jv_{j+1}$ that are not in $T$. We will also need the following observation.

\begin{claim}\label{p:each_edge_in_base_cut}
	Each edge of the form $e_iv_i$ is in the fundamental cycle $C(T,e_jv_{j+1})$ for some $j$ such that $e_jv_{j+1}\notin T$.
\end{claim}

\begin{proof}
Let us take the following walk on $T$. Start from $v_0$, and traverse $v_0e_0$. If $e_0v_1\in T$, then traverse $e_0v_1$ as well. If not, then traverse the unique path in $T$ connecting $e_0$ to $v_1$. Continue this way until for each edge of $O$, the path in $T$ connecting its two endpoints is traversed. At the end we arrive back to $v_0$. Let us call this walk the $T$-walk of $O$.

As $T$ is a tree, in this walk, each traversed edge is traversed in both directions. Since each edge of the form $v_ie_i$ is traversed from the direction of $v_i$, each of these edges is also traversed in the reverse direction. By definition, an edge $v_ie_i$ is traversed from the direction of $e_i$ only if it is part of the path in $T$ connecting some nodes $e_j$ and $v_{j+1}$. The claim follows.
\end{proof}

From this point on, we prove Lemma \ref{l:no_cycl_in_Bernardi} separately for the (ht:$E$, cut:$V$) and for the (ht:$V$, cut:$V$) Bernardi processes. This is only for notational convenience as the two arguments remain very similar.

For the (ht:$E$, cut:$V$) Bernardi process, let us consider an auxiliary directed graph $D$ on the vertex set $\{0,\dots,t-1\}$. Draw an edge from $i$ to $j$ if $v_{i}e_{i}\in C(T,e_jv_{j+1})$. By Claim \ref{p:each_edge_in_base_cut}, each vertex has outdegree at least one in $D$. Hence $D$ contains at least one directed cycle. Take a directed cycle of minimal length. Let the vertex set of this cycle be $x_0,\dots,x_{r-1}$, where there is a directed edge from $x_i$ to $x_{i+1}$ (that is, $v_{x_i}e_{x_i}\in C(T,e_{x_{i+1}}v_{x_{i+1}+1})$) for each $i$ (meant modulo $r$).

We claim that $T'=\left(T\setminus\bigcup_{i=0}^{r-1} v_{x_i}e_{x_i}\right)\cup \bigcup_{i=0}^{r-1}e_{x_i}v_{x_i+1}$ is a spanning tree realizing the hypertree $\mathbf f$ on $E$. First we show that $T'$ is a spanning tree of $\bip \HH$. (The following  argument is a special case of a matroid theoretical lemma \cite[Theorem 39.13]{Schrijver_CombOpt}, but we include it for completeness.) Having chosen a minimal cycle in $D$ ensures that $v_{x_j}e_{x_j}\notin C(T,e_{x_i}v_{x_i+1})$ for any $i$ and $j$ such that $i\neq j+1$ (which is now meant modulo $r$), since otherwise there would be a shortcut in the cycle $x_0,\dots,x_{r-1}$, contradicting its minimality. Now $T_1=(T\setminus v_{x_0}e_{x_0})\cup e_{x_1}v_{x_1+1}$ is a spanning tree since $v_{x_0}e_{x_0}\in C(T,e_{x_1}v_{x_1+1})$. Moreover, since $v_{x_0}e_{x_0}\notin C(T,e_{x_j}v_{x_j+1})$ for $j\neq 1$, we have $C(T_1,e_{x_j}v_{x_j+1})=C(T,e_{x_j}v_{x_j+1})$ for all $j\neq 1$. Hence when we replace the edge $v_{x_1}e_{x_1}$ of $T_1$ with $e_{x_2}v_{x_2+1}$, the result is another tree $T_2$ and further, the fundamental cycles of $e_{x_3}v_{x_3+1},\ldots,e_{x_{r-1}}v_{x_{r-1}+1},e_{x_0}v_{x_0+1}$ with respect to $T_2$ are still the same as with respect to $T$. By a trivial induction proof, we may continue switching edges like this until we arrive at $T'$.

Now to show that $T'$ realizes $\mathbf f$, notice that by construction the degree of each emerald node is the same in $T$ as in $T'$. As $T$ realizes $\mathbf f$, so does $T'$.

Let $I=\{x_0,\dots,x_{r-1}\}$. Consider the first moment when one of the edges $v_\ell e_\ell$ becomes current edge for an index $\ell\in I$. By Claim \ref{c:Ber_cycle_cl1}, at this moment, for any $j\in I$, the edges $v_j e_j$ and $e_j v_{j+1}$ are present in the current graph. Hence $T'$ is a spanning tree in the current graph, realizing $\mathbf f$ and not containing $v_\ell e_\ell$. This contradicts the fact that $v_\ell e_\ell$ was kept by the process.

For the (ht:$V$, cut:$V$) Bernardi process, we take a similar auxiliary directed graph $D$ on the vertex set $\{\,1,\dots,t\,\}$. This time we draw an edge from $i$ to $j$ if $v_{i}e_{i}\in C(T,e_{j-1}v_j)$. Then again, each vertex has outdegree at least one in $D$ and thus $D$ contains at least one directed cycle. We take a directed cycle of minimal length and let its vertex set be $x_0,\dots,x_{r-1}$.

The fact that $T'=T\setminus\bigcup_{i=0}^{r-1} v_{x_i}e_{x_i}\cup \bigcup_{i=1}^{r}e_{x_i-1}v_{x_i}$ is a spanning tree realizing $\mathbf f$ (on $V$) can be established in the same way as in the previous case. 

Let $I=\{x_0,\dots,x_{r-1}\}$. Take the first moment when an edge of the type $v_\ell e_\ell$ becomes current edge for an index $\ell\in I$. By Claim \ref{c:Ber_cycle_cl0}, at this moment, for any $j\in I$, the edges $v_j e_j$ and $e_{j-1} v_{j}$ are present in the current graph. Hence $T'$ is a spanning tree in the current graph realizing $\mathbf f$ and not containing $v_\ell e_\ell$. This again contradicts the fact that $v_\ell e_\ell$ was kept.

Finally, as for both versions of the process the existence of $O$ led to a contradiction, we conclude that the Bernardi process may never traverse a cycle.
\end{proof}

\begin{lemma}\label{l:traversed_edge_is_not_deleted}
If an edge is traversed during a violet Bernardi run, then it is not removed later.
\end{lemma}

\begin{proof}
If the edge $\varepsilon$ is traversed from the direction of its violet node, then it is kept and, as we have already argued after Definition \ref{d:Bernardi,vi,vi}, is never removed later.
	
Now consider the case when $\varepsilon=ev$ is traversed from the direction of its emerald node $e$. Suppose for contradiction that later on $\varepsilon$ gets deleted.  

First we establish that $v$ is first reached (by the walk associated to the process) through $ev$. Suppose that $v$ was first reached through another edge $e'v$. Consider the path from the first arrival to $v$ until $ev$ gets traversed. Since $ve$ does not get traversed from the violet direction (as it is due to be deleted), this path arrives back to $v$ so that it only traverses $ev$ in one direction. Hence there is a cycle in the set of traversed edges, contradicting Lemma \ref{l:no_cycl_in_Bernardi}.

Next, stop the Bernardi process immediately after the deletion of $\varepsilon$, and let the current graph of that moment be $G$. It is connected because $\mathbf f$ is a hypertree in it. Let $S$ be the set of nodes that were reached by the Bernardi process until the deletion of $\varepsilon$, and were first reached after reaching $v$. Let also $v\in S$ which guarantees $S\ne\varnothing$. As $e$ was reached before $v$ we have $e\notin S$ and hence $S\ne E\cup V$.

Take a vertex $x\in S$. At the current moment, the Bernardi process is at $v$, as the last step was deleting $ve$. If $x\neq v$, then when we last left $x$ before returning to $v$, we must have left it along the same edge through which we first reached it, because otherwise there would be two paths between $v$ and $x$ and hence a cycle in the subgraph of traversed edges. By Lemma \ref{c:hedges_traversed_in_good_order}, the edges incident to $x$ in $G$ that were traversed from the direction of $x$ are consecutive edges according to the cyclic order at $x$ (in $G$). Hence all the edges incident to $x$ in $G$ are already traversed from the direction of $x$. Note that this is also true for $x=v$ as the last current edge was $ve$ which is also the edge through which $v$ was first reached.

Finally, consider an edge $xy$ of $\bip \HH$ such that $x\in S$ and $y\in (E\cup V) \setminus S$. If $xy$ is in $G$, then $xy$ was traversed from the direction of $x$. Hence $y$ is already reached. As $y\notin S$, the node $y$ was also reached before reaching $v$. As $x$ was first reached after reaching $v$ (or else $x=v$), there is a path in the subgraph of traversed edges from $y$ to $x$ through $v$. But then after the traversal of $xy$ from the direction of $x$, there would be a cycle in the subgraph of traversed edges. Thus there can be no edge $xy$ in $G$ such that $x\in S$ and $y\in (E\cup V) \setminus S$. Hence $G$ is disconnected, which is a contradiction. (In other words, right before its deletion $\varepsilon$ was the only edge connecting $S$ to its complement and thus it should not have been deleted.)
\end{proof}

\begin{proof}[Proof of Theorem \ref{thm:Bernardi_well_def}]
By Lemma \ref{l:traversed_edge_is_not_deleted}, the subgraph of traversed edges is part of the current graph at all times, and by Lemma \ref{l:no_cycl_in_Bernardi}, it is cycle-free. The current graph is always connected as it contains a spanning tree realizing $\mathbf f$. Hence if we make sure that at the end of the process, the subgraph $H$ of traversed edges coincides with the current graph $G$, then we will know that $H=G$ is both cycle-free and connected, thus a spanning tree, and since it contains a realization of $\mathbf f$, it itself has to be such a realization. 

Now for this last remaining claim, it suffices to prove the assertion in the Theorem that each edge becomes current edge exactly once, because then in particular each edge is either deleted (thus belongs neither to $G$ nor to $H$) or kept and traversed (hence belongs to both). In fact, it is enough to show that the process examines each edge of $\bip \HH$ as a current edge at least once, i.e., by the time an edge would be traversed for the second time from the same direction, each edge has been current edge. This is because if an edge becomes current for a second time, then it has to have been kept the first time and thus traversed both times, ending the process.

We shall first consider where the process may end, and find that it is only possible at the base node $b_0$.

Suppose that the first edge to be traversed twice from the same direction is $ev$, and it is traversed from the direction of the emerald node $e$. If $e\neq b_0$, then take the edge $ue$ through which we first reached $e$ during the process. As the Bernardi process does not traverse a cycle, by Lemma \ref{c:hedges_traversed_in_good_order}, the edges incident to $e$ are traversed in an order compatible with the cyclic order at $e$. Hence before $ev$ is traversed for the second time, $eu$ is traversed from the emerald direction (by Lemma \ref{l:traversed_edge_is_not_deleted} it cannot be deleted from the graph, as it was already traversed). But as, again, the Bernardi process does not traverse a cycle, after traversing $eu$ from the emerald direction, we can only get back to $e$ by traversing $ue$ from the violet direction. Hence before $ev$ gets traversed for the second time from the emerald direction, $ue$ gets traversed twice from the violet direction, which is a contradiction.

Now suppose that the first edge to be traversed twice from the same direction is $ve$, and it is traversed from the violet direction, that is, from $v$. If $v\neq b_0$, then there must be an edge $e'v$ through which we first reached $v$. As $e'v$ is traversed, by Lemma \ref{l:traversed_edge_is_not_deleted}, it is not deleted from the graph during the process. Now we can repeat the argument of the previous case: The Bernardi process does not traverse a cycle, and by Lemma \ref{c:hedges_traversed_in_good_order}, the edges incident to $v$ are traversed in an order compatible with the cyclic order at $v$. Hence before $ve$ is traversed for the second time from the violet direction, $ve'$ is traversed from the violet direction (as it is not deleted). But as the Bernardi process does not traverse a cycle, after traversing $ve'$ from the violet direction, it can only get back to $v$ by traversing $e'v$ from the emerald direction. Hence before $ve$ gets traversed for the second time from the violet direction, $e'v$ gets traversed twice from the emerald direction, which is a contradiction.

We conclude that the first edge to be traversed twice by the Bernardi process can only be an edge incident to $b_0$, and it is traversed from the direction of $b_0$.

Next we claim that if a violet node $v\in V\setminus\{b_0\}$ is reached by the Bernardi process, then all the edges incident to it get examined as current edges. Indeed, as the Bernardi process does not traverse a cycle, the last time the process leaves $v$ before returning to $b_0$, it must leave $v$ on the same edge on which it was first reached. Hence from Lemma \ref{c:hedges_traversed_in_good_order}, by this time each edge incident to $v$ is examined as a current edge. This also holds for $b_0$, if $b_0\in V$, since the current edge at the end of the process is an edge that has already been a current edge. We also claim that if a node $e\in E$ is reached by the Bernardi process at all, then all the edges incident to it get either removed or traversed. Indeed, if $e\ne b_0$, then the last time the process leaves $e$ before returning to $b_0$, it must leave $e$ through the same edge on which it was first reached, and by Lemma \ref{c:hedges_traversed_in_good_order}, by this time the traversed edges incident to $e$ have been traversed in an order compatible with the cyclic order at $e$ in the current graph of that moment. Thus if an edge incident to $e$ has not yet been deleted, then it has necessarily been traversed. Again, the same holds if $e=b_0$.

To complete the proof, we need to show that each node in $V$ is reached by the process. If that was not the case, then the vertex set $X$ of the subgraph $H$ of traversed edges (at the end of the process) would be a proper subset of $E\cup V$, giving rise to a cut $K$ in $\bip\HH$. As $\bip\HH$ is connected, $K$ is a non-empty set of edges. By the previous paragraph, all edges in $K$ get examined, and hence removed, by the process. We claim that $K$ is such that each of its edges has its violet end in $X$. Indeed, edges can only be cut by the process at their violet ends and violet nodes in $(E\cup V)\setminus X$ are never reached. But this means that when the last edge of $K$ was removed, the current graph became disconnected, which contradicts the basic principle of the Bernardi process.
\end{proof}

By interchanging the roles of $E$ and $V$, we see that the statement of Theorem \ref{thm:Bernardi_well_def} holds also for the (ht:$V$, cut:$E$) and the (ht:$E$, cut:$E$) Bernardi processes.

Now that we have made sure that the Bernardi process converges to a spanning tree $T$, we may spell out the relationship between the walk associated to our process and Bernardi's original tour of $T$. As the process traverses exactly the edges of $T$, we see that the two are essentially the same, with the only difference being that we only consider edges as current edge once, whereas Bernardi's tour does so twice, once at each endpoint.
More precisely, we have the following.

\begin{lemma}
\label{lem:walk_vs_tour}
Let $T$ be a spanning tree of the ribbon graph $\bip\HH$ (with fixed base node and base edge), resulting from a violet Bernardi run on some hypertree (on $E$ or on $V$). If we list the edges of $\bip\HH$ in the order in which they become current in the tour of $T$ and delete 
\begin{itemize}
\item the second occurrences of the edges not in $T$ and
\item for each edge of $T$, its occurrence in conjunction with its emerald endpoint,
\end{itemize}
then we get the list in which the edges of $\bip \HH$ become current in our hypergraphical Bernardi process. 
\end{lemma}

\section{Hypergraph polynomials \`a la Bernardi}
\label{sec:statements}

Using the process of Definition \ref{d:Bernardi,vi,vi}, that is the version that cuts edges of the bipartite graph near where hypertree values (spanning tree degrees) are assigned, we are able to give an alternative definition (analogous to the definition of the Tutte polynomial by Bernardi) for the interior polynomial of a hypergraph. 
In section \ref{sec:shelling}, we show that this definition is equivalent to the original one; in particular, the polynomial does not depend on the choice of ribbon structure and base point. For the time being, we will denote this Bernardi-type interior polynomial by $\tilde{I}$ to avoid confusion. We will also set up an exterior version $\tilde X$. 

First we define embedding activities. Note that this would not work without Theorem \ref{thm:Bernardi_well_def}, in particular the claim that the Bernardi process reaches every part of the bipartite graph.

\begin{definition}
\label{def:hyperactivity}
Let $\HH=(V,E)$ be a hypergraph and $\mathbf f$ be a hypertree on $E$. Fix a ribbon structure on $\bip\HH$, base node, and base edge and run the (ht:$E$, cut:$E$) Bernardi process. Take the order in which the edges of $\bip \HH$ become current. This induces the following order on the elements of $E$:
\begin{equation}
\label{eq:hyperedge_order}
 e_1<_{\mathbf f}e_2 \quad \Longleftrightarrow \quad \text{\parbox{7cm}{some edge incident to $e_1$ becomes current edge earlier than any edge incident to $e_2$.}}
\end{equation}
The \emph{internal embedding inactivity} of $\mathbf f$ is the number of internally inactive (that is, not internally active) hyperedges with respect to $\mathbf f$ and the above order, defined exactly as in Definition \ref{def:activity}. We denote this quantity by $\overline{ie}(\mathbf f)$.

The \emph{external embedding inactivity} of $\mathbf f$ is the number of externally inactive hyperedges with respect to $\mathbf f$ and the above order, defined as in \cite[Definition 5.1]{hiperTutte}. We denote this quantity by $\overline{ee}(\mathbf f)$. That is, $\overline{ee}(\mathbf f)$ is the number of hyperedges which, with respect to $\mathbf f$ and $<_\mathbf f$, may receive a transfer of valence from a smaller hyperedge. 
\end{definition}

\begin{ex}
We return to Example \ref{ex:apapirrasokatir} (and its notation), where we ran the (ht:$E$, cut:$E$) Bernardi process on a hypertree $\mathbf f$ on $E$. From the order in which the edges of $\bip\HH$ became current (cf.\ Figure \ref{fig:apapirrasokatir}), we see that the induced order on $E$ is $\text{left}<_{\mathbf f}\text{top}<_{\mathbf f}\text{right}$, which is just the order in which the walk associated to the process reached the elements of $E$.

The smallest node is always active, both internally and externally. As the top node has the minimal possible $\mathbf f$-value, it is internally active (cannot transfer valence to left). Externally, however, it is inactive: it is easy to see (for instance by symmetry) that $\text{left}\mapsto0$, $\text{top}\mapsto1$, $\text{right}\mapsto2$ is also a hypertree, i.e., with respect to $\mathbf f$, top may receive a transfer of valence from left. Notice that it does not concern us at all whether the new hypertree has some realization related to the just-constructed realization of $\mathbf f$. Finally the right node is internally inactive (can transfer valence to both left and top) but externally active because its $\mathbf f$-value is the largest possible and hence may not receive any transfers of valence.

Thus, in this case we find the embedding inactivity values $\overline{ie}(\mathbf f)=\overline{ee}(\mathbf f)=1$.
\end{ex}

\begin{definition}
For a hypergraph $\HH=(V,E)$ with a set of hypertrees $B_E$ and with a ribbon structure on $\bip\HH$, we let
\begin{equation}
\label{eq:tilde_definitions}
\tilde{I}_\HH(\xi)=\sum_{\mathbf f\in B_E} \xi^{\overline{ie}(\mathbf f)}\quad\text{and}\quad\tilde{X}_\HH(\eta)=\sum_{\mathbf f\in B_E} \eta^{\overline{ee}(\mathbf f)}.
\end{equation}
\end{definition}

We prove the following theorem in Section \ref{sec:shelling}.

\begin{thm}\label{thm:Bernardi-interior_well_def}
For any connected hypergraph $\HH$, ribbon structure on $\bip\HH$, base node, and base edge, we have $\tilde{I}_\HH=I_\HH$. Here $I_\HH$ is the interior polynomial of subsection \ref{ssec:interior}, whereas $\tilde{I}_\HH$ is defined in \eqref{eq:tilde_definitions} using the (ht:$E$, cut:$E$) Bernardi process.
In particular, $\tilde{I}_{\HH}$ does not depend on the ribbon structure.
\end{thm}

The order $<_{\mathbf f}$ of \eqref{eq:hyperedge_order} has three more variants: Since any of the four versions, (ht:$E$ or $V$, cut:$E$ or $V$), of our process (assuming the presence of a hypertree) orders the set of edges of $\bip\HH$, each can be used to order $E$, as well as to order $V$, by the smallest of the incident edges. Then for a hypertree on, say, $E$, the (ht:$E$, cut:$V$) process can be used to define its embedding inactivities in complete analogy with Definition \ref{def:hyperactivity}, which used the (ht:$E$, cut:$E$) process.

When applied to a (ribbon) graph $\HH=G=(V,E)$ and one of its spanning trees $T$ (viewed as a hypertree on $E$), the (ht:$E$, cut:$E$) and (ht:$E$, cut:$V$) processes induce the same order on the set of edges $E$, namely the order $<_T$ of subsection \ref{subsec:Bernardi_for_graphs}. This is not true for general hypergraphs. Thus when we use the (ht:$E$, cut:$V$) process to define embedding inactivities (for hypertrees on $E$), the individual values of the statistic will be different from those provided by Definition \ref{def:hyperactivity}. Note also that the order on $E$ (induced by the hypertree on $E$) is less natural in this case: nodes in $E$ may `get numbered' when the walk associated to the process \emph{cuts} some edge adjacent to them, which may be well before the walk actually \emph{reaches} the node. Yet computer simulations suggest that the distribution of the interior inactivity statistic remains the same. Unfortunately, at the moment, we are only able to state this as a conjecture.

\begin{conj}
\label{conj:interior_cutatvertex}
Let $\HH=(V,E)$ be a connected hypergraph with a ribbon structure and base point for $\bip\HH$. For a hypertree $\mathbf f$ on $E$, run the (ht:$E$, cut:$V$) Bernardi process and define the order $<_{\mathbf f}$ on $E$ as in \eqref{eq:hyperedge_order}. Compute the internal embedding inactivity of $\mathbf f$ with respect to $<_{\mathbf f}$ for all $\mathbf f$ and define a one-variable polynomial as in the first equation of \eqref{eq:tilde_definitions}. The result coincides with the interior polynomial $I_\HH$. 
\end{conj}

Again, Conjecture \ref{conj:interior_cutatvertex} is true for cases when $\HH$ is a graph by the coincidence of orders mentioned above. We also conjecture that both Theorem \ref{thm:Bernardi-interior_well_def} and Conjecture \ref{conj:interior_cutatvertex} are valid for the exterior polynomial as well.

\begin{conj}
\label{conj:exterior}
For any connected hypergraph $\HH$, ribbon structure on $\bip\HH$, base node, and base edge, we have $\tilde{X}_\HH=X_\HH$. In particular, the polynomial $\tilde{X}_{\HH}$ of \eqref{eq:tilde_definitions} does not depend on the ribbon structure. The same two claims remain true if we re-define $\tilde{X}_\HH$ using the (ht:$E$, cut:$V$) Bernardi process.
\end{conj}

We are probably much farther from the proof of Conjecture \ref{conj:exterior} than that of \ref{conj:interior_cutatvertex}. This is because we do not know what geometric object should play the same role, in relation to exterior activities, as the root polytope plays for interior activities in the argument that we are about to present.

\section{Jaeger trees}
\label{sec:Jaeger_trees}

In this section we give an alternative characterization of the set of spanning trees of $\bip \HH$ that the Bernardi process outputs. Here $\HH=(V,E)$ is a connected hypergraph and just like in the previous two sections, we work with a fixed ribbon structure, base node, and base edge.

\subsection{Definitions}

Since the notion of the `walk associated to the Bernardi process' of Section \ref{sec:Bernardi-process} only applies in the presence of a hypertree, and Jaeger trees will be defined without referring to hypertrees, we will revert to the definition of the tour of a spanning tree from subsection \ref{subsec:Bernardi_for_graphs}. For a spanning tree $T$ of the graph $\bip \HH$, we can take the tour of $T$ and hence obtain the order $\leq_T$ on the set of incident node-edge pairs of $\bip \HH$. If the edge $xy$ is not in $T$, then by Lemma \ref{l:T-tour_cyclic_perm}, it is skipped twice by the tour, once with each endpoint as current node. Let us consider the first of the two events and say that $xy$ is \emph{cut} at $x$ during the tour of $T$ if $(x,xy)<_T (y,xy)$. In this case we will also think of the `time' when $(x,xy)$ is current as the `moment when $xy$ is cut from the graph' (on our way of trimming $\bip\HH$ down to $T$). We also make the following definition. 

\begin{definition}
A spanning tree $T$ of $\bip \HH$ is called a \emph{$V$-cut} (resp.\ \emph{$E$-cut}) \emph{Jaeger tree} if in the tour of $T$, each non-edge of $T$ is cut at its violet (resp.\ emerald) endpoint.
\end{definition}

\begin{figure}
\begin{center}
	\begin{tikzpicture}[-,>=stealth',auto,scale=0.4,
	thick]
	\tikzstyle{c}=[{circle,draw,font=\sffamily\small}]
	\node[c,color=green,fill,label=left:{\small $b_1=e_0$}] (0) at (0, 0.1) {};
	\node[c,color=blue,fill,label=below:{\small $v_1$}] (1) at (3, 2) {};
	\node[c,color=blue,fill,label=left:{\small $b_0=v_0$}] (2) at (-3, 2) {};
	\node[c,color=green,fill,label=right:{\small $e_1$}] (3) at (3, 5) {};
	\node[c,color=green,fill, label=left:{\small $e_2$}] (4) at (-3, 5) {};
	\node[c,color=blue,fill, label=left:{\small $v_2$}] (5) at (0, 6.9) {};
	\node[c,color=green,fill, label=below:{\small $e_3$}] (6) at (0, 3.55) {};
	\path[every node/.style={font=\sffamily\small}, line width=0.2mm]
	(1) edge node {} (3)
	(4) edge node {} (5)
    (1) edge node {} (6);
	\path[every node/.style={font=\sffamily\small}, line width=0.8mm]
	(1) edge node {} (0)
	(2) edge node {} (4)
	(0) edge node {} (2)
	(5) edge node {} (3)
    (6) edge node {} (2)
	(5) edge node {} (6);
	\end{tikzpicture}
	\hspace{0.8cm}
	\begin{tikzpicture}[-,>=stealth',auto,scale=0.4,
	thick]
	\tikzstyle{c}=[{circle,draw,font=\sffamily\small}]
	\node[c,color=green,fill, label=left:{\small $b_1=e_0$}] (0) at (0, 0.1) {};
	\node[c,color=blue,fill,label=below:{\small $v_1$}] (1) at (3, 2) {};
	\node[c,color=blue,fill, label=left:{\small $b_0=v_0$}] (2) at (-3, 2) {};
	\node[c,color=green,fill, label=right:{\small $e_1$}] (3) at (3, 5) {};
	\node[c,color=green,fill, label=left:{\small $e_2$}] (4) at (-3, 5) {};
	\node[c,color=blue,fill, label=left:{\small $v_2$}] (5) at (0, 6.9) {};
	\node[c,color=green,fill, label=below:{\small $e_3$}] (6) at (0, 3.55) {};
	\path[every node/.style={font=\sffamily\small}, line width=0.2mm]
	(0) edge node {} (2)
    (2) edge node {} (4)
    (1) edge node {} (6);
	\path[every node/.style={font=\sffamily\small}, line width=0.8mm]
	(1) edge node {} (0)
	(1) edge node {} (3)
	(5) edge node {} (3)
	(4) edge node {} (5)
    (6) edge node {} (2)
	(5) edge node {} (6);
	\end{tikzpicture}
	\end{center}
\caption{For the (positive) planar ribbon structure, the spanning tree on the left is a $V$-cut Jaeger tree, while the spanning tree on the right is not a Jaeger tree.}
\label{fig:Jaeger-trees}
\end{figure}

We say that a spanning tree of $\bip \HH$ is a \emph{Jaeger tree} if it is either a $V$-cut or an $E$-cut Jaeger tree. I.e., Jaeger trees are distinguished by the property that in their tours, each non-edge is skipped for the first time at its endpoint of the same fixed color. 

\begin{ex}
\label{ex:kektura}
Let the positive orientation of the plane induce the ribbon structure of the bipartite graph of Figure \ref{fig:Jaeger-trees}. With $v_0$ as base node and $v_0e_0$ as base edge, the tour of the spanning tree on the left is $v_0, v_0e_0$; $e_0,e_0v_1$; $v_1, v_1e_1$; $v_1, v_1e_3$; $v_1, v_1e_0$; $e_0,e_0v_0$; $v_0,v_0e_3$; $e_3,e_3v_1$; $e_3,e_3v_2$; $v_2, v_2e_1$; $e_1, e_1v_1$; $e_1, e_1v_2$; $v_2, v_2e_2$; $v_2,v_2e_3$; $e_3, e_3v_0$; $v_0, v_0e_2$;  $e_2, e_2v_2$; $e_2, e_2v_0$. This is a $V$-cut Jaeger tree since, e.g., $(v_1,v_1e_1)<_T (e_1, e_1v_1)$ and a similar property holds for the other two non-edges of the tree. On the other hand, the spanning tree on the right is not a Jaeger tree since $v_0e_0$ is cut at its violet endpoint, while $e_3v_1$ is cut at its emerald endpoint.
\end{ex}

We will see that Jaeger trees are exactly the outcomes of our hypergraphical Bernardi processes. Let us prove the easier implication first.

\begin{prop}
\label{prop:bernardi->jaeger}
The (ht:$E$, cut:$V$) and the (ht:$V$, cut:$V$) Bernardi processes produce $V$-cut Jaeger trees for any hypertree (on $E$ or on $V$, respectively).
\end{prop}

\begin{proof}
Let $T$ be one of the spanning trees in question. In Lemma \ref{lem:walk_vs_tour} we spelled out the correspondence between the tour of $T$ and the hypergraphical Bernardi process which produced $T$. To put it simply, since in the Bernardi process an edge needs to be traversed if we first arrive at it from the emerald direction, in the tour of $T$, each edge not in $T$ has to be first reached, and cut, at its violet endpoint. Hence $T$ is indeed a $V$-cut Jaeger tree.
\end{proof}

With the same proof, we also obtain the following.

\begin{prop} \label{p:B-trees_are_J-trees}
The (ht:$V$, cut:$E$)  and the (ht:$E$, cut:$E$) Bernardi processes produce $E$-cut Jaeger trees for any hypertree.
\end{prop}

Our next aim is to show the converse of Proposition \ref{prop:bernardi->jaeger}, namely that each Jaeger tree is obtained as the outcome of the Bernardi process on some hypertree. Until we finish proving this (in Theorem \ref{t:J-trees_are_B-trees}) we will continue to rely on the notions of tour and current node-edge pair of Subsection \ref{subsec:Bernardi_for_graphs}, which apply to all spanning trees. First let us establish some useful properties of Jaeger trees. 

By the \emph{reverse} of a ribbon structure, we mean the structure where all cyclic orders are turned opposite. This is equivalent to reversing the orientation of the ribbon surface. Below we consider the reversed ribbon structure with the same base point, i.e., the same boundary point of the ribbon surface.

\begin{lemma}\label{l:emerald_tree_is_also_violet}
A $V$-cut Jaeger tree with base node $b_0$ and base edge $b_0b_1$ is an $E$-cut Jaeger tree for the reversed ribbon structure, base node $b_0$, and base edge $b_0b_1^-$ (where the latter is meant in the original ribbon structure).
\end{lemma}

\begin{proof} 
Let $T$ be a spanning tree of $\bip\HH$. In the tour of $T$, each non-edge $\varepsilon$ of $T$ is visited (i.e., is current) twice, once at each endpoint. The tree $T$ is $V$-cut if and only if the violet endpoint is always reached (i.e., forms a current pair with $\varepsilon$) before the emerald one. For the reversed ribbon structure, the order of the two occurrences of $\varepsilon$ in the tour becomes opposite. Hence $T$ is $V$-cut for the original structure if and only if it is $E$-cut for the reversed structure.
\end{proof}

\begin{ex} 
The spanning tree on the left of Figure \ref{fig:Jaeger-trees} is an $E$-cut Jaeger tree for the ribbon structure induced by the clockwise (negative) orientation of the plane, with base node $v_0$ with base edge $v_0e_2$.
\end{ex}

Note that by symmetry, an $E$-cut Jaeger tree with base node $b_0$ and base edge $b_0b_1$ is also a $V$-cut Jaeger tree with base node $b_0$, base edge $b_0b_1^-$ and the reversed ribbon structure. As the reversal of the ribbon structure is an involution, we deduce the following.

\begin{lemma}\label{c:set_of_emerald_trees_eq_violet}
The set of $E$-cut Jaeger trees with base node $b_0$ and base edge $b_0b_1$ is equal to the set of $V$-cut Jaeger trees with base node $b_0$, base edge $b_0b_1^-$, and the reversed ribbon structure.
\end{lemma}

\begin{definition}
For a $V$-cut Jaeger tree $T$, we call the tour of $T$ (with respect to the original ribbon structure) the \emph{violet tour of $T$}. We call the tour of $T$ with base point $b_0$, base edge $b_0b_1^-$ and the reversed ribbon structure the \emph{emerald tour of $T$}.
\end{definition}

By Lemma \ref{l:emerald_tree_is_also_violet}, for a $V$-cut Jaeger tree, in the emerald tour, each non-edge is cut at its emerald endpoint. The violet and the emerald tours of a Jaeger tree both induce orderings of the edges of $\bip \HH$, as follows. 

\begin{definition}\label{def:T-order}
For a $V$-cut Jaeger tree $T$, let the \emph{violet $T$-order of the edges of $\bip \HH$} be the order in which the edges of $\bip \HH$ appear in the violet $T$-tour with their violet endpoint as current node. Let us denote this ordering by $<_{T,V}$. That is, $ev <_{T,V} e'v'$ if $(v,ev)<_T (v', e'v')$, where $<_T$ is as in subsection \ref{subsec:Bernardi_for_graphs}.

Similarly, let the \emph{emerald $T$-order of the edges of $\bip \HH$} be the order in which the edges of $\bip \HH$ appear in the emerald $T$-tour with their emerald endpoint as current node. Let us denote this ordering by $<_{T,E}$; i.e., $ev <_{T,E} e'v'$ if $(e,ev)<_T (e', e'v')$, where now $<_T$ is defined by the reversed ribbon structure.
\end{definition}

\begin{definition}\label{def:induced_order}
Let $T\subset\bip\HH$ be a $V$-cut Jaeger tree. The \emph{order induced on $E$ by the emerald $T$-order} of the edges of $\bip \HH$, or simply the \emph{emerald $T$-order on $E$}, is by the smallest of the incident edges. That is, $e_1$ is less than $e_2$ if there is an edge, incident to the node $e_1$, that is smaller in the emerald $T$-order than any edge incident to $e_2$. There is a similar order of the nodes in $V$, induced by the violet $T$-order of the edges.
\end{definition}

\begin{ex}
For the Jaeger tree on the left side of Figure \ref{fig:Jaeger-trees}, the violet $T$-order is $v_0e_0<_{T,V}v_1e_1<_{T,V}v_1e_3<_{T,V}v_1e_0<_{T,V}v_0e_3<_{T,V}v_2e_1<_{T,V}v_2e_2<_{T,V}v_2e_3<_{T,V}v_0e_2$, while the emerald $T$-order is $e_2v_2<_{T,E}e_2v_0<_{T,E}e_3v_2<_{T,E}e_1v_1<_{T,E}e_1v_2<_{T,E}e_3v_1<_{T,E}e_3v_0<_{T,E}e_0v_1<_{T,E}e_0v_0$. The former induces the order $v_0<v_1<v_2$ of the violet nodes, and the latter induces the order $e_2<e_3<e_1<e_0$ of the emerald nodes.
\end{ex}

\begin{remark}\label{f:orders}
The previous two definitions are almost superfluous, in that they describe orderings that we have already seen --- we just cannot prove this yet. But once we know that our Jaeger tree is an outcome of the Bernardi process, in fact in two different senses by Lemma \ref{c:set_of_emerald_trees_eq_violet}, Lemma \ref{lem:walk_vs_tour} will imply that the orders of Definition \ref{def:T-order} are the same in which the edges of $\bip\HH$ become current in those processes.

More precisely, if a $V$-cut Jaeger tree $T$ is the outcome of the (ht:$V$, cut:$V$) or the (ht:$E$, cut:$V$) Bernardi process, then in the run producing $T$, the edges of $\bip \HH$ become current edges in the violet $T$-order. 
Also, if an $E$-cut Jaeger tree $T$ is the outcome of the (ht:$E$, cut:$E$) or the (ht:$V$, cut:$E$) Bernardi process, then in the run producing $T$, the edges of $\bip \HH$ become current edges in the emerald $T$-order.

Similarly, we will eventually see that the orders of Definition \ref{def:induced_order} are instances of \eqref{eq:hyperedge_order} in Definition \ref{def:hyperactivity}.
\end{remark}

\subsection{Fundamental cuts} 

Fundamental cuts of Jaeger trees have a nice interplay with the orderings of the edges induced by the trees. The following, rather intuitive lemma is proved in \cite{Bernardi_Tutte}. We specialize it to our case and translate it into our notation.

\begin{lemma}\label{l:parents_in_T-order} \cite[Lemma 5]{Bernardi_Tutte}
	Let $T$ be a spanning tree of $\bip \HH$ and $xy\in T$. Let $T_0$ and $T_1$ be the two subtrees of $T-xy$ so that $T_0$ contains the base node. If $(x,xy) <_T (y,xy)$, then $x$ is in the node set of $T_0$. Moreover, 
\[\{(z,zw)\mid (x,xy) <_T (z,zw) \leq_T (y,xy)\}=\{(z,zw)\mid z\text{ is in the node set of } T_1\}.\]
\end{lemma}

For a subgraph of $\bip\HH$, we let its \emph{base component} be the connected component that contains the base node. See Figure \ref{fig:fundamental_cut} for an example.

\begin{lemma}\label{l:char_Jaeger_cuts}
	Let $T$ be a $V$-cut Jaeger tree, $\varepsilon \in T$,
	and $\varepsilon_1, \varepsilon_2$ edges in the fundamental cut $C^*(T,\varepsilon)$. If $\varepsilon_1$ has its violet endpoint in the base component of $T-\varepsilon$ and $\varepsilon_2$ has its emerald endpoint in the base component of $T-\varepsilon$, then 
\begin{enumerate}[label=(\roman*)]
	\item \label{egyegy} $\varepsilon_1 <_{T,V} \varepsilon_2$, and
	\item \label{kettoketto} $\varepsilon_1 \leq_{T,V} \varepsilon$.
\end{enumerate}
\end{lemma}

In other words, the tour of $T$ first visits those edges of $C^*(T,\varepsilon)$ that have their violet endpoint in the base component of $T-\varepsilon$, the last one of which, and in that case the only one not to be cut, may be $\varepsilon$. If $\varepsilon$ has its emerald endpoint in the base component of $T-\varepsilon$, then it is the first such edge to be visited by the tour, but because it is traversed from the emerald direction, it (typically) does not become smallest among them with respect to $<_{T,V}$. (In fact, somewhat counterintuitively, it will be the largest.) In both cases, by traversing $\varepsilon$ the tour leaves the base component and visits, and cuts at their violet endpoints, the remaining edges of $C^*(T,\varepsilon)$. By the time the tour returns to the base component (by traversing $\varepsilon$ for the second time), all edges of $C^*(T,\varepsilon)$, save $\varepsilon$, have been cut.

\begin{figure}
	\begin{center}
		\begin{tikzpicture}[-,>=stealth',auto,scale=0.3,
		thick]
		\tikzstyle{c}=[{circle,draw,font=\sffamily\small}]
		\node[c,color=green,fill,label=left:{\small $b_1=e_0$}] (0) at (0, 0.1) {};
		\node[c,color=blue,fill,label=below:{\small $v_1$}] (1) at (3, 2) {};
		\node[c,color=blue,fill,label=left:{\small $b_0=v_0$}] (2) at (-3, 2) {};
		\node[c,color=green,fill,label=right:{\small $e_1$}] (3) at (3, 5) {};
		\node[c,color=green,fill, label=left:{\small $e_2$}] (4) at (-3, 5) {};
		\node[c,color=blue,fill, label=left:{\small $v_2$}] (5) at (0, 6.9) {};
		\node[c,color=green,fill, label=below:{\small $e_3$}] (6) at (0, 3.55) {};
		\path[every node/.style={font=\sffamily\small}, line width=0.2mm]
		(1) edge node {} (6);
		\path[every node/.style={font=\sffamily\small}, line width=0.8mm]
		(1) edge node {} (0)
		(2) edge node {} (4)
		(0) edge node {} (2)
		(5) edge node {} (3)
		(6) edge node {} (2);
		\path[every node/.style={font=\sffamily\small}, line width=0.8mm, color=red]
		(5) edge node {} (6);
		\path[every node/.style={font=\sffamily\small}, line width=0.2mm, color=red]
		(1) edge node {} (3)
		(4) edge node {} (5);
		\end{tikzpicture}
\end{center}
\caption{
With respect to the $V$-cut Jaeger tree $T$ of Figure \ref{fig:Jaeger-trees},
the red edges comprise the fundamental cut of $v_2e_3$. The base component of $T-v_2e_3$ is the subtree $\{\,v_0e_0, e_0v_1, v_0e_3, v_0e_2\,\}$.}
\label{fig:fundamental_cut}
\end{figure}

\begin{proof}[Proof of Lemma \ref{l:char_Jaeger_cuts}]
The subgraph $T-\varepsilon$ is a forest made up of the base component $T_0$ and another component $T_1$. Let $\varepsilon=xy$ and suppose that $(x,\varepsilon)<_T (y,\varepsilon)$. For $i=1,2$, let also $\varepsilon_i=v_ie_i\in C^*(T,\varepsilon)$, where $v_i\in V$ and $e_i\in E$. Since $\varepsilon_1$ has its violet endpoint $v_1$ in $T_0$, its emerald endpoint $e_1$ is from the node set of $T_1$. Hence Lemma \ref{l:parents_in_T-order} implies $(x,\varepsilon)<_T (e_1, \varepsilon_1) \leq_T (y,\varepsilon)$. As the ordering is induced by the violet tour of $T$, we either have $\varepsilon_1=\varepsilon$ (and then $x=v_1$) or the edge $\varepsilon_1$ is cut at its violet endpoint, i.e., $(v_1,\varepsilon_1)<_T (e_1,\varepsilon_1)$. Because $v_1$ is not in the node set of $T_1$, we must then have $(v_1,\varepsilon_1) \leq_T (x,\varepsilon)$. From this, \ref{kettoketto} immediately follows, for no matter if $x$ or $y$ is the violet endpoint of $\varepsilon$, we have $(v_1,\varepsilon_1) \leq_T (x,\varepsilon)<_T (y, \varepsilon)$.
	
Since $\varepsilon_2$ has its emerald endpoint $e_2$ in $T_0$, its violet endpoint $v_2$ is from the node set of $T_1$. Hence $(x,\varepsilon)<_T (v_2, \varepsilon_2) \leq_T (y,\varepsilon)$. We conclude that $(v_1,\varepsilon_1)<_T (v_2,\varepsilon_2)$, implying $\varepsilon_1 <_{T,V} \varepsilon_2$.
\end{proof}

\begin{lemma} \label{l:subs_trees_cuts}
If the (violet) tours of the $V$-cut Jaeger trees $T$ and $T'$ coincide until $\varepsilon$ becomes current edge, furthermore $\varepsilon\in T$ but $\varepsilon\notin T'$, then $\varepsilon$ has its violet endpoint in the base component of $T-\varepsilon$. Moreover, for any $\varepsilon'\in T'$ such that $\varepsilon'\in C^*(T,\varepsilon)$, the emerald endpoint of $\varepsilon'$ lies in the base component of $T-\varepsilon$.	
\end{lemma}

\begin{proof}
Up to the time when the two tours diverge, $\varepsilon$ has been neither cut (for otherwise it could not be an edge in $T$) nor traversed (because then it would be an edge in $T'$). So this is when the tour of $T'$ will cut it and since $T'$ is a $V$-cut Jaeger tree, the current node $v$ has to be violet. As $v$ is connected to $b_0$ in $T$ by a path not traversing $\varepsilon$, the first claim follows.

From part \ref{kettoketto} of Lemma \ref{l:char_Jaeger_cuts} we know that all the edges in $C^*(T,\varepsilon)-\varepsilon$ that have their violet endpoints in the base component had already been current edges, in conjunction with their violet endpoints, in the violet tour of $T$ before reaching $\varepsilon$, and they were cut. As the violet tours of $T$ and $T'$ coincide until reaching $\varepsilon$, these edges are not included in $T'$, either. Since we also have $\varepsilon\notin T'$, all the edges of $T'$ from $C^*(T,\varepsilon)$ have their emerald endpoints in the base component of $T-\varepsilon$.
\end{proof}

\subsection{All Jaeger trees are outcomes of the Bernardi process}
\label{ssec:Jaeger->Bernardi}

In this subsection we prove that each hypertree is realized by unique $V$-cut and $E$-cut Jaeger trees, implying that Jaeger trees are exactly the trees obtained as outcomes of the Bernardi process.

\begin{thm}
\label{t:hypertrees_unique_repr_by_J_trees}
Let $\HH=(V,E)$ be a connected hypergraph with a ribbon structure on $\bip\HH$, a base node, and a base edge fixed. For each hypertree $\mathbf f\in B_V$ there is exactly one $V$-cut Jaeger tree $T$ such that $\mathbf f=\mathbf f_V(T)$.
\end{thm}

\begin{proof} We proved in Proposition \ref{p:B-trees_are_J-trees} that for each hypertree $\mathbf f\in B_V$, the (ht:$V$, cut:$V$) Bernardi process produces a $V$-cut Jaeger tree realizing $\mathbf f$. Hence it is enough to show that for each hypertree, there is at most one suitable Jaeger tree. Suppose for a contradiction that for some hypertree $\mathbf f\in B_V$ there are two $V$-cut Jaeger trees $T\neq T'$ such that $\mathbf f=\mathbf f_V(T)=\mathbf f_V(T')$. 
	
Consider the (violet) tours of $T$ and $T'$, and suppose that the two tours first differ when an edge $ve$ is included into $T$ but not included into $T'$. As $T$ and $T'$ are $V$-cut Jaeger trees, this means that the current node at this time is $v$. Cf.\ Lemma \ref{l:subs_trees_cuts}, which also says that each edge in $T'\cap C^*(T,ve)$ has its emerald endpoint in the base component of $T-ve$. 

Let $V_0$ be the set of violet nodes and $E_0$ be the set of emerald nodes in the base component of $T-ve$. We will compute $\sum_{u\in V_0}\mathbf f(u)$ twice, using the realizations $T$ and $T'$, respectively.

As edges of $T$ between elements of $V_0\cup E_0$ form a tree and other than those, $T$ only has one edge, $ve$, incident to an element of $V_0$, we have $\sum_{u\in V_0}\mathbf f(u)=(|V_0\cup E_0|-1)+1-|V_0|=|E_0|$. 

On the other hand, as $T'$ is a tree, it can have at most $|V_0\cup E_0|-1$ edges between elements of $V_0\cup E_0$. All other edges of $T'$ incident to $V_0$ belong to $C^*(T,ve)$ but we have already seen that there are no such edges. Hence this time we obtain $\sum_{u\in V_0}\mathbf f(u)\le (|V_0\cup E_0|-1)-|V_0|=|E_0|-1$, which is a contradiction.
\end{proof}

\begin{coroll}
\label{cor:V_hyper_by_emerald_J_repr}
Under the same assumptions as in Theorem \ref{t:hypertrees_unique_repr_by_J_trees}, for each hypertree $\mathbf f\in B_V$ there is exactly one $E$-cut Jaeger tree $T$ such that $\mathbf f=\mathbf f_V(T)$. Similarly, hypertrees in $B_E$ have unique $E$-cut and $V$-cut Jaeger tree representatives.
\end{coroll}

\begin{proof}
Note that $\mathbf f$ does not depend on the ribbon structure. By Lemma \ref{l:emerald_tree_is_also_violet}, the trees we are considering are exactly the $V$-cut Jaeger tree representatives of $\mathbf f$ with respect to the reversed ribbon structure. According to Theorem \ref{t:hypertrees_unique_repr_by_J_trees}, there is a unique such tree. The second claim follows by interchanging the roles of colors.
\end{proof}

In Proposition \ref{prop:bernardi->jaeger} we saw that each Bernardi process of Section \ref{sec:Bernardi-process} provides a function from the set of hypertrees (on an arbitrarily fixed color class of $\bip\HH$) to the appropriate set of Jaeger trees. The map is an injection because the same tree cannot represent different hypertrees. By Theorem \ref{t:hypertrees_unique_repr_by_J_trees} or Corollary \ref{cor:V_hyper_by_emerald_J_repr}, the map must also be a surjection. Hence we arrive at the following two conclusions.

\begin{thm}\label{t:J-trees_are_B-trees}
For any connected bipartite graph (of color classes $E$ and $V$) with a fixed ribbon structure, base node, and base edge, the outcomes of the (ht:$E$, cut:$V$) Bernardi process, as well as the outcomes of the (ht:$V$, cut:$V$) Bernardi process, are exactly the $V$-cut Jaeger trees.
	
In the same way, the outcomes of the (ht:$E$, cut:$E$) Bernardi process, as well as the (ht:$V$, cut:$E$) Bernardi process, are exactly the $E$-cut Jaeger trees.
\end{thm}

\begin{coroll} 
\label{c:Jaeger_hypertree_uniquely_realized}
	For any connected bipartite graph with color classes $E$, $V$ and a fixed ribbon structure, base node, and base edge, the (say) $V$-cut Jaeger trees give a bijection between the sets of hypertrees $B_{E}$ and $B_{V}$, namely for each hypertree on $E$, there is exactly one $V$-cut Jaeger tree that realizes it and the same holds for each hypertree on $V$.
\end{coroll}

This bijection generalizes the bijection between spanning trees and break divisors of a graph found by Bernardi \cite{Bernardi_Tutte} and studied further by Baker and Wang \cite{Baker-Wang}.

In section \ref{sec:shelling} we will give a geometric interpretation of the above facts by proving that the simplices corresponding to Jaeger trees of $\bip\HH$ form a dissection of the root polytope $Q_{\bip\HH}$. In particular, the bijection of Corollary \ref{c:Jaeger_hypertree_uniquely_realized} turns out to be an instance of Theorem \ref{l:quasitriang_gives_bijection}.

\subsection{The four types of Bernardi processes}
\label{subs:four_Bernardi}

In our treatment of hypergraphs $\HH$ through their associated bipartite graphs $\bip\HH$, the roles of the emerald and violet color classes are inherently symmetric. In particular, there are only two different versions of the Bernardi process, as described in Section \ref{sec:Bernardi-process}. But if one is solely interested in graphs $G$, as opposed to hypergraphs, the construction of $\bip G$ becomes somewhat unnatural and the symmetry is easy to miss. In this subsection we formulate statements on four Bernardi processes: the two above and their transposes. Capitalizing on the structure of Jaeger trees, we deduce several relationships between them.

\begin{thm}
Run the (ht:$E$, cut:$V$) Bernardi process on the hypertree $\mathbf f$ (on $E$), and let the resulting spanning tree of $\bip \HH$ be $T$. Then by running the (ht:$V$, cut:$V$) Bernardi-process on the induced (on $V$) hypertree $\mathbf f_V(T)$, we again get $T$ as resulting spanning tree.
\end{thm}

\begin{proof}
Let the spanning tree produced by the (ht:$V$, cut:$V$) Bernardi-process on $\mathbf f_V(T)$ be $T'$. Then by Proposition \ref{prop:bernardi->jaeger}, $T$ and $T'$ are both $V$-cut Jaeger trees, and they both realize $\mathbf f_V(T)$ on $V$. Hence by Theorem \ref{t:hypertrees_unique_repr_by_J_trees}, we have $T'=T$.
\end{proof}

\begin{thm}\label{thm:inverses}
Run the (ht:$E$, cut:$V$) Bernardi process on the hypertree $\mathbf f$ (on $E$), and let the resulting spanning tree of $\bip \HH$ be $T$. Then by running the (ht:$E$, cut:$E$) Bernardi process on $\mathbf f$, with base node $b_0$, base edge $b_0b_1^-$, and the reversed ribbon structure, we again get $T$ as resulting spanning tree.
\end{thm}

\begin{proof}
Let $T'$ be the spanning tree produced by the (ht:$E$, cut:$E$) Bernardi process on $\mathbf f$ with base node $b_0$, base edge $b_0b_1^-$, and the reversed ribbon structure. Then $T'$ is an $E$-cut Jaeger tree for the reversed ribbon structure. Hence by Lemma \ref{l:emerald_tree_is_also_violet}, $T'$ is a $V$-cut Jaeger tree with base node $b_0$, base edge $b_0b_1$ and the original ribbon structure. Thus $T$ and $T'$ are both $V$-cut Jaeger trees, and they both realize $\mathbf f$ on $E$. By Corollary \ref{cor:V_hyper_by_emerald_J_repr}, this implies $T'=T$.
\end{proof}

\begin{thm}
Run the (ht:$E$, cut:$V$) Bernardi process on the hypertree $\mathbf f$ (on $E$), and let the resulting spanning tree of $\bip \HH$ be $T$. Then by running the (ht:$V$, cut:$E$) Bernardi process on $\mathbf f_V(T)$, with base node $b_0$, base edge $b_0b_1^-$ and the reversed ribbon structure, we again get $T$ as resulting spanning tree.
\end{thm}

\begin{proof}
Let $T'$ be the spanning tree produced by the (ht:$V$, cut:$E$) Bernardi process on $\mathbf f_V(T)$, with base node $b_0$, base edge $b_0b_1^-$, and the reversed ribbon structure. Then $T'$ is an $E$-cut Jaeger tree in that setup and by Lemma \ref{l:emerald_tree_is_also_violet}, $T'$ is a $V$-cut Jaeger tree with base node $b_0$, base edge $b_0b_1$ and the original ribbon structure. Hence $T$ and $T'$ are both $V$-cut Jaeger trees, and they both realize $\mathbf f_V(T)$ on $V$. Thus by Theorem \ref{t:hypertrees_unique_repr_by_J_trees}, we have $T'=T$.
\end{proof}

\begin{remark}\label{rem:Bernardi_altalanosit_BW}
The original Bernardi process for graphs is most easily identified with the (ht:$E$, cut:$V$) version. In \cite{Baker-Wang}, Baker and Wang define a right and a left inverse for the bijection between spanning trees and break divisors given by the Bernardi process on graphs. We note that the (ht:$V$, cut:$V$) Bernardi process is a generalization of their right inverse and the (ht:$V$, cut:$E$) Bernardi process with base node $b_0$, base edge $b_0b_1^-$, and the reversed ribbon structure is a generalization of their left inverse. (By Theorem \ref{thm:inverses}, these two mappings coincide.) Hence the theorems of this subsection extend the results of Baker and Wang. This hinges on the realization that both spanning trees and break divisors are special cases of hypertrees.
\end{remark}

\section{Shellability}
\label{sec:shelling}

We start this section by introducing a natural order on the set of Jaeger trees. As always, we assume that a ribbon structure, base node, and base edge are fixed for $\bip\HH$, where $\HH=(V,E)$ is a connected hypergraph. Then we actually have two sets of Jaeger trees in $\bip\HH$, the $E$-cut set and the $V$-cut set. We will order them both in two different ways. (The exact relationship between the two orders on the same set is somewhat unclear. In some sense they should be opposites but that is not literally true.)

\begin{definition}
\label{def:treeorder}
Let $T_1$ and $T_2$ be $V$-cut Jaeger trees. Consider the last time when their (violet) tours are identical. This happens when a violet node $v$ is current\footnote{It is easy to see that the current node $v$ has to be violet, cf.\ the first paragraph of the proof of Lemma \ref{l:subs_trees_cuts}, but in fact the definition can be made without this piece of information.} together with an edge $\varepsilon$ incident to $v$, where, say, $\varepsilon\in T_2$ but $\varepsilon\notin T_1$. In this case we put $T_1 <_V T_2$. This defines a total order on the set of $V$-cut Jaeger trees. Let us call this order the \emph{violet order of $V$-cut Jaeger trees}.

We also define the \emph{emerald order of $V$-cut Jaeger trees}. This time $T_1 <_E T_2$ if in their emerald tours (which use the reversed ribbon structure), the first edge that behaves differently is part of $T_2$ but not part of $T_1$. By symmetry of color classes, we also have \emph{violet and emerald orders of $E$-cut Jaeger trees}.
\end{definition}

Figure \ref{fig:order_of_Jaeger_trees} shows examples of these relations. Like before, the ribbon structure comes from the positive orientation of the plane, the base node is the lower left blue node, with base edge going to the right and downwards.

\begin{figure}[h]
	\begin{tikzpicture}[scale=.25]
	
	\begin{scope}[shift={(-10,0)}]
	\draw [dashed, thick] (18, 8.5) -- (14,6.5);
	\draw [ultra thick] (14, 6.5) -- (14,2);
	\draw [ultra thick] (18, 8.5) -- (22,6.5);
	\draw [ultra thick, green] (18, 8.5) -- (18,4);
	\draw [ultra thick] (14, 2) -- (18,4);
	\draw [ultra thick] (14, 2) -- (18,0);
	\draw [ultra thick] (18, 0) -- (22,2);
	\draw [dashed, thick] (22, 2) -- (18,4);
	\draw [dashed, thick, blue] (22, 2) -- (22,6.5);
	\draw [fill=green,green] (18, 0) circle [radius=0.6];
	\draw [fill=green,green] (18, 4) circle [radius=0.6];
	\draw [fill=green,green] (14, 6.5) circle [radius=0.6];
	\draw [fill=green,green] (22, 6.5) circle [radius=0.6];
	\draw [fill=blue,blue] (14, 2) circle [radius=0.6];
	\draw [fill=blue,blue] (22, 2) circle [radius=0.6];
	\draw [fill=blue,blue] (18, 8.5) circle [radius=0.6];
	
	\node at (12.8, 1) {{\small{$b_0$}}};
	\node at (16.5, -0.5) {{\small{$b_1$}}};
	\end{scope}
	
	\begin{scope}[shift={(10,0)}]
	\draw [dashed, thick] (18, 8.5) -- (14,6.5);
	\draw [ultra thick] (14, 6.5) -- (14,2);
	\draw [ultra thick] (18, 8.5) -- (22,6.5);
	\draw [dashed, thick, green] (18, 8.5) -- (18,4);
	\draw [ultra thick] (14, 2) -- (18,4);
	\draw [ultra thick] (14, 2) -- (18,0);
	\draw [ultra thick] (18, 0) -- (22,2);
	\draw [dashed, thick] (22, 2) -- (18,4);
	\draw [ultra thick, blue] (22, 2) -- (22,6.5);
	\draw [fill=green,green] (18, 0) circle [radius=0.6];
	\draw [fill=green,green] (18, 4) circle [radius=0.6];
	\draw [fill=green,green] (14, 6.5) circle [radius=0.6];
	\draw [fill=green,green] (22, 6.5) circle [radius=0.6];
	\draw [fill=blue,blue] (14, 2) circle [radius=0.6];
	\draw [fill=blue,blue] (22, 2) circle [radius=0.6];
	\draw [fill=blue,blue] (18, 8.5) circle [radius=0.6];
	
	\node at (12.8, 1) {{\small{$b_0$}}};
	\node at (16.5, -0.5) {{\small{$b_1$}}};
	\end{scope}
	\end{tikzpicture}
	\caption{ 
	The Jaeger tree on the left precedes the Jaeger tree on the right in the violet order, because of the blue edge. The Jaeger tree on the right precedes the Jaeger tree on the left in the emerald order because of the green edge.}
	\label{fig:order_of_Jaeger_trees}
\end{figure}

The goal of this section is to prove that the simplices corresponding to the, say, $V$-cut Jaeger trees form a dissection of the root polytope $Q_{\bip\HH}$, and moreover, the orders of Definition \ref{def:treeorder} translate to shelling orders for this dissection. This will be a crucial ingredient in the proof of Theorem \ref{thm:Bernardi-interior_well_def}, i.e., that our alternative definition for the interior polynomial agrees with the original one. The other key to the proof will be that the interior polynomial coincides with the $h$-vector of any shellable dissection of the root polytope.

The following notion is independent of ribbon structures and will greatly help to describe the terms in the $h$-vector. We invented it based on `external semi-activity,' which, in turn, we heard about from Alexander Postnikov.

\begin{definition}
Given a bipartite graph $G$, a total order on its edges, and a spanning tree $T$ of $G$, we say that an edge $\varepsilon\in T$ is \emph{internally semi-passive} in $T$, if $\varepsilon$ ``stands opposite'' to the smallest edge $\varepsilon'$ in the fundamental cut $C^*(T,\varepsilon)$, that is, $\varepsilon$ and $\varepsilon'$ have endpoints of different color in each component of $T-\varepsilon$. 
\end{definition}

See Figure \ref{fig:passzivitas_pelda} for an example. `Passive' here just stands for `not active.' In Tutte's original sense, `internally active' means `being smallest in the fundamental cut.' In `internally semi-active,' that condition is weakened to `standing parallel to the smallest element of the fundamental cut.'

\begin{figure}
\begin{center}
	\begin{tikzpicture}[-,>=stealth',auto,scale=0.4,
	thick]
	\tikzstyle{c}=[{circle,draw,font=\sffamily\small}]
	\node[c,color=green,fill] (0) at (0, 0.2) {};
	\node[c,color=blue,fill] (1) at (3, 2) {};
	\node[c,color=blue,fill] (2) at (-3, 2) {};
	\node[c,color=green,fill] (3) at (3, 5) {};
	\node[c,color=green,fill] (4) at (-3, 5) {};
	\node[c,color=blue,fill] (5) at (0, 6.8) {};
	\node[c,color=green,fill] (6) at (0, 3.6) {};
	\path[every node/.style={font=\sffamily\small}]
	(1) edge node {$0$} (0)
    (2) edge node {$6$} (4)
    (1) edge node {$2$} (6);
	\path[every node/.style={font=\sffamily\small}, line width=0.7mm]
	(0) edge node {$1$} (2)
	(1) edge node {$4$} (3)
	(5) edge [color=red] node {$5$} (3)
	(4) edge node {$7$} (5)
    (6) edge [color=red] node {$8$} (2)
	(5) edge node {$3$} (6);
	\end{tikzpicture}
	\end{center}
\caption{The internally semi-passive edges, for the spanning tree represented by thick lines and the given edge order (numbering), are shown in red.}
\label{fig:passzivitas_pelda}
\end{figure}

The next lemma plays an important role in the proof of shellability, as well as in relating the terms of the $h$-vector to the coefficients in the Bernardi-type definition of the interior polynomial. To that end, \ref{elso} and \ref{harmadik} of the equivalent conditions below are the most important. Note that they refer to orders defined using two different ribbon structures, a reverse pair. The condition \ref{otodik} is the easiest to check in practice. We will usually (for example, in Theorem \ref{thm:shelling}) refer to the set of edges characterized by the Lemma using the property \ref{masodik}.

\begin{lemma}
\label{l:activities_description}
    Let $T$ be a $V$-cut Jaeger tree, and $\varepsilon \in T$ be an edge. The following statements are equivalent.
\begin{enumerate}[label=(\roman*)]
\item \label{elso}
    $\varepsilon$ arises as a first difference between $T$ and some tree preceding $T$ in the violet ordering of Jaeger trees, i.e., there exists a $V$-cut Jaeger tree $T'$ such that $\varepsilon\not\in T'$ but the (violet) tours of $T$ and $T'$ coincide until reaching $\varepsilon$.
\item \label{masodik}
    $\varepsilon$ is internally semi-passive in the spanning tree $T$ with respect to the emerald $T$-order (Definition \ref{def:T-order}) of the edges of $\bip \HH$.
\item \label{harmadik}
    $\varepsilon=ve$ has its violet endpoint $v$ in the base component of $T-\varepsilon$ and $e$ is internally inactive for the hypertree $\mathbf f_E(T)$ with respect to the emerald $T$-order on $E$ (Definition \ref{def:induced_order}). 
\item \label{negyedik}
    $\varepsilon$ is not the largest element in $C^*(T,\varepsilon)$ according to the violet T-order of the edges of $\bip \HH$.
\item \label{otodik}
    $\varepsilon$ has its violet endpoint in the base component, and there exists an edge in $C^*(T,\varepsilon)$ with its emerald endpoint in the base component.
   \end{enumerate}
\end{lemma}

A concrete example, with two edges of the tree satisfying the equivalent conditions, is shown in Figure \ref{fig:activities_and_semiactivities}.

\begin{proof}
    \ref{elso} $\Rightarrow$ \ref{otodik} is straightforward from Lemma \ref{l:subs_trees_cuts}. Note that as $T'$ is also a spanning tree of $\bip \HH$, it includes some edge $\varepsilon'$ from the cut $C^*(T,\varepsilon)$, which then has the required property by the Lemma. 

\smallskip

\noindent 
\ref{otodik} $\Rightarrow$ \ref{elso}. Let $T_0$ and $T_1$ be the two subtrees of $T-\varepsilon$, with $T_0$ containing the base node. Let $G_0$ be the subgraph of $\bip\HH$ spanned by the node set of $T_0$, and $G_1$ be the subgraph spanned by the node set of $T_1$. We build up a $V$-cut Jaeger tree $T'$ together with its violet tour, such that $\varepsilon$ is the first difference between $T'$ and $T$. Follow the violet tour of $T$ until reaching $\varepsilon$, but at that moment, do not include $\varepsilon$ into $T'$. Instead, stay in $T_0$ and continue with the part of the tour of $T$ that would follow the traversal of $\varepsilon$ from the emerald direction. Stop at the first moment when an edge $\varepsilon' \in C^*(T,\varepsilon)$, with its emerald endpoint $e'$ in $T_0$, becomes current edge. (By Lemma \ref{l:char_Jaeger_cuts}, $\varepsilon'$ is actually the first edge of $C^*(T,\varepsilon)$ that becomes current after $\varepsilon$. By the same Lemma, and the assumption \ref{otodik}, the existence of $\varepsilon'$ is guaranteed.) Let $\varepsilon'=e'v'$ where $v'$ is in the node set of $G_1$. If $G_1$ is not just a point, then take the first edge $v'e''\in G_1$, after $v'e'$, in the cyclic order at $v'$. Take an arbitrary $V$-cut Jaeger tree $T'_1$ of $G_1$ with base point $v'$ and base edge $v'e''$. (Such a Jaeger tree can be constructed by running, say, the (ht:$V$, cut:$V$) Bernardi process on an arbitrary hypertree of $G_1$.)

We claim that $T'=T_0\cup\{\varepsilon'\}\cup T'_1$ is a $V$-cut Jaeger tree. Once we prove this, it becomes immediate from the construction that the first difference in the violet tours of $T$ and $T'$ is that $\varepsilon$ is included into $T$ but not included into $T'$.

Until reaching $\varepsilon$, the violet tours of $T$ and $T'$ coincide. Then $\varepsilon$ is cut at its violet endpoint in the tour of $T'$. Next we continue the traversal of $T_0$ until we arrive at $\varepsilon'$. During this time, any edges that we cut are edges of $G_0$ that are cut, at the same endpoint, in the tour of $T$ as well. As $\varepsilon'$ is traversed, the tour of $T'$ thus far has not cut any edge at its emerald endpoint. If there are any edges between $v'e'$ and $v'e''$ at $v'$, then they get cut at their violet endpoints. Next we start the traversal of $T'_1$, which is a $V$-cut Jaeger tree in $G_1$. Compared to the tour of $T'_1$ with regard to $G_1$, the only difference in the tour of $T'$ is that any edges in $C^*(T,\varepsilon)$ that we encounter have to be skipped. But by Lemma \ref{l:char_Jaeger_cuts}, all the edges from $C^*(T,\varepsilon)$ that have their emerald endpoint in $G_1$ have already been cut, hence none of them gets cut at its emerald endpoint. Then when we arrive back at $e'$ after traversing $\varepsilon'$ for the second time, the set of node-edge pairs that have not been current is the same as in the violet tour of $T$ after $(e',e'v')$ serves as current pair. Moreover, the orders in which these remaining pairs become current are the same, too. Hence during the remainder of the tour of $T'$, each edge that is cut is still cut at its violet endpoint.

\smallskip

\noindent 
\ref{otodik} $\Rightarrow$ \ref{masodik}. The tree $T$ is an $E$-cut Jaeger tree for the reversed ribbon structure by Lemma \ref{l:emerald_tree_is_also_violet}. Applying Lemma \ref{l:char_Jaeger_cuts} with ``emerald'' instead of ``violet,'' we see that those edges in $C^*(T,\varepsilon)$ that have their emerald endpoint in the base component all precede, in the emerald $T$-order, those which have their violet endpoint there. Since now there exists an edge in $C^*(T,\varepsilon)$ that has its emerald endpoint in the base component, necessarily the smallest edge $\delta$ of $C^*(T,\varepsilon)$, according to the emerald $T$-order, is of this kind. As $\varepsilon$ has its violet endpoint in the base component, it stands opposite to $\delta$, which means that $\varepsilon$ is  internally semi-passive.

\smallskip

\noindent 
\ref{masodik} $\Rightarrow$ \ref{harmadik}. First we prove that if $\varepsilon=ve$ is internally semi-passive in $T$ with respect to the emerald $T$-order of the edges of $\bip \HH$, then $\varepsilon$ has its violet endpoint $v$ in the base component. The condition means that the smallest edge $v'e'$ in $C^*(T,\varepsilon)$, according to the emerald $T$-order, stands opposite to $ve$. By Lemma \ref{l:char_Jaeger_cuts}, among two edges of $C^*(T,\varepsilon)$ standing opposite to each other, the smaller one according to the emerald $T$-order has its emerald endpoint in the base component. Hence $\varepsilon$ has its violet endpoint in the base component.

Now we show that $e$ is internally inactive in $\mathbf f_E(T)$ with respect to the emerald $T$-order. Since $v'e'\in C^*(T,\varepsilon)$, replacing $ve$ with $v'e'$ in $T$ yields another spanning tree $T'$ of $\bip \HH$. In particular $\mathbf f_E(T')$ is a hypertree. As $\mathbf f_E(T')$ is obtained from $\mathbf f_E(T)$ by a transfer of valence from $e$ to $e'$ (assuming that $e'\in E$ and $v'\in V$), it is enough to show that $e'$ precedes $e$ in the emerald $T$-order. We already know that $ve$ has its violet endpoint in the base component, implying that $e$ is reached through $ve$ in the emerald tour of $T$. From Lemma \ref{l:char_Jaeger_cuts} we also know that $v'e'$ becomes current edge, with $e'$ as current node, earlier during the emerald tour of $T$ than the first traversal of $ev$. Hence any edge incident to $e$ comes later in the emerald $T$-order than $v'e'$. Thus, $e$ comes later in the emerald $T$-order than $e'$. 

\smallskip

\noindent 
\ref{harmadik} $\Rightarrow$ \ref{otodik}. As both \ref{harmadik} and \ref{otodik} claim that $\varepsilon=ve$ has its violet endpoint $v$ in the base component, it is enough to prove that if $e$ is internally inactive in $\mathbf f_E(T)$ with respect to the emerald $T$-order, then there is an edge in $C^*(T,\varepsilon)$ that has its emerald endpoint in the base component.
    
Suppose for contradiction that there is no edge in $C^*(T,\varepsilon)$ with its emerald endpoint in the base component $T_0$ of $T-\varepsilon$. Let $V_0$ be the set of violet nodes of $T_0$, and let $E_0$ be the set of its emerald nodes. Our assumption implies that for each $x\in E_0$, the set of edges of $T$ incident to $x$ equals the set of edges of $T_0$ incident to $x$. Moreover, all edges of $\bip\HH$ incident to $x$ have their violet endpoint in $V_0$.

Since $T_0$ is a tree spanning $V_0\cup E_0$, we have $\sum_{x\in E_0}\mathbf f_E(T)(x)=|V_0|+|E_0|-1-|E_0|=|V_0|-1$. Recall that for any hypertree $\mathbf f$ on $E$, we have $\sum_{x\in E_0}\mathbf f(x)\leq |V_0|-1$ \cite[Theorem 3.4]{hiperTutte}. In other words, with respect to $\mathbf f_E(T)$, the set $E_0$ is `tight,' meaning that no element of it may receive a transfer of valence from outside of $E_0$. Now this is a contradiction because all the emerald nodes that are smaller than $e$ in the emerald $T$-order lie in $E_0$, which blocks $e$ from being internally inactive. 

\smallskip

\noindent \ref{negyedik} $\Rightarrow$ \ref{otodik}. As discussed before the proof of Lemma \ref{l:char_Jaeger_cuts}, if $\varepsilon$ has its emerald endpoint in the base component, then it is the largest element, with respect to the violet $T$-order, of its fundamental cut. Indeed, such an edge only becomes current, in conjunction with its violet endpoint, at the time of its second traversal. This proves the first assertion of  \ref{otodik}; as to the second, let $\varepsilon'$ be an edge of $C^*(T,\varepsilon)$ such that $\varepsilon <_{T,V}\varepsilon'$. Then by part \ref{kettoketto} of Lemma \ref{l:char_Jaeger_cuts}, $\varepsilon'$ must have its emerald endpoint in the base component, for otherwise it would precede $\varepsilon$ in the violet $T$-order.

\smallskip

\noindent 
\ref{otodik} $\Rightarrow$ \ref{negyedik} also follows directly from Lemma \ref{l:char_Jaeger_cuts}. Let $\varepsilon'$ be an edge in $C^*(T,\varepsilon)$ having its emerald endpoint in the base component. Since $\varepsilon$ has its violet endpoint in the base component, by part \ref{egyegy} of Lemma \ref{l:char_Jaeger_cuts} we have $\varepsilon \leq_{T,V} \varepsilon'$, preventing $\varepsilon$ from being the largest edge of $C^*(T,\varepsilon)$ with respect to the violet $T$-order.
\end{proof}

We can further deduce the following characterization of internally inactive nodes. It does not matter whether the Jaeger tree of the next two corollaries is $V$-cut or $E$-cut since we only need the fact that it has an emerald tour.

\begin{figure}[h]
	\begin{tikzpicture}[scale=.25]
	
	\draw [ultra thick] (18, 8.5) -- (14,6.5);
	\draw [dashed, thick] (14, 6.5) -- (14,2);
	\draw [ultra thick, red] (18, 8.5) -- (22,6.5);
	\draw [ultra thick] (18, 8.5) -- (18,4);
	\draw [ultra thick, red] (14, 2) -- (18,4);
	\draw [ultra thick] (14, 2) -- (18,0);
	\draw [dashed, thick] (18, 0) -- (22,2);
	\draw [dashed, thick] (22, 2) -- (18,4);
	\draw [ultra thick] (22, 2) -- (22,6.5);
	\draw [fill=green,green] (18, 0) circle [radius=0.6];
	\draw [fill=green,green] (18, 4) circle [radius=0.6];
	\draw [fill=green,green] (14, 6.5) circle [radius=0.6];
	\draw [fill=green,green] (22, 6.5) circle [radius=0.6];
	\draw [fill=blue,blue] (14, 2) circle [radius=0.6];
	\draw [fill=blue,blue] (22, 2) circle [radius=0.6];
	\draw [fill=blue,blue] (18, 8.5) circle [radius=0.6];
	
	\draw (18, 4) circle [radius=1];
	\draw (22, 6.5) circle [radius=1];
	
	\node at (12.8, 1) {{\small{$b_0$}}};
	\node at (16.5, -0.5) {{\small{$b_1$}}};
	\end{tikzpicture}
	\caption{
	An $E$-cut Jaeger tree $T$ with circled internally inactive emerald nodes with respect to the emerald $T$-order on $E$. The internally semi-passive edges with respect to the emerald $T$-order of the edges are shown in red. (The ribbon structure once again comes from the counterclockwise orientation of the plane.)}
	\label{fig:activities_and_semiactivities}
\end{figure}

\begin{coroll}\label{cor:passziv_csucshoz_passziv_el}
Let $T$ be a Jaeger tree in $\bip\HH$. An emerald node $e\in E$ is internally inactive with respect to the emerald $T$-order on $E$ if and only if there is an internally semi-passive edge (with respect to the emerald $T$-order of the edges of $\bip\HH$) incident to it. Moreover, in this case there is a unique such edge, namely the edge through which $e$ is reached in the emerald tour of $T$ --- that is, the first edge of the unique path in $T$ from $e$ to the base node $b_0$.
\end{coroll}

\begin{proof}
If $e=b_0$, then $e$ is the smallest element in the emerald $T$-order of $E$ and hence it cannot be inactive. Also in this case, edges incident to $b_0$ have their emerald endpoint in the base component of their fundamental cut, which prevents them from being internally semi-passive by Lemma \ref{l:activities_description}.

If $e\neq b_0$, then there is a unique edge $\varepsilon=ve$ such that $e$ is reached through $\varepsilon$ in the emerald $T$-tour. This is the only edge incident to $e$ that has its violet endpoint in the base component of its own fundamental cut --- that is, $\varepsilon$ is the only edge incident to $e$ that may be internally semi-passive. Now by Lemma \ref{l:activities_description}, $\varepsilon$ is internally semi-passive with respect to the emerald $T$-order if and only if $e$ is internally inactive with respect to the emerald $T$-order on $E$.
\end{proof}

Now we are in a position to explain the connection between internal activities of hypertrees and internal semi-activities of Jaeger trees.

\begin{coroll}\label{cor:passziv_szemipassziv}
For each Jaeger tree $T$, the number of its internally semi-passive edges with respect to the emerald $T$-order (of the edges of $\bip\HH$) equals the number of internally inactive emerald nodes for $\mathbf f_E(T)$ with respect to the emerald $T$-order (of $E$).
\end{coroll}

Next we turn to proving that simplices corresponding to Jaeger trees form a shellable dissection. First we show that they form a dissection. 

\begin{thm}\label{t:J-trees_form_quasitr}
For any ribbon structure (and base node and base edge) on the connected bipartite graph $G$, of color classes $E$ and $V$, the simplices in the root polytope $Q_G$ corresponding to the $V$-cut Jaeger trees form a dissection of $Q_G$.
\end{thm}

It turns out that we have already done the hard part of the proof and only the following, relatively easy part remains.

\begin{lemma}\label{lem:disjoint}
	For any two $V$-cut Jaeger trees, their corresponding maximal simplices in $Q_G$ have disjoint interiors.
\end{lemma}

\begin{proof}
	Take two $V$-cut Jaeger trees $T_1$ and $T_2$, and suppose that their tours coincide until reaching $\varepsilon$, where $\varepsilon\notin T_1$ and $\varepsilon\in T_2$. It suffices to show that the maximal simplices that correspond to the pair of trees are separated by a hyperplane. Such hyperplanes may be constructed using cuts in $G$. This is described in \cite[Section 3]{KP_Ehrhart} for `directed' cuts, when the result is a supporting hyperplane of $Q_G$; here we will work with undirected cuts.
	
	Let $E_1\sqcup E_2$ be a partition of $E$ and $V_1\sqcup V_2$ a partition of $V$. Let us associate the real number $-1$ to those basis vectors of $\R^E\oplus\R^V$ that correspond to elements of $E_1\cup V_2$. Similarly, we associate $1$ to elements of $E_2\cup V_1$. This has a unique linear extension $\kappa\colon\R^E\oplus\R^V\to\R$. The value of $\kappa$ at the vertices of $Q_G$ (described in terms of the corresponding edges of $G$) is then
	\begin{itemize}
		\item $0$ for edges outside the cut defined by our partitions (when we split $E\cup V$ into $E_1\cup V_1$ and $E_2\cup V_2$), i.e., for edges between $E_1$ and $V_1$ or between $E_2$ and $V_2$,
		\item $-2$ for edges of the cut that are between $E_1$ and $V_2$, and
		\item $2$ for edges of the cut between $E_2$ and $V_1$.
	\end{itemize}
	Thus the kernel of $\kappa$ is a hyperplane that contains all vertices of $Q_G$ except for the ones that correspond to edges of the cut; the remaining vertices fall on the two sides of the hyperplane depending on whether their emerald or violet endpoint lies on an arbitrarily fixed side of the cut.
	
	Now to construct the required hyperplane, we consider the cut $C^*(T_2,\varepsilon)$. More precisely, let $E_1\cup V_1$ be the node set of the connected component of $T_2-\varepsilon$ which contains the violet endpoint $v$ of $\varepsilon$. Let $E_2\cup V_2$ be the node set of the other component. Then the corresponding functional $\kappa$ takes only non-negative values at the vertices of the maximal simplex $Q_{T_2}$: the value is $2$ for (the vertex corresponding to) $\varepsilon$ and $0$ for all other edges of $T_2$, i.e., vertices of $Q_{T_2}$. 
	
	We claim that the opposite is true for $T_1$: at the vertices of $Q_G$ corresponding to those edges, all values of $\kappa$ end up non-positive. To show this, we need to ascertain that all edges of $C^*(T_2,\varepsilon)\cap T_1$ are connecting $E_1$ and $V_2$. But that is exactly the statement of Lemma \ref{l:subs_trees_cuts}. 
\end{proof}

\begin{proof}[Proof of Theorem \ref{t:J-trees_form_quasitr}]
    By Proposition \ref{p:B-trees_are_J-trees}, each hypertree is realized by at least one $V$-cut Jaeger tree and hence the number of $V$-cut Jaeger trees is at least the number of hypertrees in $B_V$. (In fact in subsection \ref{ssec:Jaeger->Bernardi} we have showed that these numbers are equal, but the present proof does not rely on that.) By Lemma \ref{lem:disjoint}, the (maximal) simplices in $Q_G$ corresponding to the $V$-cut Jaeger trees have disjoint interiors. Moreover, each maximal simplex in $Q_G$ has the same volume \cite[Section 12]{alex} and the volume of $Q_G$ itself is the number of hypertrees times this common volume. Hence the simplices corresponding to $V$-cut Jaeger trees fill $Q_G$, in other words, they form a dissection. 
\end{proof}

This reasoning also implies that each hypertree is realized by at most one Jaeger tree, which we have already proven in Theorem \ref{t:hypertrees_unique_repr_by_J_trees}. Having a dissection allows us to use Theorem \ref{l:quasitriang_gives_bijection} and to thus give a new proof of Corollary \ref{c:Jaeger_hypertree_uniquely_realized}. But to prove Theorem \ref{thm:Bernardi-interior_well_def}, we need the shellability of the dissection, too.

\begin{thm}\label{thm:shelling}
The violet ordering of the $V$-cut Jaeger trees induces a shelling order of the dissection given in Theorem \ref{t:J-trees_form_quasitr}. For each $V$-cut Jaeger tree $T$, the number of facets of the corresponding simplex $Q_{T}$, that lie in the union of previous simplices of the shelling, equals the number of internally semi-passive edges in $T$ with respect to the emerald $T$-order of the edges.
\end{thm}

In other (more precise) words, the edges of $T$ described in five equivalent ways in Lemma \ref{l:activities_description} correspond exactly to those vertices of $Q_T$ whose opposite facets make up the intersection of $Q_T$ with the union of the previous simplices.

\begin{proof}
Let us fix a $V$-cut Jaeger tree $T$. We have to show that for an edge $\varepsilon\in T$, if $\varepsilon$ is internally semi-passive in $T$ with respect to the emerald $T$-order, then the facet $Q_{T-\varepsilon}$ of the simplex $Q_T$ lies in the union of the simplices corresponding to $V$-cut Jaeger trees preceding $T$ in the violet order. We also need to prove that if $\varepsilon$ is internally semi-active with respect to the emerald $T$-order, then the interior of $Q_{T-\varepsilon}$ is disjoint from $\bigcup_{T'<_V T}Q_{T'}$.

Take an edge $\varepsilon=ev \in T$ (where $e\in E$ and $v\in V$) such that $\varepsilon$ is internally semi-passive in $T$ with respect to the emerald $T$-order. In the same way as in the proof of Lemma \ref{lem:disjoint}, we consider the fundamental cut and the linear functional defined by $T$ and $\varepsilon$. Let us partition the edges in $C^*(T,\varepsilon)$ into the sets $P$ and $N$, where elements of $P$ have their violet endpoint in the root component $T_0$ of $T-\varepsilon$ and elements of $N$ have their emerald endpoint there. (This is the same partition as the one induced by the two sides, positive and negative, of the hyperplane that corresponds to the cut.) By Lemma \ref{l:activities_description}, condition \ref{otodik}, we have $\varepsilon\in P$ and $N\neq \varnothing$. 

Let us consider the edge $\varepsilon'\in N$ as in the part \ref{otodik} $\Rightarrow$ \ref{elso} of the proof of Lemma \ref{l:activities_description}. Namely, $\varepsilon'$ is the first edge of $C^*(T,\varepsilon)$ to become current in the tour of $T$ after the second traversal of $\varepsilon$. This time let us call its emerald endpoint $y\in T_0$. Let, as usual, $T_1$ be the complement of $T_0$ in $T-\varepsilon$ and let $G_1$ be the subgraph spanned by the node set of $T_1$. Let also $x\in G_1$ be the violet endpoint of $\varepsilon'$. Let us denote the set of $V$-cut Jaeger trees of $G_1$ with base node $x$ and base edge $xy'$ by $\mathcal{T}_1$, where $xy'$ is the first edge following $xy$ in the cyclic order at $x$ such that $xy'\in G_1$. As $G_1$ is connected and contains at least two nodes, $e$ and $x$, such an edge exists. Finally, consider the following set of trees: 
\begin{equation}
\label{eq:sokfa}
\mathcal{T}=\{\,T_0\cup \varepsilon' \cup T'_1\mid T'_1 \in \mathcal{T}_1\,\}.
\end{equation}
From the part \ref{otodik} $\Rightarrow$ \ref{elso} of the proof of Lemma \ref{l:activities_description}, we know that all of these are $V$-cut Jaeger trees preceding $T$ in the violet order of Jaeger trees. We claim that 
\begin{equation*}
Q_{T-\varepsilon}\subset \bigcup_{T'\in \mathcal{T}}Q_{T'}.
\end{equation*}

Take a point $\mathbf p$ from $Q_{T - \varepsilon}$. Then 
\[\mathbf p=\sum_{e'v'\in T - \varepsilon}\lambda_{e'v'} (\mathbf{e'}+\mathbf{v'})=\sum_{e'v'\in T_0}\lambda_{e'v'} (\mathbf{e'}+\mathbf{v'}) + \sum_{e'v'\in T_1}\lambda_{e'v'} (\mathbf{e'}+\mathbf{v'}),\]
where $\lambda_{e'v'}\geq 0$ and $\sum_{e'v'\in T - \varepsilon}\lambda_{e'v'}=1$. Let $\lambda_0=\sum_{e'v'\in T_0}\lambda_{e'v'}$ and $\lambda_1=\sum_{e'v'\in T_1}\lambda_{e'v'}$. If $\lambda_0=0$ then $\mathbf p\in Q_{G_1}$, which by Theorem \ref{t:J-trees_form_quasitr} is dissected by the faces $Q_{T_1'}$ of the simplices $Q_{T'}$, for $T'\in\mathcal T$ as in \eqref{eq:sokfa}, making $\mathbf p\in \bigcup_{T'\in \mathcal{T}}Q_{T'}$ obvious. If $\lambda_1=0$, then $\mathbf p\in Q_{T_0}$, which in turn is part of all the $Q_{T'}$. Otherwise, write
\[\mathbf p=\lambda_0\cdot\sum_{e'v'\in T_0}\frac{\lambda_{e'v'}}{\lambda_0} (\mathbf{e'}+\mathbf{v'}) + \lambda_1\cdot \sum_{e'v'\in T_1}\frac{\lambda_{e'v'}}{\lambda_1} (\mathbf{e'}+\mathbf{v'})\]
and let $\mathbf p_0=\sum_{e'v'\in T_0}\frac{\lambda_{e'v'}}{\lambda_0} (\mathbf{e'}+\mathbf{v'})$ and $\mathbf p_1=\sum_{e'v'\in T_1}\frac{\lambda_{e'v'}}{\lambda_1} (\mathbf{e'}+\mathbf{v'})$. That is, $\mathbf p=\lambda_0\mathbf p_0 + \lambda_1\mathbf p_1$, where $\lambda_0,\lambda_1>0$ and $\lambda_0 + \lambda_1=1$. Furthermore, we have $\mathbf p_1\in Q_{T_1}\subset Q_{G_1}$ and $\mathbf p_0\in Q_{T_0}$. Combining this with the convexity of $Q_{G_1}=\bigcup_{T'_1\in \mathcal{T}_1}Q_{T'_1}$, we obtain
\begin{multline*}
\bigcup_{T'\in \mathcal{T}}Q_{T'}\supset \bigcup_{T'\in \mathcal{T}}Q_{T'-\varepsilon'}=\bigcup_{T'_1\in \mathcal{T}_1} \conv(Q_{T_0}\cup Q_{T'_1})=\conv\left(Q_{T_0}\cup\bigcup_{T'_1\in \mathcal{T}_1}Q_{T'_1}\right)\\
=\conv(Q_{T_0} \cup Q_{G_1})\ni \mathbf p.
\end{multline*}

Now if $\varepsilon$ is internally semi-active with respect to the emerald $T$-order, then by Lemma \ref{l:activities_description}, condition \ref{elso}, $\varepsilon$ cannot be obtained as the first difference between $T$ and some Jaeger tree preceding $T$ in the violet ordering of Jaeger trees. Moreover, by the proof of Lemma \ref{lem:disjoint}, for any Jaeger tree $T'$ preceding $T$ in the violet ordering of Jaeger trees, $Q_{T}$ is separated from $Q_{T'}$ by the hyperplane containing the facet $Q_{T - \varepsilon''}$ where $\varepsilon''$ is the first difference between $T$ and $T'$. Therefore the interior of $Q_{T - \varepsilon}$ is indeed disjoint from $Q_{T'}$ and hence from $\bigcup_{T'<_V T}Q_{T'}$, too. 
\end{proof}

It is finally time to summarize our findings and prove our main theorem.

\begin{proof}[Proof of Theorem \ref{thm:Bernardi-interior_well_def}]
By Theorem \ref{t:interior_poly_ehrhart_h-vector} and the discussion preceding it, the coefficients of the interior polynomial $I_\HH$ are the entries in the $h$-vector of any shellable dissection (together with any shelling order) of $Q_{\bip \HH}$. Hence it suffices to show that there exists a shellable dissection of $Q_{\bip \HH}$, with a shelling order yielding the $h$-vector $(a_0,a_1,\dots )$, so that $a_i$ is equal to the number of hypertrees of $\HH$ with internal embedding inactivity $i$ (defined using our fixed ribbon structure and the (ht:$E$, cut:$E$) Bernardi process).

We choose the dissection of $Q_{\bip \HH}$ by $E$-cut Jaeger trees. By Lemma \ref{c:set_of_emerald_trees_eq_violet}, these are the same as the $V$-cut Jaeger trees of the \emph{reversed} ribbon structure (with base node $b_0$ and base edge $b_0b_1^-$). The \emph{violet} order of Jaeger trees will be our shelling order, as in Theorem \ref{thm:shelling}.

By Remark \ref{f:orders} (whose statements we already \emph{can} justify because we have Theorem \ref{t:J-trees_are_B-trees} in hand) and Corollary \ref{cor:passziv_szemipassziv}, a hypertree on $E$ has internal embedding inactivity $i$ with respect to the (ht:$E$, cut:$E$) Bernardi process if and only if the $E$-cut Jaeger tree $T$ realizing it (i.e., the outcome of the process) has precisely $i$ internally semi-passive edges with respect to the emerald $T$-order. Now the statement follows from the second sentence in Theorem \ref{thm:shelling}.
\end{proof}

\section{Examples}
\label{sec:examples}

In this section we recall two ways of triangulating some root polytopes and show that both are special cases of the dissections of Section \ref{sec:Jaeger_trees}. But before that, let us work out two concrete examples where the dissection fails to be a triangulation.

\begin{figure}[h]
\begin{tikzpicture}[scale=.25]

\begin{scope}[shift={(-2,0)}]

\path [fill=lightgray] (0,4.25) to [out=225,in=120] (-1,1.5) to [out=-60,in=200] (2,1) to [out=110,in=-110] (2,3) to [out=160,in=-45] (0,4.25);
\path [fill=lightgray] (6,1) to [out=-20,in=-120] (9,1.5) to [out=60,in=-45] (8,4.25) to [out=225,in=20] (6,3) to [out=-70,in=70] (6,1);
\path [fill=lightgray] (6,7.5) to [out=110,in=0] (4,9.7) to [out=180,in=70] (2,7.5) to [out=-20,in=135] (4,6.25) to [out=45,in=200] (6,7.5);
\path [fill=gray] (2,3) to [out=70,in=225] (4,6.25) to [out=-45,in=110] (6,3) to [out=200,in=-20] (2,3);
\path [fill=gray] (6,7.5) to [out=20,in=120] (9,7) to [out=-60,in=45] (8,4.25) to [out=135,in=-70] (6,7.5);
\path [fill=gray] (0,4.25) to [out=135,in=-120] (-1,7) to [out=60,in=160] (2,7.5) to [out=-110,in=45] (0,4.25);
\path [fill=gray] (6,1) to [out=-110,in=0] (4,-1.2) to [out=180,in=-70] (2,1) to [out=20,in=160] (6,1);

\draw [ultra thick, red] (0,2) -- (4,0);
\draw [ultra thick, red] (4,0) -- (8,2);
\draw [ultra thick, red] (8,2) -- (8,6.5);
\draw [ultra thick, red] (8,6.5) -- (4,8.5);
\draw [ultra thick, red] (4,8.5) -- (0,6.5);
\draw [ultra thick, red] (0,6.5) -- (0,2);
\draw [ultra thick, red] (0,2) -- (4,4);
\draw [ultra thick, red] (4,4) -- (4,8.5);
\draw [ultra thick, red] (4,4) -- (8,2);

\draw [fill=green,green] (4, 0) circle [radius=0.4];
\draw [fill=green,green] (4, 4) circle [radius=0.4];
\draw [fill=green,green] (0, 6.5) circle [radius=0.4];
\draw [fill=green,green] (8, 6.5) circle [radius=0.4];
\draw [fill=blue,blue] (0, 2) circle [radius=0.4];
\draw [fill=blue,blue] (8, 2) circle [radius=0.4];
\draw [fill=blue,blue] (4, 8.5) circle [radius=0.4];
\draw [fill] (-1,1.5) circle [radius=.3];

\draw [ultra thick] (-.15,4.1) to [out=225,in=120] (-1,1.5) to [out=-60,in=200] (2,1) to [out=20,in=160] (5.8,1.1);
\draw [ultra thick] (6.2,.9) to [out=-20,in=-120] (9,1.5) to [out=60,in=-45] (8,4.25) to [out=135,in=-70] (6.1,7.3);
\draw [ultra thick] (5.9,7.7) to [out=110,in=0] (4,9.7) to [out=180,in=70] (2,7.5) to [out=-110,in=45] (.15,4.4);
\draw [ultra thick] (1.9,1.2) to [out=110,in=-110] (2,3) to [out=70,in=225] (3.85,6.1);
\draw [ultra thick] (4.15,6.4) to [out=45,in=200] (6,7.5) to [out=20,in=120] (9,7) to [out=-60,in=45] (8.15,4.4);
\draw [ultra thick] (7.85,4.1) to [out=225,in=20] (6,3) to [out=200,in=-20] (2.2,2.9);
\draw [ultra thick] (1.8,3.1) to [out=160,in=-45] (0,4.25) to [out=135,in=-120] (-1,7) to [out=60,in=160] (1.8,7.6);
\draw [ultra thick] (2.2,7.4) to [out=-20,in=135] (4,6.25) to [out=-45,in=110] (5.9,3.2);
\draw [ultra thick] (6.1,2.8) to [out=-70,in=70] (6,1) to [out=-110,in=0] (4,-1.2) to [out=180,in=-70] (2.1,.8);

\draw [->,thick] (-2,2) arc [radius=2,start angle=180,end angle=240];

\end{scope}

\begin{scope}[shift={(11,0)}]

\draw [ultra thick, red] (4,0) -- (8,2);
\draw [ultra thick, red] (8,2) -- (8,6.5);
\draw [ultra thick, red] (8,6.5) -- (4,8.5);
\draw [ultra thick, red] (4,8.5) -- (0,6.5);
\draw [ultra thick, red] (0,6.5) -- (0,2);
\draw [ultra thick, red] (4,4) -- (8,2);

\draw [fill=green,green] (4, 0) circle [radius=0.4];
\draw [fill=green,green] (4, 4) circle [radius=0.4];
\draw [fill=green,green] (0, 6.5) circle [radius=0.4];
\draw [fill=green,green] (8, 6.5) circle [radius=0.4];
\draw [fill=blue,blue] (0, 2) circle [radius=0.4];
\draw [fill=blue,blue] (8, 2) circle [radius=0.4];
\draw [fill=blue,blue] (4, 8.5) circle [radius=0.4];

\node at (2.5,0) {\small $e_0$};
\node at (9,7.5) {\small $e_1$};
\node at (-1,7.5) {\small $e_2$};
\node at (2.5,4) {\small $e_3$};
\node at (-1,1) {\small $v_0$};
\node at (9,1) {\small $v_1$};
\node at (4,7) {\small $v_2$};
\node at (4,-2) {\small $T_1$};

\end{scope}

\begin{scope}[shift={(23,0)}]

\path [fill=pink] (4,4) -- (8,2) -- (8,6.5) -- (4,8.5);

\draw [ultra thick, red] (4,0) -- (8,2);
\draw [ultra thick, red] (8,6.5) -- (4,8.5);
\draw [ultra thick, red] (4,8.5) -- (0,6.5);
\draw [ultra thick, red] (0,6.5) -- (0,2);
\draw [ultra thick, red] (4,4) -- (4,8.5);
\draw [ultra thick, red] (4,4) -- (8,2);

\draw [fill=green,green] (4, 0) circle [radius=0.4];
\draw [fill=green,green] (4, 4) circle [radius=0.4];
\draw [fill=green,green] (0, 6.5) circle [radius=0.4];
\draw [fill=green,green] (8, 6.5) circle [radius=0.4];
\draw [fill=blue,blue] (0, 2) circle [radius=0.4];
\draw [fill=blue,blue] (8, 2) circle [radius=0.4];
\draw [fill=blue,blue] (4, 8.5) circle [radius=0.4];

\node at (4,-2) {\small $T_2$};

\end{scope}

\begin{scope}[shift={(35,0)}]

\path [fill=pink] (4,4) -- (8,2) -- (8,6.5) -- (4,8.5);

\draw [ultra thick, red] (4,0) -- (8,2);
\draw [ultra thick, red] (8,2) -- (8,6.5);
\draw [ultra thick, red] (0,6.5) -- (0,2);
\draw [ultra thick, red] (0,2) -- (4,4);
\draw [ultra thick, red] (4,4) -- (4,8.5);
\draw [ultra thick, red] (4,4) -- (8,2);

\draw [fill=green,green] (4, 0) circle [radius=0.4];
\draw [fill=green,green] (4, 4) circle [radius=0.4];
\draw [fill=green,green] (0, 6.5) circle [radius=0.4];
\draw [fill=green,green] (8, 6.5) circle [radius=0.4];
\draw [fill=blue,blue] (0, 2) circle [radius=0.4];
\draw [fill=blue,blue] (8, 2) circle [radius=0.4];
\draw [fill=blue,blue] (4, 8.5) circle [radius=0.4];

\node at (4,-2) {\small $T_3$};

\end{scope}

\begin{scope}[shift={(0,-13)}]

\draw [ultra thick, red] (4,0) -- (8,2);
\draw [ultra thick, red] (8,2) -- (8,6.5);
\draw [ultra thick, red] (8,6.5) -- (4,8.5);
\draw [ultra thick, red] (0,6.5) -- (0,2);
\draw [ultra thick, red] (0,2) -- (4,4);
\draw [ultra thick, red] (4,4) -- (4,8.5);

\draw [fill=green,green] (4, 0) circle [radius=0.4];
\draw [fill=green,green] (4, 4) circle [radius=0.4];
\draw [fill=green,green] (0, 6.5) circle [radius=0.4];
\draw [fill=green,green] (8, 6.5) circle [radius=0.4];
\draw [fill=blue,blue] (0, 2) circle [radius=0.4];
\draw [fill=blue,blue] (8, 2) circle [radius=0.4];
\draw [fill=blue,blue] (4, 8.5) circle [radius=0.4];

\end{scope}

\begin{scope}[shift={(11,-13)}]

\draw [ultra thick, red] (0,2) -- (4,0);
\draw [ultra thick, red] (4,0) -- (8,2);
\draw [ultra thick, red] (8,6.5) -- (4,8.5);
\draw [ultra thick, red] (4,8.5) -- (0,6.5);
\draw [ultra thick, red] (0,6.5) -- (0,2);
\draw [ultra thick, red] (4,4) -- (4,8.5);

\draw [fill=green,green] (4, 0) circle [radius=0.4];
\draw [fill=green,green] (4, 4) circle [radius=0.4];
\draw [fill=green,green] (0, 6.5) circle [radius=0.4];
\draw [fill=green,green] (8, 6.5) circle [radius=0.4];
\draw [fill=blue,blue] (0, 2) circle [radius=0.4];
\draw [fill=blue,blue] (8, 2) circle [radius=0.4];
\draw [fill=blue,blue] (4, 8.5) circle [radius=0.4];

\end{scope}

\begin{scope}[shift={(22,-13)}]

\draw [ultra thick, red] (0,2) -- (4,0);
\draw [ultra thick, red] (4,0) -- (8,2);
\draw [ultra thick, red] (8,6.5) -- (4,8.5);
\draw [ultra thick, red] (0,6.5) -- (0,2);
\draw [ultra thick, red] (0,2) -- (4,4);
\draw [ultra thick, red] (4,4) -- (4,8.5);

\draw [fill=green,green] (4, 0) circle [radius=0.4];
\draw [fill=green,green] (4, 4) circle [radius=0.4];
\draw [fill=green,green] (0, 6.5) circle [radius=0.4];
\draw [fill=green,green] (8, 6.5) circle [radius=0.4];
\draw [fill=blue,blue] (0, 2) circle [radius=0.4];
\draw [fill=blue,blue] (8, 2) circle [radius=0.4];
\draw [fill=blue,blue] (4, 8.5) circle [radius=0.4];

\end{scope}

\begin{scope}[shift={(33,-13)}]

\draw [ultra thick, red] (0,2) -- (4,0);
\draw [ultra thick, red] (4,0) -- (8,2);
\draw [ultra thick, red] (8,2) -- (8,6.5);
\draw [ultra thick, red] (8,6.5) -- (4,8.5);
\draw [ultra thick, red] (0,6.5) -- (0,2);
\draw [ultra thick, red] (0,2) -- (4,4);

\draw [fill=green,green] (4, 0) circle [radius=0.4];
\draw [fill=green,green] (4, 4) circle [radius=0.4];
\draw [fill=green,green] (0, 6.5) circle [radius=0.4];
\draw [fill=green,green] (8, 6.5) circle [radius=0.4];
\draw [fill=blue,blue] (0, 2) circle [radius=0.4];
\draw [fill=blue,blue] (8, 2) circle [radius=0.4];
\draw [fill=blue,blue] (4, 8.5) circle [radius=0.4];

\end{scope}

\end{tikzpicture}
\caption{A plane bipartite graph with a ribbon structure that rotates in opposite directions at its two color classes. The $V$-cut Jaeger trees corresponding to the indicated base point are listed in the associated shelling order.}
\label{fig:knot}
\end{figure}

\begin{ex}
\label{ex:non-triangulation}
Let us consider the plane bipartite graph of Figure \ref{fig:Jaeger-trees} again, with the same base node and base edge as in Example \ref{ex:kektura} and onward, but this time let us use a ribbon structure very similar to that of Figure \ref{fig:apapirrasokatir}. That is, we use counterclockwise rotations about violet nodes and clockwise ones about emerald nodes. (This is the same rule as in Example \ref{ex:apapirrasokatir}, except that the colors traded places.) In Figure \ref{fig:knot} we show again an embedding in $\R^3$ of the corresponding ribbon surface with the base point along its boundary. The seven spanning trees of the Figure are the resulting $V$-cut Jaeger trees, listed in the shelling order of Section \ref{sec:shelling}. One may check that the trees do realize all hypertrees on both color classes, cf.\ Corollary \ref{c:Jaeger_hypertree_uniquely_realized}. On the other hand, they do not form a triangulation of the root polytope. For example, the second and third trees are such that, with regard to the upper right quadrangular region of the embedding, one tree contains one pair of opposite edges and the other tree contains the other pair. That violates Postnikov's compatibility condition \cite[Lemma 12.6]{alex}, i.e., the simplices corresponding to the two trees do not intersect in a common face.

In order to have a closer look, let us use the notation of Figures \ref{fig:Jaeger-trees} and \ref{fig:fundamental_cut} again. Let us also denote the first three trees in the order by $T_1$, $T_2$, $T_3$ and let us refer to the corresponding five-dimensional simplices in $\R^E\oplus\R^V$ as $\sigma_1$, $\sigma_2$, and $\sigma_3$, respectively (i.e., $\sigma_i=Q_{T_i}$). Since the edge $e_3v_0$ appears for the first time in $T_3$, we see that $\sigma_3$ attaches to $\sigma_1\cup\sigma_2$ along its four-dimensional facet $\varphi$ that is opposite the vertex $\mathbf e_3+\mathbf v_0$. On the other hand, $\varphi$ is not a facet either of $\sigma_1$ or of $\sigma_2$. Instead, the following happens. The vectors $\mathbf e_1+\mathbf v_1$, $\mathbf e_1+\mathbf v_2$, $\mathbf e_3+\mathbf v_2$, and $\mathbf e_3+\mathbf v_1$ (corresponding to the edges bounding the shaded quadrangle) span a two-dimensional square. It is bisected by the diagonal from $\mathbf e_1+\mathbf v_2$ to $\mathbf e_3+\mathbf v_1$ so that the two halves belong to $\sigma_1$ and $\sigma_2$, respectively. The part of the square that belongs to $\sigma_3$, and indeed to the facet $\varphi$, is bounded by the \emph{other} diagonal. The three half-squares become facets of $\sigma_1$, $\sigma_2$, and $\sigma_3$ by taking their convex hulls with the common vertices $\mathbf e_0+\mathbf v_1$ and $\mathbf e_2+\mathbf v_0$, which are affine independent from the square.
\end{ex}

The bipartite graph of the previous example had points of degree at least three in both color classes, in other words it was not of the form $\bip G$ for any graph $G$. But that was not the reason for what we observed; rather, the ribbon structure was. In the next example we show that the dissection can fail to be a triangulation even in Bernardi's original family of cases.

\begin{figure}[h]
\begin{tikzpicture}[scale=.25]

\draw [thick] (2,0) -- (6,0);
\draw [thick] (6,0) -- (10,0);
\draw [thick] (10,0) -- (11,3);
\draw [thick] (11,3) -- (12,6);
\draw [thick] (12,6) -- (9,8);
\draw [thick] (9,8) -- (6,10);
\draw [thick] (6,10) -- (3,8);
\draw [thick] (3,8) -- (0,6);
\draw [thick] (0,6) -- (1,3);
\draw [thick] (1,3) -- (2,0);
\draw [thick] (2,0) -- (7,3);
\draw [thick] (7,3) -- (12,6);
\draw [thick] (12,6) -- (6,6);
\draw [thick] (6,6) -- (0,6);
\draw [thick] (0,6) -- (5,3);
\draw [thick] (5,3) -- (10,0);
\draw [thick] (10,0) -- (8,5);
\draw [thick] (8,5) -- (6,10);
\draw [thick] (6,10) -- (4,5);
\draw [thick] (4,5) -- (2,0);
\draw [fill] (2, 0) circle [radius=0.4];
\draw [fill] (10, 0) circle [radius=0.4];
\draw [fill] (12, 6) circle [radius=0.4];
\draw [fill] (6, 10) circle [radius=0.4];
\draw [fill] (0, 6) circle [radius=0.4];
\draw [thick,fill=white] (6, 0) circle [radius=0.4];
\draw [thick,fill=white] (11, 3) circle [radius=0.4];
\draw [thick,fill=white] (9, 8) circle [radius=0.4];
\draw [thick,fill=white] (3, 8) circle [radius=0.4];
\draw [thick,fill=white] (1, 3) circle [radius=0.4];
\draw [thick,fill=white] (7, 3) circle [radius=0.4];
\draw [thick,fill=white] (8, 5) circle [radius=0.4];
\draw [thick,fill=white] (6, 6) circle [radius=0.4];
\draw [thick,fill=white] (4, 5) circle [radius=0.4];
\draw [thick,fill=white] (5, 3) circle [radius=0.4];
\draw [->,thick] (1,0) arc [radius=1, start angle=180, end angle=290];
\draw [<-,thick] (11,0) arc [radius=1, start angle=0, end angle=-110];
\draw [->,thick] (12,7) arc [radius=1, start angle=90, end angle=-20];
\draw [<-,thick] (0,7) arc [radius=1, start angle=90, end angle=200];
\draw [->,thick] (6,11) arc [radius=1, start angle=90, end angle=140];
\draw [thick] (6,11) arc [radius=1, start angle=90, end angle=40];
\draw [fill] (6, 11) circle [radius=0.2];

\begin{scope}[shift={(20,0)}]
\draw [thick] (2,0) -- (6,0);
\draw [thick] (6,0) -- (10,0);
\draw [ultra thick] (10,0) -- (11,3);
\draw [thick] (11,3) -- (12,6);
\draw [thick] (9,8) -- (6,10);
\draw [thick] (3,8) -- (0,6);
\draw [thick] (0,6) -- (1,3);
\draw [ultra thick] (1,3) -- (2,0);
\draw [ultra thick] (7,3) -- (12,6);
\draw [thick] (12,6) -- (6,6);
\draw [ultra thick] (0,6) -- (5,3);
\draw [thick] (10,0) -- (8,5);
\draw [thick] (8,5) -- (6,10);
\draw [thick] (4,5) -- (2,0);
\draw [fill] (2, 0) circle [radius=0.4];
\draw [fill] (10, 0) circle [radius=0.4];
\draw [fill] (12, 6) circle [radius=0.4];
\draw [fill] (6, 10) circle [radius=0.4];
\draw [fill] (0, 6) circle [radius=0.4];
\draw [thick,fill=white] (6, 0) circle [radius=0.4];
\draw [thick,fill=white] (11, 3) circle [radius=0.4];
\draw [thick,fill=white] (9, 8) circle [radius=0.4];
\draw [thick,fill=white] (3, 8) circle [radius=0.4];
\draw [thick,fill=white] (1, 3) circle [radius=0.4];
\draw [thick,fill=white] (7, 3) circle [radius=0.4];
\draw [thick,fill=white] (8, 5) circle [radius=0.4];
\draw [thick,fill=white] (6, 6) circle [radius=0.4];
\draw [thick,fill=white] (4, 5) circle [radius=0.4];
\draw [thick,fill=white] (5, 3) circle [radius=0.4];
\end{scope}

\begin{scope}[shift={(35,0)}]
\draw [thick] (6,0) -- (10,0);
\draw [thick] (10,0) -- (11,3);
\draw [ultra thick] (11,3) -- (12,6);
\draw [thick] (9,8) -- (6,10);
\draw [thick] (3,8) -- (0,6);
\draw [ultra thick] (0,6) -- (1,3);
\draw [thick] (1,3) -- (2,0);
\draw [ultra thick] (2,0) -- (7,3);
\draw [thick] (7,3) -- (12,6);
\draw [thick] (12,6) -- (6,6);
\draw [ultra thick] (5,3) -- (10,0);
\draw [thick] (8,5) -- (6,10);
\draw [thick] (6,10) -- (4,5);
\draw [thick] (4,5) -- (2,0);
\draw [fill] (2, 0) circle [radius=0.4];
\draw [fill] (10, 0) circle [radius=0.4];
\draw [fill] (12, 6) circle [radius=0.4];
\draw [fill] (6, 10) circle [radius=0.4];
\draw [fill] (0, 6) circle [radius=0.4];
\draw [thick,fill=white] (6, 0) circle [radius=0.4];
\draw [thick,fill=white] (11, 3) circle [radius=0.4];
\draw [thick,fill=white] (9, 8) circle [radius=0.4];
\draw [thick,fill=white] (3, 8) circle [radius=0.4];
\draw [thick,fill=white] (1, 3) circle [radius=0.4];
\draw [thick,fill=white] (7, 3) circle [radius=0.4];
\draw [thick,fill=white] (8, 5) circle [radius=0.4];
\draw [thick,fill=white] (6, 6) circle [radius=0.4];
\draw [thick,fill=white] (4, 5) circle [radius=0.4];
\draw [thick,fill=white] (5, 3) circle [radius=0.4];
\end{scope}

\end{tikzpicture}
\caption{A ribbon structure for the complete graph $K_5$ and two Jaeger trees.}
\label{fig:otszog}
\end{figure}

\begin{ex}
\label{ex:otszog}
Let us consider the complete graph $K_5$ and its ribbon structure indicated in Figure \ref{fig:otszog}. That is, we refer to the planar drawing shown and let edges be ordered clockwise around two of the vertices and counterclockwise around the other three. This extends uniquely to the degree two nodes of $\bip K_5$, as discussed in Section \ref{sec:Bernardi-process}. For $K_5$, let the base vertex be the one on top, and let the top left edge be the base edge. Equivalently, for $\bip K_5$, let the base node be the same point and let the base edge be the upper half of the previous.

Now if cutting edges of $\bip K_5$ is allowed near the five vertices of $K_5$, then the two spanning trees in Figure \ref{fig:otszog} are Jaeger trees. The simplices in $Q_{\bip K_5}$ that correspond to these trees do not intersect in a common face: indeed, the thickened edges (four from each tree) form a cycle in $\bip K_5$ that violates the condition in \cite[Lemma 12.6]{alex}.
\end{ex}

Now we turn to showing that certain triangulations found in the literature are special cases of the dissection by Jaeger trees. Let us first consider the `triangulation by non-crossing trees' \cite{GGP}, see also \cite[Example 5.3]{KP_Ehrhart}. This  applies in the case of a complete bipartite graph $K$, say on $(m+1)+(n+1)$ vertices, whose root polytope is the product $Q_K=\Delta_m\times\Delta_n$ of an $m$- and an $n$-dimensional unit simplex. The idea is to draw the vertices of $K$ on two parallel lines in the plane, separated by color, and then consider those spanning trees whose edges do not cross each other in the drawing. (See Figure \ref{f:examples_for_examples} for an example.) The maximal simplices corresponding to these trees form a triangulation of $Q_K$.

\begin{figure}[h]
\begin{center}
	\begin{tikzpicture}[-,>=stealth',auto,scale=1,
    thick]
    \tikzstyle{e}=[circle,fill,draw, color=green]
    \tikzstyle{v}=[circle,fill,draw,color=blue]
    \node[e,label=left:{\small $b_0$}] (1) at (-0.25, 0) {};
    \node[e] (2) at (1, 0) {};
    \node[e] (3) at (2.25, 0) {};
    \node[v] (4) at (-0.5, 2) {};
    \node[v] (5) at (0.5, 2) {};
    \node[v] (6) at (1.5, 2) {};
    \node[v,label=right:{\small $b_1$}] (7) at (2.5, 2) {};
    \path[every node/.style={font=\sffamily\small},dashed]
    (1) edge node [below] {} (4)
    (1) edge node [below] {} (5)
    (1) edge node [below] {} (6)
    (1) edge node [below] {} (7)
    (2) edge node [below] {} (4)
    (2) edge node [below] {} (5)
    (2) edge node [below] {} (6)
    (2) edge node [below] {} (7)
    (3) edge node [above] {} (4)
    (3) edge node [above] {} (5)
    (3) edge node [above] {} (6)
    (3) edge node [above] {} (7);
    \path[every node/.style={font=\sffamily\small},line width=0.7mm]
    (4) edge node {} (1)
    (4) edge node {} (2)
    (5) edge node {} (2)
    (6) edge node {} (2)
    (7) edge node {} (2)
    (7) edge node {} (3);
    \end{tikzpicture}
    \hspace{0.3cm}
    \begin{tikzpicture}[-,>=stealth',auto,scale=0.4,
	thick]
	\tikzstyle{c}=[{circle,draw,fill,font=\sffamily\small}]
	\node[c,color=green] (0) at (0, 0.8) {};
	\node at (1,.4) {\small $b_1$};
	\node[c,color=blue] (1) at (3, 2.3) {};
	\node[c,color=blue,label=left:{\small $b_0$}] (2) at (-3, 2.3) {};
	\node[c,color=green] (3) at (3, 4.7) {};
	\node[c,color=green] (4) at (-3, 4.7) {};
	\node[c,color=blue] (5) at (0, 6.2) {};
	\node[c,color=green] (6) at (0, 3.6) {};
    \node[c,color=red,label=right:{\small $r_0$}] (7) at (5, 4.2) {};
    \node[c,color=red] (8) at (1.5, 4.2) {};
    \node[c,color=red] (9) at (-1.5, 4.2) {};
    \node[c,color=red] (10) at (0, 2.2) {};
    \draw[->,rounded corners=10pt] (7) |- (1, 7.2) -- (0, 7.2)  -|  (9);
    \draw[->] (7) to[bend left=30] (8);
    \draw[->] (8) to[bend left=30] (10);
	\path[every node/.style={font=\sffamily\small},dashed]
	(4) edge node {} (5)
    (1) edge node {} (3)
    (1) edge node {} (6);
	\path[every node/.style={font=\sffamily\small}, line width=0.7mm]
	(0) edge node {} (2)
    (1) edge node {} (0)
    (2) edge node {} (4)
	(5) edge node {} (3)
    (6) edge node {} (2)
	(5) edge node {} (6);
	\end{tikzpicture}
\end{center}
\caption{Left: A non-crossing tree in $K_{3,4}$. Right: A $V$-cut Jaeger tree in a plane bipartite graph and the corresponding arborescence of the dual graph.}
\label{f:examples_for_examples}
\end{figure}

Now let us imagine the two lines as horizontal, with emerald vertices on the lower one and violet vertices on the upper. Let us define a ribbon structure by rotating counterclockwise around each vertex. Our base node is the lower left (emerald) one and the base edge is the one connecting the base node diagonally to the upper right violet node. (I.e., in the sense of Remark \ref{rem:surface}, we place our base point slightly below the lower left vertex.) Then it is easy to see that non-crossing trees are $E$-cut Jaeger trees. The converse is not hard to check either but since we are not aware of an elegant proof, let us just point out that it follows from the fact that all dissections of a root polytope consist of the same number of maximal simplices. The shelling order provided by Theorem \ref{thm:shelling} is the same in this case as the lexicographic order of the corresponding hypertrees on $E$.

Our other class of triangulations applies whenever the connected bipartite graph $G$ comes with an embedding in the plane. (That is, unlike in the previous situation, edges of $G$ are now not allowed to cross.) In that case, the dual graph $G^*$ has a natural orientation by the rule that every edge of $G^*$ is oriented so that the violet endpoint of the corresponding edge of $G$ is on its left. (See Figure \ref{f:examples_for_examples}.) Let us fix a vertex $r_0$ of $G^*$, that is to say, a face $r_0$ of $G$. Then the spanning trees of $G$ dual to spanning arborescences of $G^*$ rooted at $r_0$ form a triangulation of $Q_G$ \cite{KM}. (Here an arborescence is a tree so that all of its edges point away from the root. A spanning tree of $G$ is dual to a spanning arborescence of $G^*$ if it contains precisely the edges corresponding to the non-edges of the arborescence.)

We claim that the trees given above are exactly the $V$-cut Jaeger trees of $G$, where the ribbon structure is induced by the positive orientation of the plane (so that, unlike in Examples \ref{ex:apapirrasokatir} and \ref{ex:non-triangulation}, both colors spin in the \emph{same} direction). To be more specific, let the base node $b_0$ be violet and incident to the face $r_0$, and let the base edge $b_0b_1$ be also incident to $r_0$ so that when we orient it from $b_0$ to $b_1$, then $r_0$ is on the right. 

To justify our claim it is enough to show that the dual $A^*$ of each arborescence $A\subset G^*$ is a $V$-cut Jaeger tree of $G$. Tracing the boundary of a neighborhood of $A$ or of $A^*$ are equivalent. (The common base point should be thought of as an interior point of the arc connecting $b_0$ and $r_0$.) When we do the latter, because $A$ is an arborescence, we first walk along each edge of $A$ in the direction of its orientation and on its left side. But by definition, this means that the corresponding dual edge of $G$ is cut at its violet endpoint.

Conversely, each $V$-cut Jaeger tree is the dual of an arborescence rooted at $r_0$. Indeed, consider the violet tour of a $V$-cut Jaeger tree $T$. Notice that until cutting the first edge, we move along the boundary of the face $r_0$, with $r_0$ on the right. When we first cut an edge $ve$, we start to see a new face $r$ on the right, and $r_0r$ is an oriented edge of $G^*$ which `starts' the dual arborescence. Continuing this argument inductively, we see that cutting an edge at its violet endpoint corresponds to inserting a new edge pointing away from $r_0$ in the dual graph.

\appendix

\section{The Bernardi bijection is activity-preserving for graphs} \label{sec:act_preserv_for_graphs}

Let us now take a look at the case of ordinary graphs $G=(V,E)$, loopless but possibly with multiple edges, i.e., hypergraphs so that $d(e)=2$ for each $e\in E$. We claim that in this case, the Bernardi process (that is, Corollary  \ref{c:Jaeger_hypertree_uniquely_realized}) gives an internal-activity-preserving bijection between the hypertrees on $E$ and on $V$. The precise statement is as follows. Recall that ribbon structures on $G$ and on $\bip G$ are equivalent. We use the latter sense; in particular, the base node $b_0$ may be an element of $E$ or an element of $V$.

\begin{thm}\label{thm:Bernardi_activity_preserving_graph}
For a ribbon graph $G=(V,E)$, and a $V$-cut Jaeger tree $T$ of $\bip G$, the internal embedding inactivity of $\mathbf f_E(T)$ with respect to the violet $T$-order on $E$ equals the internal embedding inactivity of $\mathbf f_V(T)$ with respect to the violet $T$-order on $V$. 
\end{thm}

Note that here $\mathbf f_E(T)$ is just the characteristic function of the unique spanning tree $T'$ of $G$ that can be built from those half-edges of $G$ that occur as edges in $T$. The hypertree $\mathbf f_V(T)$ on the other hand is derived less directly from $T'$: this time we also need the ribbon structure to decide which half-edges of the non-edges of $T'$ should be added to $T'$ in order to create $T$. 

We also have the following refinement. See Figure \ref{fig:parositas} for an illustration.

\begin{thm}\label{thm:Bernardi_activity_preserving_graph_refined}
For a ribbon graph $G=(E,V)$ and a $V$-cut Jaeger tree $T$ of $\bip G$, the internally semi-passive edges with respect to the violet $T$-order of the edges of $\bip G$ match exactly the internally inactive elements of $V$ for $\mathbf f_V(T)$ with respect to the violet $T$-order on $V$, and the internally inactive elements of $E$ for $\mathbf f_E(T)$ with respect to the violet $T$-order on $E$. 
\end{thm}

To a degree, this statement justifies our preference for inactive objects over active ones: if we pass to the complementer sets of nodes and of edges, respectively, the perfect matching does not hold any more.

\begin{figure}
\begin{center}
	\begin{tikzpicture}[-,>=stealth',auto,scale=0.5,
    thick]
    \tikzstyle{o}=[circle,fill,draw]
    \tikzstyle{s}=[circle,draw,minimum size=0.6cm]
    \node[o,color=green] (1) at (6, -1) {};
    \node[o,color=green] (2) at (2, -1) {};
    \node[o,color=green] (3) at (2, 1) {};
    \node[s] (11) at (2, 1) {};
    \node[o,color=green] (4) at (6, 1) {};
    \node[o,color=green] (5) at (4, 0) {};
    \node[s] (12) at (4, 0) {};
    \node[o,color=blue] (6) at (8, 0) {};
    \node[o,color=blue] (7) at (4, -2) {};
    \node[s] (13) at (4, -2) {};
    \node[o,color=blue] (8) at (0, 0) {};
    \node[o,color=blue] (9) at (4, 2) {};
    \node[s] (10) at (4, 2) {};
    \path[every node/.style={font=\sffamily\small}, line width=0.8mm]
    (1) edge node [below right] {$6$} (6)
    (7) edge node [below left] {$3$} (2)
    (3) edge node [above left] {$1$} (8)
    (3) edge [color=red] node {$9$} (9)
    (4) edge node [above right] {$8$} (9)
    (5) edge node {$2$} (9)
    (1) edge node {$4$} (7)
    (5) edge [color=red] node {$7$} (7);
    \path[every node/.style={font=\sffamily\small},dashed]
    (4) edge node {$5$} (6)
    (2) edge node {$0$} (8);
    \end{tikzpicture}
\end{center}
\caption{The ribbon structure for the graph shown is the positive orientation of the plane at each node. The base point is to the left of the leftmost violet node. The thick edges form the $V$-cut Jaeger tree $T$. The numbers on the edges show the violet $T$-order. The circled nodes are the internally inactive elements with respect to the violet $T$-order on $E$ and on $V$. The red edges are the internally semi-passive edges in $T$ with respect to the violet $T$-order.}
\label{fig:parositas}
\end{figure}

\begin{proof}
Let $T$ be a $V$-cut Jaeger tree and recall (Lemma \ref{l:emerald_tree_is_also_violet}) that $T$ is an $E$-cut Jaeger tree with respect to the reversed ribbon structure. Throughout the proof we use Lemma \ref{l:activities_description} and its corollaries with interchanging the roles of the colors. 

By Corollary \ref{cor:passziv_csucshoz_passziv_el}, a violet node $v$ is internally  inactive for $\mathbf f_V(T)$ with respect to the violet $T$-order on $V$ if and only if there is an edge of $\bip G$ incident to it which is internally semi-passive with respect to the violet $T$-order, moreover, in this case there is exactly one such edge and it has its emerald endpoint (not $v$) in its base component with respect to $T$. Let us denote the collection of these edges by $S$.

What we need to show is that the emerald endpoints of the edges in $S$ are all different, all internally inactive for $\mathbf f_E(T)$ with respect to the violet $T$-order, and that every other emerald node is internally active.

The first point is easy to check: If $e\neq b_0$, then $T$ has at most one edge incident to $e$ that has its emerald endpoint in the base component (defined by the edge), since for any $e\in E$ there are at most two incident edges in $T$, and if $e$ is not the base node, then one of them starts the path from $e$ to $b_0$, making it have its violet endpoint in its base component. If $e=b_0$, then either the base edge $b_0b_1$ is not in $T$, or by Lemma \ref{l:char_Jaeger_cuts}, every edge of the fundamental cut $C^*(T,b_0b_1)$ has its emerald endpoint in the base component of $T-b_0b_1$; hence the base edge cannot be internally semi-passive by part \ref{otodik} of Lemma \ref{l:activities_description}. Thus, in either case there cannot be two internally semi-passive edges incident to $b_0$.

Now it suffices to show that an emerald node $e\in E$ is internally inactive for $\mathbf f_E(T)$ with respect to the violet $T$-order on $E$ if and only if there is an edge of $T$ incident to $e$ that is internally semi-passive with respect to the violet $T$-order.   

If $e$ is internally inactive then $\mathbf f_E(T)(e)>0$, i.e., both edges of $\bip G$ incident to $e$ are in $T$. Moreover, there is an emerald node $e'\in E$ that comes before $e$ in the violet $T$-order, and $\mathbf f_E(T)$ is such that a transfer of valence is possible from $e$ to $e'$. Hence necessarily $\mathbf f_E(T)(e')=0$, that is, $T$ contains only one edge incident to $e'$. Let $v'e'$ be the other edge, the one not in $T$. As the hypertrees on $E$ in $\bip G$ determine their realizing spanning trees at nodes $x\in E$ with $\mathbf f(x)=1$, the transfer of valence from $e$ to $e'$ can be achieved by removing an edge incident to $e$ from $T$, and adding $v'e'$. (Note that the new tree does not have to be Jaeger.) Therefore one of the edges incident to $e$ is in the fundamental cycle $C(T,v'e')$ -- and thus both are. In other words, $v'e'$ is in the fundamental cut of $T$ with respect to both edges incident to $e$ in $\bip G$.

If $e=b_0$, then we claim that $b_0b_1^-$ is internally semi-passive with respect to the violet $T$-order. Indeed, $b_0b_1^-$ is the first edge in the emerald tour of $T$, and since $v'e'$ is also in its fundamental cut, this ensures that $b_0b_1^-$ is not the largest element of $C^*(T,b_0b_1^-)$ with respect to the emerald order of the edges of $\bip\HH$. Hence our claim follows by the \ref{negyedik} $\Rightarrow$ \ref{masodik} implication of Lemma \ref{l:activities_description}.

When $e\ne b_0$, note that as $T$ is a $V$-cut Jaeger tree, $v'e'$ is smaller in the violet $T$-order than the other edge of $\bip G$ incident to $e'$. Furthermore, $v'e'$ is cut at $v'$ and this happens before either edge of $\bip G$ incident to $e$ becomes current with its violet endpoint in the violet tour of $T$. As $e\neq b_0$, the first of the four traversals of the edges incident to $e$ is from the direction of the violet endpoint. Hence $v'e'$ is cut before any edge incident to $e$ gets traversed, in particular $v'e'$ has its violet endpoint $v'$ in the base component of the fundamental cut of $T$ with respect to either edge incident to $e$. One of these edges has its emerald endpoint in its base component. Now by the \ref{otodik} $\Rightarrow$ \ref{masodik} implication in Lemma \ref{l:activities_description} (and Lemma \ref{l:emerald_tree_is_also_violet}), this edge is internally semi-passive with respect to the violet $T$-order. 

Conversely, if an edge $\varepsilon=ve$ is internally semi-passive in $T$ with respect to the violet $T$-order, then $e$ is internally inactive for $\mathbf f_E(T)$ with respect to the violet $T$-order on $E$ for the following reason.

We may use the \ref{masodik} $\Rightarrow$ \ref{otodik} implication of Lemma \ref{l:activities_description} to find an edge $\varepsilon'=e'v' \in C^*(T,\varepsilon)$ that has its violet endpoint in the base component of $T-\varepsilon$ (which is the component containing $e$). Hence by Lemma \ref{l:char_Jaeger_cuts}, $\varepsilon'$ is cut in the violet tour of $T$ before the traversal of $\varepsilon$ from the direction of its emerald endpoint. Since there are only two edges in $\bip G$ incident to $e$, the node $e$ `receives its number' in the violet $T$-order immediately before the traversal of $\varepsilon$ from the direction of its emerald endpoint. Thus, $e'$ precedes $e$ in the violet $T$-order on $E$. As $\varepsilon'\in C^*(T,\varepsilon)$, we may realize a transfer of valence from $e$ to $e'$ by replacing $\varepsilon$ in $T$ with $\varepsilon'$. This completes the proof of the theorem.
\end{proof}

We note that Theorem \ref{thm:Bernardi_activity_preserving_graph} is not true for all hypergraphs, as in general the internally semi-passive edges of a Jaeger tree do not necessarily form a matching.

\bibliographystyle{plain}

\bibliography{Bernardi}

\begin{thebibliography}{10}

\bibitem{Baker-Wang}
Matthew Baker and Yao Wang.
\newblock The {B}ernardi process and torsor structures on spanning trees.
\newblock {\em Int. Math. Res. Not. IMRN}, (16):5120--5147, 2018.

\bibitem{Bernardi_first}
Olivier Bernardi.
\newblock A characterization of the {T}utte polynomial via combinatorial
  embeddings.
\newblock {\em Ann. Comb.}, 12(2):139--153, 2008.

\bibitem{Bernardi_Tutte}
Olivier Bernardi.
\newblock Tutte polynomial, subgraphs, orientations and sandpile model: new
  connections via embeddings.
\newblock {\em Electron. J. Combin.}, 15(1):Research Paper 109, 53, 2008.

\bibitem{universal}
Olivier Bernardi, Tam\'as K\'alm\'an, and Alexander Postnikov.
\newblock Universal {T}utte polynomial.
\newblock https://arxiv.org/pdf/2004.00683.pdf, 2020.

\bibitem{Cameron-Fink}
Amanda Cameron and Alex Fink.
\newblock A lattice point counting generalisation of the {T}utte polynomial.
\newblock https://arxiv.org/pdf/1604.00962.pdf, 2016.

\bibitem{Frank-K}
Andr\'as Frank and Tam\'as K\'alm\'an.
\newblock Root polytopes and strong orientations of bipartite graphs.
\newblock in preparation.

\bibitem{GGP}
Israel~M. Gelfand, Mark~I. Graev, and Alexander Postnikov.
\newblock Combinatorics of hypergeometric functions associated with positive
  roots.
\newblock In {\em The {A}rnold-{G}elfand mathematical seminars}, pages
  205--221. Birkh\"{a}user Boston, Boston, MA, 1997.

\bibitem{Jaeger}
Fran\c{c}ois Jaeger.
\newblock A combinatorial model for the {H}omfly polynomial.
\newblock {\em European J. Combin.}, 11(6):549--558, 1990.

\bibitem{hiperTutte}
Tam\'{a}s K\'{a}lm\'{a}n.
\newblock A version of {T}utte's polynomial for hypergraphs.
\newblock {\em Adv. Math.}, 244:823--873, 2013.

\bibitem{KM}
Tam\'{a}s K\'{a}lm\'{a}n and Hitoshi Mura\-kami.
\newblock Root polytopes, parking functions, and the {HOMFLY} polynomial.
\newblock {\em Quantum Topol.}, 8(2):205--248, 2017.

\bibitem{KP_Ehrhart}
Tam\'{a}s K\'{a}lm\'{a}n and Alexander Postnikov.
\newblock Root polytopes, {T}utte polynomials, and a duality theorem for
  bipartite graphs.
\newblock {\em Proc. Lond. Math. Soc. (3)}, 114(3):561--588, 2017.

\bibitem{Kato}
Keiju Kato.
\newblock Interior polynomial for signed bipartite graphs and the {HOMFLY}
  polynomial.
\newblock arXiv:1705.05063.

\bibitem{hibi}
Hidefumi Ohsugi and Takayuki Hibi.
\newblock Normal polytopes arising from finite graphs.
\newblock {\em J. Algebra}, 207(2):409--426, 1998.

\bibitem{alex}
Alexander Postnikov.
\newblock Permutohedra, associahedra, and beyond.
\newblock {\em Int. Math. Res. Not. IMRN}, (6):1026--1106, 2009.

\bibitem{Schrijver_CombOpt}
Alexander Schrijver.
\newblock {\em Combinatorial optimization. {P}olyhedra and efficiency. {V}ol.
  {B}}, volume~24 of {\em Algorithms and Combinatorics}.
\newblock Springer-Verlag, Berlin, 2003.
\newblock Matroids, trees, stable sets, Chapters 39--69.

\bibitem{Tutte}
William~T. Tutte.
\newblock A contribution to the theory of chromatic polynomials.
\newblock {\em Canad. J. Math.}, 6:80--91, 1954.

\end{thebibliography}

\end{document}